\documentclass[a4paper,twoside,10pt]{article}

\usepackage[a4paper,left=3cm,right=3cm, top=3cm, bottom=3cm]{geometry}
\usepackage{mathrsfs}
\usepackage{amsfonts}
\usepackage{doi}
\usepackage{cite}
\usepackage{amsmath}
\usepackage{cases}
\usepackage[titletoc,title]{appendix}
\usepackage{stmaryrd}
\usepackage{amssymb}
\usepackage{soul,xcolor}
\usepackage{amsthm}
\usepackage{tikz}
\usepackage{float}
\usepackage{hyperref}
\usepackage{algorithm,algpseudocode}
\usepackage{verbatim}
\usepackage{epstopdf} 
\newcommand{\todo}[1]{{\color{magenta}{[#1]}}} 

\newcommand{\AP}[1]{{\color{orange}{#1}}} 

\newcommand{\h}{h} 
\newcommand{\p}{p} 
\newcommand{\g}{g} 
\newcommand{\Vg}{V_\g} 
\newcommand{\Vz}{V_0} 
\renewcommand{\u}{u}  
\newcommand{\uI}{\u_I} 
\newcommand{\vv}{v} 
\newcommand{\un}{\u_n} 
\newcommand{\vn}{\vv_n} 
\newcommand{\vnbar}{\overline{v}_n} 
\renewcommand{\a}{a} 
\newcommand{\an}{\a_n} 
\newcommand{\anE}{\an^\E} 
\newcommand{\E}{K} 
\newcommand{\T}{T} 
\newcommand{\F}{F} 
\newcommand{\Q}{Q} 
\newcommand{\aE}{\a^\E} 
\newcommand{\taun}{\mathcal T_n} 
\newcommand{\tautilden}{\widetilde{\mathcal T}_n} 
\newcommand{\e}{e} 
\newcommand{\hE}{\h_\E} 
\newcommand{\he}{\h_\e} 
\newcommand{\diam}{\text{diam}} 
\newcommand{\x}{\mathbf x} 
\newcommand{\xE}{\mathbf x_\E} 
\newcommand{\n}{\mathbf n} 
\renewcommand{\k}{k} 
\newcommand{\nc}{nc} 
\newcommand{\npDelta}{n_p^\Delta} 
\newcommand{\q}{q}  
\newcommand{\qp}{\q_\p}  
\newcommand{\qpmu}{\q_{\p-1}}  
\newcommand{\qpmue}{\q^\e_{\p-1}} 
\newcommand{\qpmd}{\q_{\p-2}}  
\newcommand{\qDeltap}{\q^\Delta_\p} 
\newcommand{\qtildeDeltap}{\widetilde\q^\Delta_\p} 
\newcommand{\Poincare}{Poincar\'{e}} 
\newcommand{\Vnkg}{V^{\Delta,\k}_{n,\g}} 
\newcommand{\Vnkgc}{V^{\k}_{n,\g}} 
\newcommand{\Vnpg}{V^{\Delta,\p}_{n,\g}} 
\newcommand{\Vnpgc}{V^{\p}_{n,\g}} 
\newcommand{\Vnpzc}{V^{\p}_{n,0}} 
\newcommand{\Vnpz}{V^{\Delta,\p}_{n,0}} 
\newcommand{\VDeltaE}{V^\Delta(\E)} 
\newcommand{\VE}{V(\E)} 
\newcommand{\Pie}{\Pi^{0,\e}_{\p-1}} 
\newcommand{\Piei}{\Pi^{0,\e_i}_{\p-1}} 
\newcommand{\PiEextended}{\Pi^{\nabla,\Delta,\E}_\p} 
\newcommand{\PiE}{\Pi^{\nabla,\E}_\p} 
\newcommand{\Pinabla}{\Pi^{\nabla}_\p} 
\newcommand{\PiF}{\Pi^{0,\F}_{\p-1}} 
\newcommand{\PiT}{\Pi^{0,\T}_{\p-1}} 
\newcommand{\Pipartial}{\Pi^{0,\partial \E}_{\p-1}} 
\newcommand{\dof}{\text{dof}} 
\newcommand{\NE}{N_\E} 
\newcommand{\card}{\text{card}} 
\newcommand{\SE}{S^\E} 
\newcommand{\deltan}{\delta_n} 
\newcommand{\Omegaext}{\Omega_{\textup{ext}}} 
\renewcommand{\L}{L} 
\newcommand{\dist}{\text{dist}} 
\newcommand{\pbold}{\mathbf p} 
\newcommand{\pboldE}{\pbold^{\mathcal E}} 
\newcommand{\psiI}{\psi_I} 
\newcommand{\boldalpha}{\boldsymbol \alpha} 
\newcommand{\boldPi}{\boldsymbol \Pi} 
\newcommand{\boldPiStar}{\boldsymbol \Pi_\ast} 

\newcommand{\ds}{\textup{d}s} 
\newcommand{\dx}{\textup{d}x} 

\newcommand{\vertiii}[1]{{\left\vert\kern-0.25ex\left\vert\kern-0.25ex\left\vert #1 \right\vert\kern-0.25ex\right\vert\kern-0.25ex\right\vert}}

\newtheorem{thm}{Theorem}[section]

\newtheorem{lem}[thm]{Lemma}
\newtheorem{prop}[thm]{Proposition}
\theoremstyle{definition}

\theoremstyle{remark}
\newtheorem{remark}{Remark}

\title{Non-conforming harmonic virtual element method: $\h$- and $\p$-versions}
\date{}
\author{
Lorenzo Mascotto\thanks{Faculty of Mathematics, University of Vienna,
   1090 Vienna, Austria (lorenzo.mascotto@univie.ac.at,
   ilaria.perugia@univie.ac.at, alex.pichler@univie.ac.at).},\ 
Ilaria Perugia\footnotemark[1],\ 
Alexander Pichler\footnotemark[1]}

\begin{document}
\maketitle

\begin{abstract}
We study 
the $\h$- and $\p$-versions 
of non-conforming harmonic virtual element methods (VEM) for the
approximation of 
the Dirichlet-Laplace problem on a 2D polygonal 
domain,
providing quasi-optimal error bounds. 
Harmonic VEM do not make use of internal degrees of freedom. 
This leads to a faster convergence, in terms of the number
of degrees of freedom, as compared to standard VEM.
Importantly, the technical tools used in our 
$\p$-analysis 
can be employed as well in the analysis
of more general
non-conforming finite element methods and VEM.
The theoretical results are validated in a series of numerical
  experiments. The $hp$-version of the method is numerically tested,
demonstrating exponential convergence with rate given by the square
root of the number of degrees of freedom.


\end{abstract}

\section{Introduction} \label{section introduction}
In recent years, Galerkin methods based on polygonal/polyhedral meshes
have attracted
a lot of attention, owing to 
their flexibility in dealing with complex geometries \cite{Weisser_basic, BLM_MFD, abcd, cockburn_HDG, 	hpDEFEM_polygon, dipietroErn_hho, bridging_HHOandHDG}. 
In this paper, we focus on the virtual element method (VEM) introduced
in \cite{VEMvolley, hitchhikersguideVEM}. 
The main feature of VEM, in addition to the fact that they allow for
general polytopal meshes, is that they are based on
trial and test spaces that consist of solutions to local problems
mimicking the target one. These functions
are not known in a closed form, which is at 
the origin of the name ``virtual''.
Importantly, the construction of the method does not rely on an
explicit representation of the basis functions, but rather on the
explicit knowledge of degrees of freedom. This
allows to compute certain projection operators  from local VE spaces
into polynomial ones, which are instrumental in the definition of
proper bilinear forms.

Owing to its 
flexibility and simplicitity  
of the implementation,
despite its 
novelty, the basic VEM paradigm has 
already been extended to highly-regular \cite{BeiraoManzini_VEMarbitraryregularity} and non-conforming \cite{nonconformingVEMbasic, cangianimanzinisutton_VEMconformingandnonconforming}
approximating spaces, combined with 
domain decomposition techniques \cite{VEM_DD_basic}, adaptive mesh
refinement \cite{cangianigeorgulispryersutton_VEMaposteriori},
adapted to curved domains \cite{beiraorussovacca_curvedVEM}, and 
applied to a wide variety of problems; among them, we recall 
general second-order elliptic problems
\cite{bbmr_VEM_generalsecondorderelliptic}, 
eigenvalue problems \cite{VEMchileans, gardiniVacca},
Stokes problem
\cite{BLV_StokesVEMdivergencefree},
elasticity problem \cite{VEMelasticity},
Helmholtz problem \cite{Helmholtz-VEM},
Cahn-Hilliard equation \cite{absv_VEM_cahnhilliard}, 
discrete fracture network simulations \cite{Berrone-VEM}, and 
topology optimization~\cite{Topology-VEM}.


In this paper, we introduce and analyze 
\textit{non-conforming} harmonic VEM 
for the approximation of 
the Dirichlet-Laplace problem on 
polygonal domains. 
These methods can be seen as the intermediate conformity level between
the continuous harmonic VEM developed in
\cite{HarmonicVEM},
and the harmonic discontinuous Galerkin finite element method (DG-FEM) of 
\cite{hmps_harmonicpolynomialsapproximationandTrefftzhpdgFEM,li2006local,li2012negative}.
As typically done in non-conforming methods, instead of requiring
$C^0$-continuity of 
test and trial
functions 
over the entire physical domain,
one only imposes that
the moments, up to a certain order, of their jumps across two adjacent elements are zero.
We highlight that non-conforming VEM were introduced in \cite{nonconformingVEMbasic} for the approximation of the Poisson problem and were subsequently extended to the
approximation of general elliptic and Stokes problems in
\cite{cangianimanzinisutton_VEMconformingandnonconforming,
  CGM_nonconformingStokes}, respectively. Our method inherits the 
  structure of that of \cite{nonconformingVEMbasic}, but makes use of
  harmonic basis functions, which yield to faster convergence, when
  approximating harmonic solutions, as
compared to 
standard basis functions.

We are particularly interested in the investigation of the $\h$- and
$\p$-versions of these methods.
In the former version, convergence is achieved by fixing the dimension
of local spaces and refining the mesh, whereas, in the latter,
by fixing a single mesh and increasing the dimension of local
spaces. A combination of the two 
goes by 
the name of $\h\p$-version.
The literature regarding the $\p$- and $\h\p$-versions of
VEM is restricted to \cite{hpVEMbasic, hpVEMcorner, fetishVEM2D, fetishVEM3D, pVEMmultigrid}, in addition to the above-mentioned
work \cite{HarmonicVEM};
for the $\h\p$-version of DG-FEM and hybrid-high order methods on 
polytopal grids, see \cite{hpDEFEM_polygon} and \cite{hpHHO}
and the references therein.
We derive quasi-optimal error bounds in the broken $H^1$ norm and
  in the
  $L^2$ norm, which are explicit in terms of the mesh size and of the degree of accuracy of
the method.
Although not covered by our theoretical analysis,
we provide numerical evidence that, similarly as for the harmonic VEM and
harmonic DG-FEM \cite{HarmonicVEM,hmps_harmonicpolynomialsapproximationandTrefftzhpdgFEM},
the exponential convergence
of the $\h\p$-version of the non-conforming
harmonic VEM is faster than the one 
of standard FEM \cite{BabuGuo_hpFEM, SchwabpandhpFEM} and VEM
\cite{hpVEMbasic, hpVEMcorner}.

The tools that we
employ in the forthcoming $\p$-analysis of 
non-conforming
harmonic VEM can actually be employed as well in the
$\p$-analysis of 
non-conforming FEM and of 
non-conforming VEM.
For instance, 
our argument to trace back best approximation estimates by means of
non-conforming harmonic VE functions to 
best approximation estimates by means of discontinuous harmonic
polynomials (Proposition  \ref{theorem best approximation VE
  functions}) 
extends to the non-harmonic case (Proposition
\ref{theorem best error Manzini space}). This provides a useful tool
in order to develop a
$p$-analysis of the non-conforming VEM of
\cite{nonconformingVEMbasic}. 

We stress that, in the high-order case, the construction of
an explicit basis for
non-conforming harmonic VEM, as well as for
non-conforming standard VEM,
is much simpler than for 
non-conforming FEM, see for instance \cite{CiarletSauter2016}.

The design and analysis of the non-conforming harmonic VEM developed in this paper pave the way for the study of VEM for the Helmholtz problem in a truly Trefftz setting, alternative to the conforming plane wave VEM of~\cite{Helmholtz-VEM},
which was based on a partition of unity approach.
In fact, the non-conforming framework seems to be the most appropriate one in order to design virtual Helmholtz-Trefftz approximation spaces.
Such an extension has been investigated in the recent work~\cite{ncTVEM_Helmholtz}.

\medskip
The outline of this paper is as follows. In Section \ref{section continuous problem}, the model problem is formulated and the concept of regular polygonal decompositions needed for the analysis is introduced;
besides, we recall the definition of non-conforming Sobolev spaces subordinated to polygonal decompositions of the physical domain.
Section \ref{section VEM} is dedicated to the construction of the 2D
non-conforming harmonic VEM 
and to the analysis of its
$\h$- and $\p$-versions; further, a hint for the extension to the 3D case is given.
Next, in Section \ref{section numerical results}, numerical results
validating the theoretical convergence estimates 
are presented; 
a numerical investigation of
the full $\h\p$-version of the method is also provided.
Finally, details on the implementation of the method are given
in Appendix \ref{section appendix implementation details}.

\paragraph*{Notation}
We fix here once and for all the notation employed throughout the paper.
Given any domain $D \subseteq \mathbb{R}^d$, $d \in \mathbb N$, and $\ell \in \mathbb N$, we denote by $\mathbb P_\ell(D)$ and $\mathbb H_\ell(D)$ the spaces of polynomials and harmonic polynomials up to order $\ell$ over $D$, respectively;
moreover, we set $\mathbb P_{-1}(D) = \mathbb H_{-1}(D) = \emptyset$.

We use the standard notation for Sobolev spaces, norms, seminorms and inner products.
More precisely, we denote the Sobolev space of functions with square integrable weak derivatives up to order $s$ over $D$ by $H^s(D)$,
and the corresponding seminorms and norms by $|\cdot|_{s,D}$ and $\lVert \cdot \rVert_{s,D}$, respectively.
Sobolev spaces of non-integer order can be defined, for instance, by interpolation theory.
In addition, for bounded $D$, $H^{1/2}(\partial D)$ denotes the space of traces of $H^1(D)$ functions; $H^1_0(D)$ and $H^1_g(D)$ are the Sobolev spaces of $H^1$ functions with traces equal to zero
and equal to a given function $g\in H^{1/2}(\partial D)$,
respectively. Further, $(\cdot,\cdot)_{0,D}$ is the usual $L^2$ inner
product over~$D$. 

We employ the following multi-index
notation: for 
$\boldsymbol{\alpha}=(\alpha_1,\dots,\alpha_d)$,
\begin{align*}
	\boldsymbol x^{\boldsymbol{\alpha}} = x_1^{\alpha_1} x_2^{\alpha_2} \dots x_d^{\alpha_d}, \quad \quad \quad \partial^{\boldsymbol{\alpha}} = \partial_1^{\alpha_1} \partial_2^{\alpha_2} \dots \partial_d^{\alpha_d},
\end{align*}
with $|\boldsymbol{\alpha}|=\alpha_1+\dots+\alpha_d$, and where $\partial _\ell^\alpha$ denotes the $\alpha$-th partial derivative along direction $x_\ell$.

In the sequel, we also 
use the notation $a \lesssim b$ meaning that there exists a constant
$c>0$, independent of $\h$ and $\p$,
such that $a \leqslant c\, b$.
Finally, we use the notation $a \approx b$ in lieu of $a \lesssim b$ and $b \lesssim a$ simultaneously.

\section{Continuous problem, polygonal decompositions and functional setting} \label{section continuous problem}
Here, we want to set the target problem and some basic notation we need for the construction of the non-conforming harmonic virtual element method (VEM).
More precisely, the outline of the section is as follows. In Section \ref{subsection continuous problem}, we introduce the model problem, that is a Laplace problem on a polygonal domain.
Then, in Section \ref{subsection regular polygonal decomposition}, we define the concept of regular decompositions into polygons of the physical domain of the problem. Finally, in Section \ref{subsection non conforming Sobolev spaces}, we describe non-conforming Sobolev spaces over such polygonal decompositions.

\subsection{The continuous problem} \label{subsection continuous problem}
The target problem we aim to approximate is a Laplace problem over a polygonal domain $\Omega \subset \mathbb R^2$ with boundary $\partial \Omega$.
More precisely, given 
$\g \in  H^{1/2}(\partial \Omega)$, we look for a function $u$ solving
\begin{align} \label{Laplace problem strong formulation}
\left\{
\begin{alignedat}{2}
-\Delta \u & = 0 && \quad \text{in }\Omega \\
\u & = \g && \quad \text{on } \partial \Omega.
\end{alignedat}
\right.
\end{align}
The weak formulation of \eqref{Laplace problem strong formulation} reads
\begin{equation} \label{Laplace problem weak formulation}
\begin{cases}
\text{find } \u \in \Vg \text{ such that}\\
\a(\u,\vv) = 0\quad  \forall \vv \in \Vz,\\
\end{cases}
\end{equation}
where
\begin{equation} \label{basic notation weak formulation}
\a(\u,\vv) := (\nabla \u, \nabla \vv)_{0,\Omega},\quad \Vg := H^1_\g (\Omega),\quad \Vz := H^1_0(\Omega).
\end{equation}
Well-posedness of problem \eqref{Laplace problem weak formulation} follows from a lifting argument and the Lax-Milgram lemma.

\subsection{Regular polygonal decompositions} \label{subsection regular polygonal decomposition}
In this section, we introduce the concept of regular sequences of
polygonal decompositions of the domain $\Omega$, which will be needed in the forthcoming analysis of the method.

Let $\{\taun\}_{n\in \mathbb N}$ be a sequence of \emph{conforming} polygonal decompositions of $\Omega$;
by conforming, we mean that, for each $n\in \mathbb N$, every internal
edge $\e$ of $\taun$ is contained in the boundary of precisely two elements in the decomposition.
This automatically includes the possibility of dealing with hanging nodes. 

For all $n\in \mathbb N$, with each $\taun$, we associate $\mathcal E_n$, $\mathcal E_n^I$ and $\mathcal E_n^B$, which denote its set of edges, internal edges and boundary edges, respectively.
Moreover, with each element $\E$ of $\taun$, we associate $\mathcal E^\E$, the set of its edges.
Finally, we set for all $\E \in \taun$ and for all $n \in \mathbb N$,
\begin{equation*} 
\hE := \diam(\E), \; \quad h:= \max_{\E \in \taun} \hE, \quad \he := \text{length}(\e), \, \forall \e \in \mathcal E^\E,
\end{equation*}
and we denote by $\xE$ the centroid of $\E$.

We say that $\{\taun\}_{n\in \mathbb N}$ is a {\em regular sequence of
polygonal decompositions} if the following assumptions are satisfied:
\begin{itemize}
\item[(\textbf{D1})] there exists a positive constant
  $\rho_1$ such that, for all $n\in \mathbb N$ and for all $\E \in
  \taun$, $\he \ge \rho_1 \hE$ for all edges $\e$ of $\E$;
\item[(\textbf{D2})] there exists a positive constant
  $\rho_2$ such that, for all $n\in \mathbb N$ and for all $\E \in
  \taun$,  
$\E$ is star-shaped with respect to a ball of radius greater than or equal to $\rho_2 \hE$.
\end{itemize}
The assumptions (\textbf{D1}) and (\textbf{D2}) imply the following property:
\begin{itemize}
\item[(\textbf{D3})] there exists a constant $\Lambda \in \mathbb N$ such that,
for all $n\in \mathbb N$ and for all $\E \in \taun$, card($\mathcal
E^\E$)$\le \Lambda$, that is, 
the number of edges of 
each element is uniformly bounded.
\end{itemize}
We point out that, in this definition, 
we are not requiring any quasi-uniformity on the size of the elements.
A discussion of VEM under more general mesh assumptions is the topic of \cite{beiraolovadinarusso_stabilityVEM, BrennerGuanSung_someestimatesVEM}.

\begin{remark}
In the forthcoming analysis, we will employ a number of standard functional inequalities (such as the Poincar\'e inequality and trace inequalities).
It can be proven that the constants appearing in such inequalities depend solely on the parameters~$\rho_1$, $\rho_2$, and~$\Lambda$ introduced in (\textbf{D1})-(\textbf{D3)}. We will omit such a dependence, for ease of notation.
\end{remark}

For future use, we also define local bilinear forms on polygons $\E \in \taun$ as
\begin{equation} \label{local bilinear form on polygon}
\aE(\u,\vv) := (\nabla \u, \nabla \vv)_{0,\E} \quad \forall \u,\,\vv \in H^1(\E).
\end{equation}

\subsection{Non-conforming Sobolev spaces} \label{subsection non conforming Sobolev spaces}
Having introduced the concept of regular sequences of meshes, we pinpoint the concept of
sequences of broken and non-conforming Sobolev spaces, along with their norms.
For all $n\in \mathbb N$ and $s>0$, we define the broken Sobolev spaces on $\taun$ as
\begin{equation*} 
H^{s}(\taun) := \{\vv \in L^2(\Omega) \mid \vv_{|_\E} \in H^s(\E) \ \forall \E \in \taun\}
\end{equation*}
and the corresponding broken seminorms and norms
\begin{equation} \label{broken s Sobolev}
\vert \vv \vert^2_{s,\taun} := \sum_{\E \in \taun} \vert \vv \vert_{s,\E}^2, \quad \quad \quad \Vert \vv \Vert^2_{s,\taun} := \sum_{\E \in \taun} \Vert \vv \Vert_{s,\E}^2.
\end{equation}
Particular emphasis is stressed on the broken $H^1$ bilinear form 
\begin{equation*} 
(\u,\vv) _{1,\taun} := \sum_{\E \in \taun} (\nabla \u,\nabla \vv)_{0,\E}.
\end{equation*}
In order to define non-conforming Sobolev spaces associated with polygonal decompositions, we need to fix some additional notation.
In particular, given any internal edge $\e \in \mathcal E_n^I$ shared
by the polygons $\E^-$ and $\E^+$ in $\taun$, we
denote by $\n_{\E^\pm}^e$ 
the two outer normal unit vectors with respect
to $\E^{\pm}$. For simplicity, we will later only write
$\n_{\E^\pm}$ instead of $\n_{\E^\pm}^e$.
Moreover, for boundary edges $\e \in \mathcal
E_n^B$, we introduce the normal unit vector $\n_\Omega$ 
pointing outside $\Omega$. Having this, 
for any $\vv \in H^1(\taun)$, we set the jump operator across an edge $\e \in \mathcal E_n$ to
\begin{equation} \label{jump operator}
\llbracket \vv \rrbracket :=
\begin{cases}
\vv_{|_{\E^+}} \n_{\E^+} + \vv_{|_{\E^-}} \n_{\E^-}& \text{if } \e \in \mathcal E^I_n\\
\vv \n_\Omega & \text{if } \e \in \mathcal E^B_n.\\
\end{cases}
\end{equation}

Finally, we introduce the global non-conforming Sobolev space of order
$\k\in \mathbb N$ with respect to the decomposition $\taun$ 
incorporating 
boundary conditions in a \emph{non-conforming sense}:
%
%
%
%
Given $\g \in H^{1/2}(\partial \Omega)$ and $\k \in \mathbb N$, we define
\begin{equation} \label{non conforming space non homogeneous}
\begin{split}
H_{\g}^{1,\nc} (\taun, \k) := \{ \vv \in H^1(\taun) \, \mid \, & \int_\e \llbracket \vv \rrbracket \cdot \n \, \q_{\k-1} \, \ds = 0  \quad \forall \q_{\k-1} \in \mathbb{P}_{k-1}(e), \ \forall \e \in \mathcal E_n^I\\
& \int_\e \llbracket \vv \rrbracket \cdot \n \, \q_{\k-1} \, \ds = \int_\e \g \q_{\k-1} \, \ds \quad \forall \q_{\k-1} \in \mathbb{P}_{k-1}(e), \ \forall \e \in \mathcal E_n^B \},\\
\end{split}
\end{equation}
where $\n$ is either of the two normal unit vectors to $e$, but fixed,
if $e \in \mathcal{E}_n^I$, and $\n=\n_\Omega$, if $e \in
\mathcal{E}_n^B$.
In the homogeneous case, definition~\eqref{non conforming space
  non homogeneous} becomes
\begin{equation} \label{non conforming space homogeneous}
\begin{split}
H^{1,\nc}_0(\taun, \k) := \{\vv \in H^1(\taun) \, \mid \, \int_\e \llbracket \vv \rrbracket\cdot \n\, \q_{\k-1} \, \ds= 0 \quad \forall \q_{\k-1} \in \mathbb{P}_{k-1}(e), \ \forall \e \in \mathcal E_n \}.\\
\end{split}
\end{equation}
Importantly, the seminorm $\vert \cdot \vert_{1,\taun}$ is actually a
norm for functions in $H^{1,\nc}_0 (\taun,
k)$. In~\cite{brenner2003poincare}, the validity of the following
\Poincare\, inequality was proven: there exists a positive constant
$c_P$ only depending on $\Omega$ such that, for all $k\in \mathbb N$,
\begin{equation} \label{Poincare Brenner}
\Vert \vv \Vert_{0,\Omega} \le c_P \vert \vv \vert_{1,\taun} \quad \forall \vv \in H^{1,\nc}_0(\taun, \k).
\end{equation}

\section{Non-conforming harmonic virtual element methods} \label{section VEM}

In this section, we introduce a non-conforming harmonic virtual element method for the approximation of problem \eqref{Laplace problem weak formulation} and investigate its $\h$- and $\p$-versions. To this purpose, 
in addition to (\textbf{D1})-(\textbf{D3}), we will also require on the sequence of meshes $\{\taun\}_{n\in \mathbb N}$
the following quasi-uniformity assumption: 
\begin{itemize}
\item[(\textbf{D4})] there exists a constant
$\rho_3 \ge 1$ such that, for all $n \in \mathbb N$ and for all $\E_1$ and $\E_2$ in $\taun$, 
it holds $\h_{\E_2} \le \rho_3 \h_{\E_1}$.
\end{itemize}

We want to approximate problem \eqref{Laplace problem weak formulation} with a method of the following type:
\begin{equation} \label{VEM}
\begin{cases}
\text{find } \un \in \Vnpg \text{ such that}\\
\an (\un, \vn) = 0 \quad \forall \vn \in \Vnpz,\\
\end{cases}
\end{equation} 
where the space of trial functions $\Vnpg$ and the space of test
functions $\Vnpz$ are finite dimensional (non-conforming) spaces
on a mesh $\taun$, ``mimicking'' the infinite dimensional spaces $\Vg$ and $\Vz$, defined in \eqref{basic notation weak formulation},
respectively.
Moreover, $\an(\cdot,\cdot): \Vnpg \times \Vnpz \rightarrow \mathbb R$ is a \emph{computable} discrete bilinear form mimicking its continuous counterpart defined again in \eqref{basic notation weak formulation}.
Such approximation spaces and discrete bilinear forms have to be tailored so that method \eqref{VEM} is well-posed and provides ``good'' $\h$- and $\p$-approximation estimates.

The outline of this section is as follows.
We first introduce suitable global approximation spaces $\Vnpg$ and $\Vnpz$ in Section \ref{subsection non conforming HVE Spaces}, highlighting their approximation properties in Section \ref{subsection approximation HVEM}.
Next, in Section \ref{subsection discrete bilinear forms}, we define
and provide an \emph{explicit} discrete bilinear form and, moreover, we discuss its properties.
An abstract error analysis is carried out in Section \ref{subsection
  abstract error analysis}; such analysis is instrumental for the
$\h$- and $\p$-error estimates proved in Section \ref{subsection h and p non conforming HVEM}.
$L^2$ error bounds are provided in Section \ref{subsection L2 error estimate}.
Finally, in Section \ref{subsection extension 3D}, we give a hint
concerning the extension to the 3D case and we stress the main
differences between the 2D and 3D cases. 
%
Some details on the implementation of the method are presented in Appendix~\ref{section appendix implementation details}.

\subsection{Non-conforming harmonic virtual element spaces} \label{subsection non conforming HVE Spaces}
The aim of the present section is to introduce non-conforming harmonic virtual element spaces with \emph{uniform} degree of accuracy.
To this purpose, we begin with the description of local harmonic VE spaces, modifying those in \cite{HarmonicVEM}
into a new setting suited for building global non-conforming spaces.

Let $\p\in \mathbb N$ be a given parameter. For all $n\in \mathbb N$ and for all $\E \in \taun$, we set
\begin{equation} \label{local VE space}
\VDeltaE := \{\vn \in H^1(\E) \mid \Delta \vn = 0 \text{ in } \E,\,
(\nabla \vn\cdot\n_K)_{|_e}
\in \mathbb P_{\p-1}(\e) \ \forall \e \in \mathcal E^K \}.
\end{equation}
In words, $\VDeltaE$ 
consists of \emph{harmonic} functions with
piecewise (discontinuous) polynomial normal traces on the boundary of
$\E$.

The space $\VDeltaE$ has dimension $\NE\p$, $\NE$ being the number of edges of $\E$.
A set of $\NE\p$ degrees of freedom for $\VDeltaE$ is the following. Given $\vn \in \VDeltaE$,
\begin{equation} \label{local dofs}
\frac{1}{h_e} \int_\e \vn m_{r}^\e \, \ds \quad \forall r=0,\dots, \p-1, \, \forall \e \in \mathcal E^K ,
\end{equation}
where 
$\{m_r^\e\}_{r=0,\ldots,\p-1}$
is \emph{any} basis of
$\mathbb P_{\p-1}(\e)$. 
These degrees of freedom are in fact unisolvent since,
if $\vn \in \VDeltaE$ 
has all the degrees of freedom equal to $0$, then
\begin{equation*} 
|\nabla \vn|_{1,K}^2 = \int_\E (\underbrace{-\Delta \vn}_{=0}) \, \vn
\, \dx + \int_{\partial \E} 
(\nabla \vn\cdot\n_K)
\, \vn \, \ds = \sum_{\e \in \mathcal E^\E} \int_\e \underbrace{
(\nabla \vn\cdot\n_K)
}_{\in \mathbb P_{\p-1}(\e)} \vn \, \ds = 0,
\end{equation*}
which implies that $\vn$ is constant. This, in addition to
\begin{equation*}
h_e \vn = \int_\e  \vn \, \ds =  \int_\e 1\, \vn \, \ds =0 ,
\end{equation*}
for some edge $\e\in \mathcal E^K$, implies $v_n=0$, providing unisolvence.

We denote by $\{\varphi_{j,r}\}_{j=1\ldots,\NE \atop r=0,\ldots, \p-1}$
the local canonical basis associated with the set of degrees of freedom~\eqref{local dofs}, namely
\begin{equation} \label{definition canonical basis}
\dof_{i,s}(\varphi_{j,r}) = \begin{cases}
1 & \text{if } i=j \ \text{and} \ s=r\\
0 & \text{otherwise}\\
\end{cases}\quad \forall\, i,j =1,\dots,\NE,\, \forall\, s,r=0,\dots,\p-1.
\end{equation}
We underline that the indices $i$ and $j$ refer to the edge, whereas the indices $s$ and $r$ refer to the polynomial $m_{r}^\e$ employed in the 
definition of the local degrees of freedom \eqref{local dofs}.

It is worth to note that the local canonical basis consists of functions that are not explicitly known inside the element and even their polynomial normal traces over the boundary are unknown.

By employing the degrees of freedom defined in \eqref{local dofs}, it
is possible to compute the following two projectors. The first one is
the edge $L^2$ projector onto the space of polynomials of degree
$\p-1$ 
\begin{equation} \label{L2 edge projector}
\begin{split}
\Pie: \, &\, \VDeltaE|_\e \rightarrow \mathbb P_{\p-1}(\e),  \\
&\int_\e (\vn - \Pie \vn)\qpmue \ \ds= 0\quad \forall \vn \in \VDeltaE, \, \forall \qpmue \in \mathbb P_{\p-1}(e). 
\end{split}
\end{equation}
The second one is the bulk $H^1$ projector onto the space of harmonic polynomials of degree $\p$
\begin{equation} \label{H1 bulk projector}
\begin{split}
\PiEextended = \PiE : \, &\, \VDeltaE \rightarrow \mathbb H_\p(\E),\\ 
&\int_\E \nabla(\vn - \PiE\vn) \cdot \nabla \qDeltap \, \dx = 0\quad \forall \vn \in \VDeltaE,\, \forall \qDeltap \in \mathbb H_\p(\E),   \\
&\int_{\partial \E} (\vn - \PiE\vn) \, \ds = 0\quad \forall \vn \in \VDeltaE, 
\end{split}
\end{equation}
where the last condition is imposed in order to define the projector in a unique way.

We are ready to define global non-conforming harmonic VE spaces, which
incorporate Dirichlet boundary conditions in a ``non-conforming sense''.
Let $\p \in \mathbb N$ 
be a given parameter.
Then, for any $\g \in H^{1/2}(\partial \Omega)$, we set
\begin{equation} \label{global non conforming harmonic VES}
\Vnpg := \{\vn \in H^{1,\nc}_{\g} (\taun,p) \, \mid \, \vv_{n|_\E} \in \VDeltaE \ \forall K \in \taun \}.
\end{equation}

We observe the following facts:
\begin{itemize}
\item Definition \eqref{global non conforming harmonic VES} includes the space of test functions $\Vnpz$, by selecting $\g=0$.
\item The parameter $p$ in~\eqref{global non conforming harmonic VES} indicates the level of non-conformity of the method.
The fact that the non-conformity is defined with respect to Dirichlet traces allows us to easily couple the local degrees of freedom into a global set, 
provided that we choose the same value $p$ for the non-conformity parameter and for the polynomial degree entering
definition~\eqref{local VE space} of the local spaces. The resulting global set of degrees of freedom is of dimension $\card(\mathcal E_n)\p$.
\item Dirichlet boundary conditions on $\partial \Omega$ are imposed weakly via the definition of the non-conforming spaces \eqref{non conforming space non homogeneous} and \eqref{non conforming space homogeneous}.
For instance, 
given a Dirichlet datum 
$\g$, on all boundary edges $\e \in \mathcal E_n^B$, we set
\[
\int_\e \llbracket \vn \rrbracket \cdot \n_\Omega \, \qpmue \, \ds = \int_\e \vn \qpmue \, \ds= \int_\e \g \qpmue \, \ds \quad \forall \vn \in \Vnpg,\ \forall \qpmue \in \mathbb P_{\p-1}(\e).
\]
\end{itemize}

\begin{remark} \label{remark how to deal with Dirichlet boundary conditions}
We highlight that, at the discrete level, one should also take into account the approximation of the Dirichlet boundary condition $\g$.
In practice, assuming $\g\in H^{\frac{1}{2} + \varepsilon}(\partial \Omega)$, for any $\varepsilon >0$ arbitrarily small, and
denoting by $g_\p$ the approximation of $\g$ obtained by interpolating $\g$ at the $\p+1$ Gau\ss-Lobatto nodes on each edge in $\mathcal E_n^B$, one should define the trial space as
\begin{equation*} 
\Vnpg := \{\vn \in H^{1,\nc}_{\g_\p} (\taun,p) \, \mid \, \vv_{n_{|_\E}} \in \VDeltaE \ \forall K \in \taun \}.
\end{equation*}
With this definition, in the forthcoming analysis (see Proposition \ref{theorem best approximation VE functions}, Theorem \ref{theorem abstract error analysis}, Theorem \ref{theorem h and p VEM},
Proposition \ref{theorem best error Manzini space}, and Theorem \ref{theorem L2 estimates} below),
an additional term related to 
the approximation of the Dirichlet datum via Gau\ss-Lobatto interpolants should be taken into account.
However, following \cite[Theorem 4.2, Theorem 4.5]{bernardimaday1992polynomialinterpolationinsobolev}, it is possible to show that the $\h$- and $\p$-rates
of convergence of the method are not spoilt by this term.
For this reason and for the sake of simplicity, we will neglect 
in the following the presence of this term and assume 
that the approximation space is the one defined in \eqref{global non conforming harmonic VES}.
\end{remark}

\subsection{Approximation properties of functions in non-conforming harmonic virtual element spaces} \label{subsection approximation HVEM}
In this section, we deal with approximation properties of functions in the non-conforming harmonic VE spaces $\Vnpg$ and $\Vnpz$.

Since
$\h$- and $\p$-approximation properties of harmonic functions via
harmonic polynomials are 
known,
see e.g. \cite{babumelenk_harmonicpolynomials_approx, hmps_harmonicpolynomialsapproximationandTrefftzhpdgFEM},
we want to relate best approximation estimates in the non-conforming harmonic VE spaces to the corresponding ones in {\em discontinuous} harmonic polynomial spaces.
In particular, we prove the following result.

\begin{prop} \label{theorem best approximation VE functions}
Given $\g \in H^{1/2}(\partial \Omega)$, let $\u \in \Vg$, where $\Vg$ is defined in \eqref{basic notation weak formulation}.
For any polygonal partition $\taun$ of $\Omega$, there exists $\uI \in \Vnpg$, with $\Vnpg$ introduced in \eqref{global non conforming harmonic VES}, such that
\begin{align*}
\vert \u - \uI \vert_{1,\taun} \le 2 \vert \u - \qDeltap \vert_{1,\taun}\quad \forall \qDeltap \in \mathcal{S}^{p,\Delta,-1}(\taun),
\end{align*}
where $\mathcal{S}^{p,\Delta,-1}(\taun)$ is the space of discontinuous piecewise harmonic polynomials of degree at most $p$, that is,
\begin{align} \label{space SpDelta-1}
\mathcal{S}^{p,\Delta,-1}(\taun):= \{ q \in L^2(\Omega): \, q_{|_\E} \in \mathbb H_\p(\E) \ \forall \E \in \taun \}.
\end{align}
\end{prop}
\begin{proof}
Define $\uI \in \Vnpg$ by 
\begin{equation} \label{virtual interpolant}
\int_\e (\u - \uI) \qpmue \, \ds= 0\quad \forall \qpmue \in \mathbb
P_{\p-1}(\e),\ 
\forall \e \in \mathcal E_n,
\end{equation}
that is, we fix the degrees of freedom \eqref{local dofs} of $\uI$ to
be equal to the values of the same functionals applied to the solution $\u$.
Having this, it holds
\begin{equation} \label{triangulating with piecewise harmonic}
\vert \u - \uI \vert_{1,\taun} \le \vert \u - \qDeltap \vert_{1,\taun} + \vert \uI - \qDeltap \vert_{1,\taun} \quad \forall \qDeltap \in \mathcal{S}^{p,\Delta,-1}(\taun),
\end{equation}
where $\mathcal{S}^{p,\Delta,-1}(\taun)$ is defined in \eqref{space SpDelta-1}.
We focus on the second term on the right-hand side of \eqref{triangulating with piecewise harmonic}. By integrating by parts and using \eqref{virtual interpolant}, together with the definition of the space \eqref{global non conforming harmonic VES}, we get
\begin{equation}\label{bound1}
\begin{split}
\vert \uI - \qDeltap \vert^2_{1,\taun} &= \sum_{\E \in \taun} \vert \uI - \qDeltap \vert^2_{1,\E} \\
& = \sum_{\E \in \taun} \left\{ \int_\E (\uI - \qDeltap)
  (\underbrace{-\Delta(\uI - \qDeltap)}_{=0}) \, \dx + \sum_{\e \in
    \mathcal E^\E}\int_\e (\uI-\qDeltap) \, 
\nabla(\uI-\qDeltap)\cdot\n_K
\, \ds  \right\}\\
& = \sum_{\E \in \taun} \sum_{\e \in \mathcal E^\E} \int_\e (\u -
\qDeltap) \, 
\nabla(\uI-\qDeltap)\cdot\n_K
 \, \ds.
\end{split}
\end{equation}
By expanding the right-hand side of~\eqref{bound1} and using
the Cauchy-Schwarz inequality, we obtain
\[
\begin{split}
\vert \uI - \qDeltap \vert_{1,\taun}^2	& = \sum_{\E \in \taun} \int_{\E} \nabla (\u-\qDeltap) \cdot \nabla (\uI - \qDeltap) \, \dx + \int_\E (\u-\qDeltap) \underbrace{\Delta(\uI-\qDeltap)}_{=0} \, \dx \\
& \le \vert \u - \qDeltap \vert_{1,\taun} \vert \uI - \qDeltap \vert_{1,\taun}.
\end{split}
\]
Inserting this into \eqref{triangulating with piecewise harmonic} gives the result.
\end{proof}

We remark that,
with a similar proof of that of Proposition~\ref{theorem best approximation VE functions}, one can prove an equivalent result for the non-conforming (non-harmonic) 
VE spaces of~\cite{nonconformingVEMbasic}; 
see Proposition~\ref{theorem best error Manzini space} below.

\subsection{Discrete bilinear forms} \label{subsection discrete
  bilinear forms}

In this section, we complete the definition of the method~\eqref{VEM}
by introducing a suitable bilinear form $\an(\cdot,\cdot)$, which is
explicitly computable.
We follow here the typical VEM gospel \cite{VEMvolley, hpVEMcorner, HarmonicVEM}. 
It is important to highlight that the local bilinear forms $\aE(\cdot,\cdot)$ defined in \eqref{local bilinear form on polygon} are not explicitly computable on the whole discrete spaces since an explicit representation of functions in the harmonic VE spaces is not available
in closed form.

Therefore, we aim at introducing explicit computable discrete bilinear forms $\anE(\cdot,\cdot)$ which mimic their continuous counterparts $\aE(\cdot,\cdot)$.
To this purpose, we observe that the Pythagorean
theorem yields
\begin{equation} \label{Pythagoras theorem}
\aE(\un,\vn) = \aE( \PiE \un, \PiE \vn) + \aE( (I-\PiE) \un, (I-\PiE) \vn) \quad \forall \un,\,\vn \in \VDeltaE,
\end{equation}
where we recall that $\PiE$ is defined in \eqref{H1 bulk projector}.
The first term on the right-hand side of \eqref{Pythagoras theorem} is computable, whereas the second is not. Thus, following \cite{HarmonicVEM} and the references therein,
we replace this term by a \emph{computable} symmetric bilinear form $\SE:\ker(\PiE)\times \ker (\PiE)\rightarrow \mathbb R$, such that
\begin{equation} \label{stability bounds}
c_*(\p) \vert \vn \vert^2_{1,\E} \le \SE(\vn,\vn) \le c^*(\p) \vert \vn \vert^2_{1,\E}\quad \forall \vn \in \ker (\PiE),
\end{equation}
where $c_*(\p)$ and $c^*(\p)$ are two positive constants which may
depend on $\p$, but are independent of $\E$ and, in particular, of $\hE$.

Hence, depending on the choice of the stabilization, a class of candidates for the local discrete symmetric bilinear forms is
\begin{equation} \label{local discrete bilinear form}
\anE(\un,\vn) = \aE(\PiE \un, \PiE \vn) + \SE( (I-\PiE) \un, (I-\PiE) \vn) \quad \forall \un,\,\vn \in \VDeltaE.
\end{equation}
The forms $\anE(\cdot,\cdot)$ satisfy the two following properties:
\begin{itemize}
\item[(\textbf{P1})] \textbf{$\p$-harmonic consistency}: for all $\E \in \taun$ and for all $\p \in \mathbb N$,
\begin{equation} \label{consistency}
\aE(\qDeltap, \vn) = \anE(\qDeltap, \vn) \quad\forall \qDeltap\in \mathbb H_\p(\E),\, \forall \vn \in \VDeltaE;
\end{equation}
\item[(\textbf{P2})] \textbf{stability}: for all $\E \in \taun$ and for all $\p \in \mathbb N$,
\begin{equation} \label{stability}
\alpha_*(\p) \vert \vn \vert^2_{1,\E} \le \anE(\vn,\vn) \le \alpha^*(\p) \vert \vn \vert^2_{1,\E} \quad \forall \vn \in \VDeltaE,
\end{equation}
where $\alpha_*(\p) =\min (1,c_*(\p))$ and $\alpha^*(\p) =\max (1,c^*(\p))$.
\end{itemize}
Owing to property (\textbf{P1}), $\p$ can be addressed to as
\emph{degree of accuracy of the method}, since whenever either of its
two entries is a harmonic polynomial of degree $\p$,
the local discrete bilinear form 
can be computed exactly, up to machine precision. Moreover, since $\anE(\cdot,\cdot)$ is assumed to be symmetric, (\textbf{P2}) implies continuity
\begin{align} \label{continuity_an}
\anE(\un,\vn) \le \left( \anE(\un,\un) \right)^{1/2} \left( \anE(\vn,\vn) \right)^{1/2} \le \alpha^*(\p) \vert \un \vert_{1,\E} \vert \vn \vert_{1,\E} \quad \forall \un, \vn \in \VDeltaE.
\end{align}

The global discrete bilinear form is defined as
\begin{equation} \label{global discrete bilinear form}
\an(\un,\vn) = \sum_{\E \in \taun} \anE(\un,\vn) \quad \forall \un \in  V_{n,g_1}^{\Delta,p}, \, \forall \vn \in V_{n,g_2}^{\Delta,p}
\end{equation}
for all $g_1,g_2 \in H^{1/2}(\partial \Omega)$.
The remainder of this section is devoted to introduce an explicit
stabilization $\SE(\cdot,\cdot)$ with explicit bounds of the constants $c_*(\p)$ and $c^*(\p)$.

For all $\E \in \taun$, we define
\begin{equation} \label{explicit stabilization}
\SE(\un,\vn) = \sum_{\e\in \mathcal E^\E} \frac{\p}{\he} (\Pie \un, \Pie \vn)_{0,\e} \quad \forall \un,\, \vn \in \ker (\PiE).
\end{equation}
For this choice of stabilization forms, the following result holds
true.

\begin{thm} \label{theorem explicit bounds on explicit stability}
Assume that (\textbf{D1}) and (\textbf{D2}) hold true. Then, for any $\E \in \taun$, the stabilization $\SE(\cdot,\cdot)$ defined in \eqref{explicit stabilization} satisfies \eqref{stability bounds} with the bounds
\begin{equation} \label{explicit bounds stability}
c_*(\p) \gtrsim \p^{-2},\quad \quad c^*(\p) \lesssim \begin{cases}
\p\left( \frac{\log(\p)}{\p} \right)^{\frac{\lambda_\E}{2 }}  					& \text{if } \E \text{ is convex}\\
\p\left( \frac{\log(\p)}{\p} \right)^{\frac{\lambda_\E}{2 \omega_\E} - \varepsilon} 	& \text{otherwise }\\
\end{cases}
\end{equation}
for all $\varepsilon>0$ arbitrarily small, where the hidden constants in \eqref{explicit bounds stability} are independent of $\h$ and $\p$,
and where $\omega_\E \pi$ and $\lambda_\E \pi$, with $\omega_\E$ and $\lambda_\E \in (0,2)$, denote the largest interior and the smallest exterior angles of $\E$, respectively.
\end{thm}
\begin{proof}
We assume, without loss of generality, that $\hE = 1$; the general result follows from a scaling argument.

For any function $\vn$ in
$\VDeltaE$, we have
\begin{equation} \label{estimate stabilization}
\begin{split} 
\vert \vn \vert^2_{1,\E} &= 
-\int_\E \underbrace{(\Delta \vn)}_{=0}\vn \,\dx + \int_{\partial\E}\nabla\vn\cdot\n_K \, \vn \, \ds\\
&= \sum_{\e \in \mathcal E^K} \int_\e 
\nabla\vn\cdot\n_K  (\Pie \vn) \, \ds \le \Vert \nabla \vn \cdot \n_\E \Vert_{0,\partial \E} \Vert \Pipartial \vn \Vert_{0,\partial \E}
\end{split}
\end{equation}
where we have set, with an abuse of notation, $(\Pipartial \vn) _{|_\e} = \Pie (\vv _{n_{|_\e}})$.
We prove that
\begin{equation}\label{eq:polynbound}
\Vert \nabla \vn \cdot \n_\E \Vert_{0,\partial \E} \lesssim p^{\frac{3}{2}}
\Vert \nabla \vn \cdot \n_\E \Vert_{-\frac{1}{2},\partial \E}.
\end{equation}
To this end, 
we set, for the sake of simplicity, $r_\p:=\nabla \vn \cdot
\n_\E$, and consider the case $r_\p\ne 0$. One has $r_\p \in
L^2(\partial \E)$ with $r_{p|_e} \in \mathbb P_\p(e)$ for all $e \in
\mathcal E^{\E}$. In general, $r_p \notin H^{1/2}(\partial \E)$. 
Further, we introduce the piecewise bubble function $b_{\partial \E} \in H^{1/2}(\partial \E)$ defined edgewise as
\begin{equation*}
(b_{\partial \E})_{|_e}(\x):=(\beta \circ \phi_e^{-1})(\x) \quad \forall e \in \mathcal{E}^{\E},
\end{equation*}
where $\phi_e: [-1,1] \to e$ is the linear transformation mapping the interval $[-1,1]$ to the edge $\e$, and $\beta: [-1,1] \to [0,1]$ is the 1D quadratic bubble function 
$\beta(x):=4(1-x^2)$.

From the definition of the $H^{-1/2}(\partial \E)$ norm,
the fact that $r_\p \in L^2(\partial \E)$, and $r_p b_{\partial \E}
\in H^{1/2}(\partial \E) \backslash \{0\}$, we have
\begin{equation} \label{estimate stabilization 6}
\lVert r_p \rVert_{-\frac{1}{2},\partial \E} 
=\sup_{\psi \in H^{1/2}(\partial \E) \backslash \{0\}}
\frac{(r_p,\psi)_{0,\partial \E}}{\lVert \psi \rVert_{\frac{1}{2},\partial \E}}
\ge
\frac{(r_p,r_p b_{\partial \E})_{0,\partial \E}}{\lVert r_p b_{\partial \E} \rVert_{\frac{1}{2},\partial \E}}
=\frac{\lVert r_p b_{\partial \E}^{\frac{1}{2}} \rVert^2_{0,\partial \E}}{\lVert r_p b_{\partial \E} \rVert_{\frac{1}{2},\partial \E}}.
\end{equation}
We have the two following 
polynomial $\p$-inverse inequalities: 
\begin{equation} \label{estimate stabilization 3_4}
\lVert r_p b_{\partial \E} \rVert_{0,e}
\le \lVert r_p b_{\partial \E}^{\frac{1}{2}} \rVert_{0,e},
\quad
\lvert r_p b_{\partial \E} \rvert_{1,e}
\lesssim p \lVert r_p b_{\partial \E}^{\frac{1}{2}} \rVert_{0,e}
\quad \forall e \in \mathcal{E}^\E.
\end{equation}
The first one
is a direct consequence of the fact that the range of $b_{\partial \E}$ is
$[0,1]$, and the second one 
follows from \cite[Lemma 2]{bank2013saturation}.
Using \eqref{estimate stabilization 3_4}, summing over all edges $e \in \mathcal{E}^\E$, and applying interpolation theory, lead to
\begin{equation*} 
\lVert r_p b_{\partial \E} \rVert_{\frac{1}{2},\partial \E} \lesssim p^{\frac{1}{2}} \lVert r_p b_{\partial \E}^{\frac{1}{2}} \rVert_{0,\partial \E},
\end{equation*}
which, together with
\eqref{estimate stabilization 6}, gives
\begin{equation*}
\lVert r_p \rVert_{-\frac{1}{2},\partial \E} 
\gtrsim p^{-\frac{1}{2}} \lVert r_p b_{\partial \E}^{\frac{1}{2}} \rVert_{0,\partial \E}
\gtrsim p^{-\frac{3}{2}} \lVert r_p \rVert_{0,\partial \E},
\end{equation*}
where, in the last inequality, \cite[Lemma 4]{bernardi2001error} was
used.
The bound \eqref{eq:polynbound} follows immediately.

From \eqref{estimate stabilization} and \eqref{eq:polynbound},
taking also into account that $\Delta \vn = 0$ in $K$, we get
\[
\vert \vn \vert^2_{1,\E} 
\lesssim \p^{\frac{3}{2}} \Vert \nabla \vn \cdot \n_\E \Vert_{-\frac{1}{2},\partial \E} \Vert \Pipartial \vn \Vert_{0,\partial \E}
\lesssim \p^{\frac{3}{2}} \vert \vn \vert_{1,\E} \Vert \Pipartial \vn \Vert_{0,\partial \E},
\]
where in the last step we have used a Neumann trace inequality, see e.g. \cite[Theorem A.33]{SchwabpandhpFEM}.
This proves the first inequality of \eqref{stability bounds} with $c_*(\p) \gtrsim p^{-2}$.

In order to prove the second one, we can write
\begin{equation} \label{isola}
\Vert \Pipartial \vn \Vert_{0,\partial \E} \le \Vert \vn \Vert _{0,\partial \E} \lesssim \Vert \vn \Vert_{0,\E}^{\frac{1}{2}} \vert \vn \vert_{1,\E}^{\frac{1}{2}},
\end{equation}
where we have used the stability of the $L^2$ projection,
the multiplicative trace inequality, and the \Poincare\, inequality, see~\cite{brenner2003poincare}, 
which is valid since $\vn \in \ker(\PiE)$ and thus has zero mean value on $\partial\E$, see~\eqref{H1 bulk projector}. 

Let us bound the first factor on the right-hand side of~\eqref{isola}. To this end, we define $\vnbar$ as the average of $\vn$ over the polygon $\E$. A triangle inequality yields
\begin{equation} \label{triangle inequality isola}
\Vert \vn \Vert_{0,\E} \le \Vert \vn - \vnbar \Vert_{0,\E} + \Vert \vnbar \Vert_{0,\E}.
\end{equation}
Recalling that $\vn$ has zero average over $\partial \E$, we have
\[
\Vert  \vnbar \Vert_{0,\E} = \vert \E \vert ^{\frac{1}{2}} \vert \vnbar \vert = \frac{\vert \E \vert^{\frac{1}{2}}}{\vert \partial \E \vert} \left\vert \int_{\partial \E} \vnbar - \vn  \, \ds \right\vert.
\]
A Cauchy-Schwarz inequality, together with the multiplicative trace inequality, yields
\[
\Vert  \vnbar \Vert_{0,\E} \lesssim \Vert \vn - \vnbar \Vert_{0,\E}^{\frac{1}{2}} \vert \vn \vert^{\frac{1}{2}}_{1,\E}.
\]
Inserting this inequality in~\eqref{triangle inequality isola} gives
\begin{equation}\label{eq:boundnewstar}
\Vert  \vn \Vert_{0,\E} \lesssim \Vert \vn - \vnbar \Vert_{0,\E}  + 
\Vert \vn - \vnbar \Vert_{0,\E}^{\frac{1}{2}} \vert \vn \vert^{\frac{1}{2}}_{1,\E}. 
\end{equation}
From~\cite[Lemma 3.2]{HarmonicVEM}, we have
\[
\Vert  \vn  - \vnbar\Vert_{0,\E}^{\frac{1}{2}} \lesssim
\begin{cases}
\left( \frac{\log(\p)}{\p} \right)^{\lambda_\E} \vert  \vn \vert_{1,\E}					& \text{if } \E \text{ is convex}\\
\left( \frac{\log(\p)}{\p} \right)^{\frac{\lambda_\E}{\omega_\E} - \varepsilon} \vert  \vn \vert_{1,\E}	& \text{otherwise}\\
\end{cases}
\]
for all $\varepsilon>0$ arbitrarily small. 
Inserting this into~\eqref{eq:boundnewstar} gives
\[
\Vert  \vn \Vert_{0,\E} \lesssim
\begin{cases}
\left( \frac{\log(\p)}{\p} \right)^{\frac{\lambda_\E}{2}} \vert  \vn \vert_{1,\E}					& \text{if } \E \text{ is convex}\\
\left( \frac{\log(\p)}{\p} \right)^{\frac{\lambda_\E}{2\omega_\E} - \varepsilon} \vert  \vn \vert_{1,\E}	& \text{otherwise},\\
\end{cases}
\]
which, together with~\eqref{isola}, gives \eqref{stability bounds} with $c^*(p)$
as in~\eqref{explicit bounds stability}.
\end{proof}

Owing to \eqref{stability} and \eqref{explicit bounds stability} one deduces
\begin{equation*} 
\alpha_*(\p) \gtrsim \p^{-2},\quad \quad \alpha^*(\p) \lesssim
\begin{cases}
\p\left( \frac{\log(\p)}{\p} \right)^{\frac{\lambda_\E}{2 }} 						& \text{if } \E \text{ is convex}\\
\p\left( \frac{\log(\p)}{\p} \right)^{\frac{\lambda_\E}{2 \omega_\E} - \varepsilon}  	& \text{otherwise }\\
\end{cases}
\end{equation*}
for all $\varepsilon>0$ arbitrarily small.

\begin{remark}
In the conforming harmonic VEM setting \cite{HarmonicVEM}, the following local stabilization forms were introduced:
\[
\SE(\un,\vn) = (\un,\vn)_{\frac{1}{2}, \partial \E} \quad \forall \E \in \taun.
\]
It was proven that employing such stabilization forms leads to have stability constants $\alpha_*(\p)$ and $\alpha^*(\p)$ that are independent of the degree of accuracy~$\p$.
However, in the present non-conforming setting, such a stabilization is not computable, as the Dirichlet traces of functions in the local VE spaces are not available in closed form.

\end{remark}

We investigate numerically the behavior of the conditioning of the global VE matrix in terms of the degree of accuracy~$\p$, when employing the local stabilization forms in~\eqref{explicit stabilization}.
In Figure \ref{fig:condition number}, we plot the condition number for
different values of $\p$, when computing the global stiffness matrix on a Cartesian mesh, a Voronoi-Lloyd mesh , and an Escher horses mesh, see Figure~\ref{fig:meshes}, and note that it grows algebraically  with~$\p$.
%
We remark that the condition number of standard (non-harmonic) VEM
can grow exponentially or algebraically with~$\p$, depending on the choice of the internal degrees of freedom. This was investigated in~\cite{fetishVEM2D}.

\begin{figure}[htbp]
\centering
\begin{minipage}{0.49\textwidth}
\includegraphics[width=\textwidth]{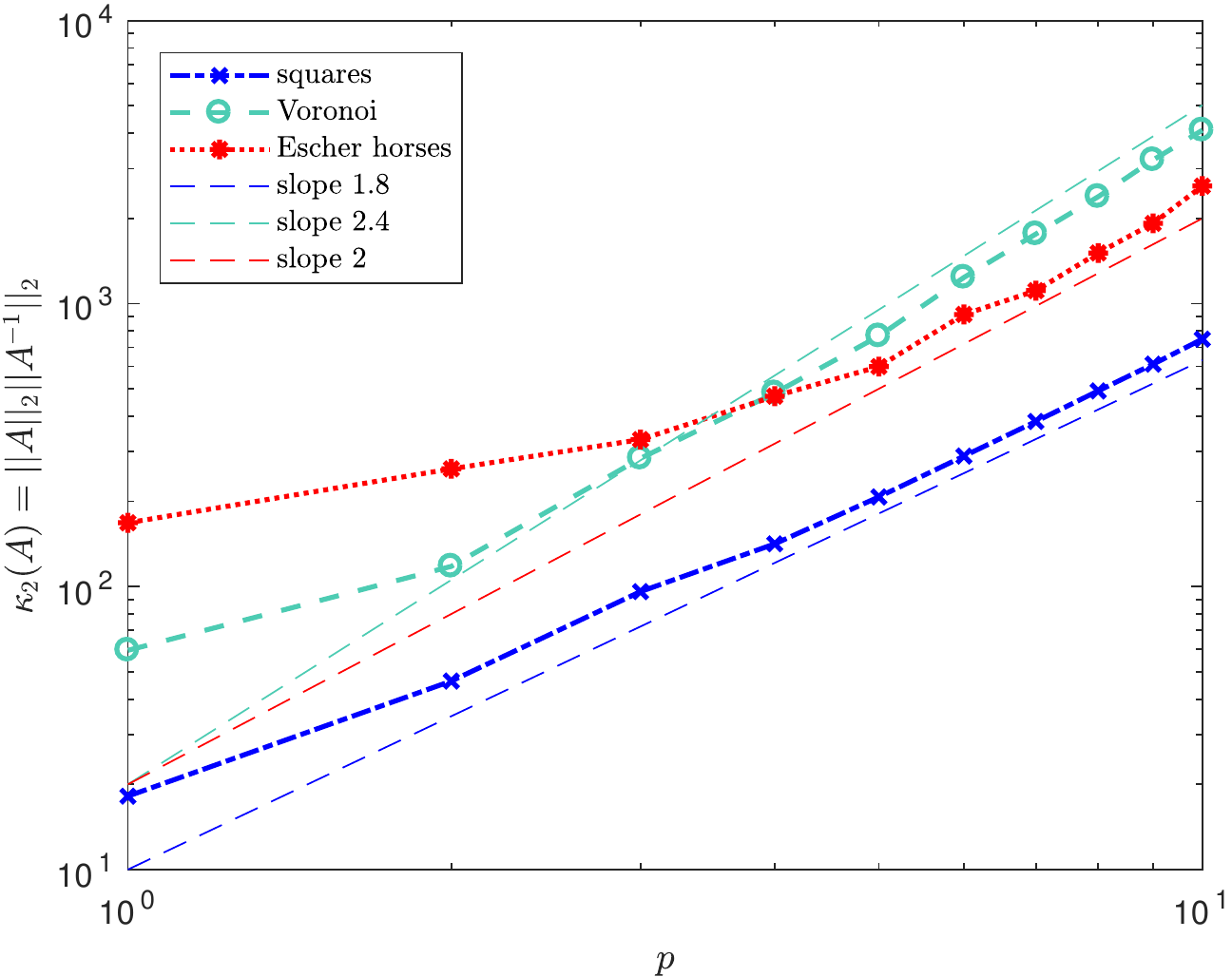}
\end{minipage}
\caption{Condition number for different values of $\p$ of the global stiffness matrix obtained
 with the local stabilization forms in~\eqref{explicit stabilization}.
A Cartesian mesh, a Voronoi-Lloyd mesh, and an Escher horses mesh have been considered.
We observe algebraic growth of the condition number with~$\p$ for all the tested meshes.}
\label{fig:condition number} 
\end{figure}



\subsection{Abstract error analysis} \label{subsection abstract error
  analysis}
Along the lines of
\cite{VEMvolley, hpVEMbasic, HarmonicVEM}, we provide in this section an abstract error analysis of the method \eqref{VEM}, taking into account the non-conformity of the approximation. To this purpose, we introduce the auxiliary bilinear form
\begin{equation} \label{non conformity term}
\mathcal N_n : H^1(\Omega) \times 
H^{1,\nc}_{0} (\taun,\p) 
\rightarrow \mathbb R,\quad \mathcal N_n (\u,\vv)  = \sum_{\e \in \mathcal E_n} \int_\e \nabla \u \cdot\llbracket \vv \rrbracket \, \ds.
\end{equation}
The following 
convergence result holds true.
\begin{thm} \label{theorem abstract error analysis}
Assume that (\textbf{D1}) and (\textbf{D2}) hold true and
consider the non-conforming harmonic VEM \eqref{VEM} defined by choosing the harmonic VE spaces as in \eqref{global non conforming harmonic VES} and \eqref{local VE space}, with level of non-conformity, as
well as degree of accuracy, equal to $\p$, and by choosing the discrete bilinear form as in~\eqref{global discrete bilinear form} and~\eqref{local discrete bilinear form},
with stabilization form $\SE(\cdot,\cdot)$ satisfying \eqref{stability bounds}. Then, the method is well-posed and
the following bound holds true:
\begin{equation} \label{abstract bound}
\vert \u - \un \vert_{1,\taun} \le 
 \frac{\alpha^*(\p)}{\alpha_*(\p)} \left\{ 
6 \inf_{\qDeltap \in \mathcal{S}^{p,\Delta,-1}(\taun)}
\vert \u - \qDeltap \vert_{1,\taun} + \sup_{\vn \in \Vnpz} \frac{\mathcal N_n (\u,\vn)}{\vert \vn \vert_{1,\taun}}  \right\}, 
\end{equation}
where we recall that $\mathcal{S}^{p,\Delta,-1}(\taun)$ is defined in \eqref{space SpDelta-1}, $\mathcal N_n (\cdot, \cdot)$ 
is given in \eqref{non conformity term}, and the stability constants $\alpha_*(\p)$ and
$\alpha^*(\p)$ are introduced in \eqref{stability}. 
\end{thm}
\begin{proof}
The well-posedness of the method 
follows directly from~\eqref{Poincare Brenner}, \eqref{stability} and the Lax-Milgram lemma. 

For the bound \eqref{abstract bound}, 
we observe that
\[
\vert \u - \un \vert_{1,\taun} \le \vert \u - \uI \vert_{1,\taun} + \vert \un - \uI \vert_{1,\taun} \quad \forall \uI \in \Vnpg.
\]
We estimate the second term on the right-hand side. 
Set $\deltan:= \un - \uI$. Since $\un,\uI\in \Vnpg$, then
  $\deltan\in \Vnpz$. Therefore,
for all $\qDeltap \in \mathcal{S}^{p,\Delta,-1}(\taun)$,
using~\eqref{stability}, \eqref{VEM} and \eqref{consistency}, we have
\[
\begin{split}
\vert \deltan \vert^2_{1,\taun} 	& = \sum_{\E \in \taun} \vert \deltan \vert_{1,\E}^2 \le \frac{1}{\alpha_*(\p)} \sum_{\E \in \taun} \anE(\deltan,\deltan) = -\frac{1}{\alpha_*(\p)} \sum_{\E \in \taun} \anE(\uI,\deltan)\\
& = -\frac{1}{\alpha_*(\p)} \left\{ \sum_{\E \in \taun} \left[ \anE(\uI - \qDeltap, \deltan) + \aE(\qDeltap - \u,\deltan) \right] + \sum_{\E \in \taun} \aE(\u, \deltan) \right\}.\\
\end{split}
\]
The last term on the right-hand side can be rewritten in the spirit of non-conforming methods.
More precisely, we observe that an
integration by parts, the fact that 
$\Delta u=0$ in every $\E\in \taun$,
and the definition \eqref{non conformity term},
yield
\[
\sum_{\E \in \taun} \aE (\u, \deltan) = \sum_{\E \in \taun}
\int_{\partial \E} 
\nabla \u\cdot\n_\E
\,\deltan \, \ds = \sum_{\e \in \mathcal E_n} \int_\e \nabla \u \cdot \llbracket \deltan \rrbracket \, \ds = \mathcal N_n(\u,\deltan).
\]
This, together with the stability property \eqref{stability}, 
  the triangle and the Cauchy-Schwarz inequalities, gives
\[
\vert \deltan \vert^2_{1,\taun}\le
\frac{1}{\alpha_*(\p)}\Big[\Big(\alpha^*(\p) (\vert u_I-u
    \vert_{1,\taun}+\vert u- \qDeltap\vert_{1,\taun})
+\vert \qDeltap-u \vert_{1,\taun}\Big)
\vert \deltan \vert_{1,\taun}
+\mathcal N_n(\u,\deltan)
\Big].
\]
Therefore, using Proposition~\ref{theorem best approximation VE functions} and $\alpha^*(\p)\ge 1$, we obtain
\[
\begin{split}
\vert \deltan \vert_{1,\taun}&\le
\frac{1}{\alpha_*(\p)}\left[\alpha^*(\p) (2\vert u-\qDeltap
    \vert_{1,\taun}+\vert u- \qDeltap\vert_{1,\taun})
+\vert \qDeltap-u \vert_{1,\taun}+\frac{\mathcal N_n(\u,\deltan)}{\vert \deltan
  \vert_{1,\taun}}\right]\\
&\le
\frac{\alpha^*(\p)}{\alpha_*(\p)} \left[ 4\vert u-  \qDeltap\vert_{1,\taun}+\frac{\mathcal N_n(\u,\deltan)}{\vert \deltan
  \vert_{1,\taun}}
\right],
\end{split}
\]
and bound~\eqref{abstract bound} readily follows.
\end{proof}
We refer to the term $\frac{\alpha^*(\p)}{\alpha_*(\p)}$ appearing in \eqref{abstract bound} as \emph{pollution factor}.
\begin{remark} \label{remark on abstract error analysis}
It is interesting to note that the counterpart of Theorem \ref{theorem abstract error analysis} in the conforming version of the harmonic VEM in~\cite{HarmonicVEM} states that the error of the method is bounded, up to a constant
times the pollution factor $\frac{\alpha^*(\p)}{\alpha_*(\p)}$,
by a best approximation error with respect to piecewise discontinuous
harmonic polynomials, plus the best approximation
error with respect to functions in the global approximation space.
In the non-conforming setting of the present paper, however, the
latter term is not present, thanks to Proposition~\ref{theorem best
  approximation VE functions}. The additional term here is related to the
non-conformity.
\end{remark}

\subsection{$\h$- and $\p$-error analysis} \label{subsection h and p non conforming HVEM}
This section is devoted to the $\h$- and $\p$-analysis of the method \eqref{VEM} employing non-conforming harmonic VE spaces with degree of non-conformity equal to the degree of accuracy of the method.

For the analysis, we have to discuss how to bound the two terms on the right-hand side of \eqref{abstract bound} in terms of $\h$ and $\p$.
The first term, i.e., the best approximation error with respect to discontinuous
harmonic polynomials, can be dealt with following \cite{melenk_phdthesis, melenk1999operator}.
In particular, we recall the following result from~\cite[Theorem 2.9]{melenk1999operator} (see also \cite[Chapter~II]{melenk_phdthesis}).

\begin{lem} \label{lemma approximation harmonic polynomials}
Under the 
star-shapedness
assumption (\textbf{D2}), 
for a given $\E \in \taun$, 
we denote by $\lambda_\E\, \pi$, $0<\lambda_\E<2$, its smallest exterior angle.
Then, for every harmonic function $\u$ in $H^{s+1}(\E)$, $s \ge 0$, 
there exists a
sequence $\{\qDeltap\}_{\p\in \mathbb N}$,
with $\qDeltap \in \mathbb H_\p(\E)$ for all $\p \in \mathbb N$ with $\p \ge s-1$, such that
\begin{equation}  \label{Babu Melenk estimate}
\vert \u - \qDeltap \vert_{1,\E} \le c \hE^s \left( \frac{\log(\p)}{\p}  \right)^{\lambda_\E s} \Vert \u \Vert_{s+1,\E},
\end{equation}
for some positive constant $c$ depending only on $\rho_2$.
\end{lem}
\begin{remark}
We underline that the $p$-version approximation of harmonic functions by means of harmonic polynomials has different rates of convergence than that of generic (non-harmonic) functions by means of full polynomials. In particular, from~\eqref{Babu Melenk estimate},
one deduces that, on convex elements, a better convergence rate is achieved (i.e., harmonic functions can be better approximated by polynomials than generic functions, even by considering harmonic polynomials only),
while on non-convex elements, the rate of approximation gets worse (i.e., the best approximation of harmonic functions by full polynomials fails to be achieved with harmonic polynomials).
\end{remark}

Next, we have to bound the non-conformity term $\mathcal N_n (\u,\vn)$ introduced in~\eqref{non conformity term}. To this purpose, we use tools of non-conforming methods and $\h\p$-analysis.

Firstly, we define $\Omegaext$ as an extension 
of the domain $\Omega$, subordinated to polygonal decompositions.
More precisely, let $\tautilden$ be a triangulation of $\Omega$ which
is given by the union of local triangulations $\tautilden(\E)$ over
each polygon $\E \in \taun$ ($\tautilden$ is nested in $\taun$);
such local triangulations are obtained by connecting the vertices of $\E$ to the center of the ball with respect to which $\E$ is star-shaped, see assumption (\textbf{D2}).
With each triangle $\T \in \tautilden$, we associate $\Q(\T)$, a
parallelogram obtained by reflecting $\T$ with respect to the midpoint
of one of its edges, which is arbitrarily fixed.
Then, we set
\begin{equation} \label{inflated domain}
\Omegaext := \bigcup_{\T\in \tautilden} \Q(\T).
\end{equation}
Notice that $\Omegaext$ could coincide with $\Omega$.

The following lemma provides an upper bound for the non-conformity term $\mathcal N_n(u,\vn)$. 
\begin{lem} \label{lemma approximation non-conforming term}
Assume that (\textbf{D1})-(\textbf{D4}) are satisfied. Then, for all $s\ge 1$ and for all $\u \in H^{s+1}(\Omegaext)$, the following bound holds true:
\begin{equation*} 
\vert \mathcal N_n(\u, \vn) \vert \le c\, d^s\frac{\h^{\min(s,\p)}}{\p^{s}} \Vert \u \Vert_{s+1,\Omegaext} \vert \vn \vert_{1,\taun}\quad \forall \vn \in \Vnpz,
\end{equation*}
where $c$ is a positive constant depending only on $\rho_1$, $\rho_2$,
$\rho_3$, and $\Lambda$, and $d$ is a positive constant. 
\end{lem}
\begin{proof}
Without loss of generality, let us assume that 
$h=1$, so that $\rho_3^{-1}\le h_\E\le 1$ for all $\E \in \taun$, due to the assumption (\textbf{D4}); the general assertion follows from a scaling argument.

First, we observe that, for all $\vn \in \Vnpz$, the definition of non-conforming spaces and basic properties of orthogonal projectors yield
\begin{equation} \label{initial bound non conformity}
\begin{split}
\vert \mathcal N_n (\u,\vn) \vert 	& = \left\vert \sum_{\e \in \mathcal E_n} \int_\e \nabla \u \cdot \llbracket\vn\rrbracket \, \ds \right\vert=\left\vert  \sum_{\e \in \mathcal E_n} \int_\e (\nabla \u - \Pie(\nabla\u)) \cdot(\llbracket \vn \rrbracket - \Pie \llbracket \vn \rrbracket) \, \ds  \right\vert\\
& \le \sum_{\e \in \mathcal E_n} \left\Vert \nabla \u - \Pie (\nabla \u) \right \Vert_{0,\e} \left\Vert \llbracket \vn \rrbracket - \Pie \llbracket \vn \rrbracket \right\Vert_{0,\e},
\end{split}
\end{equation}
where we have denoted by $\Pie$, with an abuse of notation, the $L^2$
projector onto the vectorial polynomial spaces of degree $\p-1$ on $e$.

In order to estimate the first term on the right-hand side, we proceed as follows. 
Let us consider $\tautilden$, the union of the local triangulations
$\tautilden(\E)$ of each $\E \in \taun$ defined as above.
%
The triangulation $\tautilden$ has the property that each 
$T\in \tautilden$ 
is star-shaped with respect to a ball of radius greater than or equal
to $\rho_4 h_\T$, where $\rho_4$ is a positive constant
and $h_\T$ is the diameter of the triangle $\T$, see \cite{VEMchileans}. 
Let now $\e \in \mathcal E_n$ 
be fixed and $\E \in \taun$ be a polygon with $\e \in \mathcal E^\E$. Then,
\[
\Vert \nabla \u - \Pie (\nabla \u) \Vert_{0,\e} \le \Vert \nabla \u - \PiT (\nabla \u) \Vert_{0,\e},
\]
where $\PiT$ is the $L^2$ projector onto the space of vectorial polynomials of
degree at most $\p-1$ over $\T$, and $\T$ is the triangle in
$\tautilden (\E)$ with $\e \subset \partial \T$ (this inequality holds true because the restriction of $\PiT (\nabla \u)$ to $e$ is a vectorial polynomial of degree $\p-1$).

For any $v\in H^2(\T)$, due to \cite[Theorem 3.1]{ChernovoptimalconvergencetracepolynomialsL2projectiononasimplex}, we have
\begin{equation} \label{Chernov bound}
\Vert \nabla v - \PiT (\nabla v) \Vert_{0,\e} \le \frac{\sqrt 5 + 1}{\sqrt 2} \p^{-\frac{1}{2}} \vert \nabla v \vert_{1,\T}.
\end{equation}
Using that $\PiT \nabla\qp = \nabla\qp$ for all $\qp \in \mathbb{P}_{\p}(\T)$, owing to \eqref{Chernov bound}, we get
\[
\Vert \nabla \u - \PiT (\nabla \u) \Vert_{0,\e} 
= \Vert (\nabla (\u- \qp)) - \PiT (\nabla (\u - \qp)) \Vert_{0,\e}
\lesssim \p^{-\frac{1}{2}} \vert \nabla (\u - \qp) \vert_{1,\T}.
\]
Applying now standard $\h\p$-polynomial approximation results, see e.g. \cite[Lemma 5.1]{hpVEMbasic}, we obtain for every $\qp \in \mathbb P_{\p} (\T)$,
\begin{equation} \label{equation stellina}
\vert \nabla (\u - \qpmu) \vert_{1,T} \lesssim d^s \p^{-s+1} \vert \nabla \u \vert_{s,\Q(\T)},
\end{equation}
where $d$ is a positive constant and $\Q(\T)$ is the
parallelogram given by the union of $\T$ and its reflection defined above.

Moving to the second term in \eqref{initial bound non conformity}, 
assuming that $e=\partial \T^-\cap\partial\T^+$, where $\T^{\pm} \in \tautilden$ and $\T^{\pm} \subset \E^{\pm}$,
we have
\[
\left\Vert \llbracket \vn \rrbracket - \Pie \llbracket \vn \rrbracket
\right \Vert_{0,\e}  \le \Vert v_{n|_{\T^+}} - 
\Pi_{p-1}^{0,\T^+} 
v_{n|_{\T^+}} \Vert_{0,\e} + \Vert v_{n|_{\T^-}} - 
\Pi_{p-1}^{0,\T^-} 
v_{n|_{\T^-}}\Vert_{0,\e}.
\]
Then, applying once again \cite[Theorem 3.1]{ChernovoptimalconvergencetracepolynomialsL2projectiononasimplex}, we deduce
\[
\left\Vert \llbracket \vn \rrbracket - \Pie \llbracket \vn \rrbracket
\right \Vert_{0,\e} \lesssim \p^{-\frac{1}{2}} \left( \vert v_{n|_{\T^+}}
\vert_{1,\T^+} 
+ \vert v_{n|_{\T^-}}
\vert_{1,\T^-} 
\right).
\]
By combining the bounds of the two terms on the right-hand side of
\eqref{initial bound non conformity} and the definition of the
extended 
domain $\Omegaext$ in \eqref{inflated domain}, we get the
assertion. 
\end{proof}

We are now ready to state the main $\h$- and $\p$-error estimate result.
\begin{thm} \label{theorem h and p VEM}
Let $\{\taun\}_{n\in \mathbb N}$ be a sequence of polygonal decompositions satisfying (\textbf{D1})-(\textbf{D4}). Let $\u$ and $\un$ be the solutions to \eqref{Laplace problem weak formulation} and \eqref{VEM}, respectively;
we assume that $\u$ is the restriction to $\Omega$ of an
$H^{s+1}$, $s\ge1$, function
(which we still denote $\u$, with a slight abuse of notation), 
over $\Omegaext$, where $\Omegaext$ is defined in \eqref{inflated domain}.
Then, the following a priori $\h$- and $\p$-error estimate holds true:
\[
\vert \u - \un \vert_{1,\taun} \le c\, d^s\,
\frac{\alpha^*(\p)}{\alpha_*(\p)}\, \h^{\min(s,\p)} \left\{  \left( \frac{\log(\p)}{\p}\right)^{\min_{\E \in \taun}(\lambda_\E)\,s} +\p^{-s} \right\} \Vert \u \Vert_{s+1,\Omegaext},
\]
where
$c$ is a positive constant depending only on $\rho_1$, $\rho_2$, $\rho_3$, and $\Lambda$,
d is a positive constant,
$\lambda_\E \, \pi$ denotes the smallest exterior angle of $\E$ for
each $\E\in \taun$, and $\frac{\alpha^*(\p)}{\alpha_*(\p)}$ is the pollution factor appearing
in \eqref{abstract bound}, which is related to the choice of the stabilization.
\end{thm}
\begin{proof}
It is enough to combine Theorem \ref{theorem abstract error analysis} with Lemmata \ref{lemma approximation harmonic polynomials} and \ref{lemma approximation non-conforming term}.
\end{proof}

Assuming, moreover, that $\u$, the solution to the problem \eqref{Laplace problem weak formulation}, is the restriction to $\Omega$ of an analytic function defined over $\Omegaext$, where $\Omegaext$ was introduced in \eqref{inflated domain},
it is possible to prove the following result.
\begin{thm} \label{theorem exponential convergence analytic solution}
Let (\textbf{D1})-(\textbf{D4}) be valid and assume that $\u$, the solution to the problem \eqref{Laplace problem weak formulation}, is the restriction to $\Omega$ of an analytic function defined over $\Omegaext$, given in \eqref{inflated domain}.
Then, the following a priori $\p$-error estimate holds true:
\[
\vert \u - \un \vert_ {1,\taun} \le c \exp{(-b\,\p)},
\]
for some positive constants $b$ and $c$, depending again only on $\rho_1$, $\rho_2$, $\rho_3$, $\rho_4$ and $\Lambda$.
\end{thm}
\begin{proof}
It is enough to use Theorem \ref{theorem h and p VEM}, in combination with the tools employed in \cite[Theorem~5.2]{hpVEMbasic}.
\end{proof}

\begin{remark}
The construction involving the collection of parallelograms in~\eqref{inflated domain} is instrumental for proving Theorem~\ref{theorem exponential convergence analytic solution}.
In order to derive the bound of Theorem~\ref{theorem exponential convergence analytic solution} from that of Theorem~\ref{theorem h and p VEM},
one needs to know the explicit dependence on~$s$ of the constant in the bound of Theorem~\ref{theorem h and p VEM}.
This comes at the price of involving the extended domain~$\Omegaext$.
If one were interested in approximating solutions with finite Sobolev regularity, then there would be no need of employing the construction with the parellelograms $Q(T)$.
In particular, equation \eqref{equation stellina} would be valid also
with the norm over the triangle $T$, instead of over $Q(T)$,  on the right-hand side. 
As a consequence, the bounds in Lemma \ref{lemma approximation non-conforming term} and in Theorem \ref{theorem h and p VEM} would
be valid also with the norm of $u$ over $\Omega$, instead of over $\Omegaext$,  on the right-hand sides.
See \cite{hpVEMbasic} for additional details on the $\h\p$-version in the case of the standard VEM setting.
\end{remark}

\subsection{Error estimates in the $L^2$ norm} \label{subsection L2 error estimate}
This section is devoted to bound the $L^2$ error of method \eqref{VEM}
in terms of the energy error and best approximation error with respect to piecewise discontinuous harmonic polynomials. 
For simplicity, we restrict ourselves to the case of convex domains and of sequences of convex polygons; the non-convex case is discussed in Remark~\ref{rem:nonconvex}.

To this purpose, we firstly recall the definition of non-conforming VE spaces introduced in~\cite{nonconformingVEMbasic} for the approximation of the Poisson problem,
and then we prove $\h\p$-best approximation estimates by functions in those spaces. The obtained results will be instrumental for proving $L^2$ error estimates for method \eqref{VEM}.
Throughout the whole section, we assume that $p$, the parameter used in the definition of local spaces \eqref{local VE space}, is equal to $\k$, the non-conformity parameter, appearing in the definition of the global non-conforming VE space \eqref{non conforming space non homogeneous}.

Let $\E \in \taun$. We define, for $\p\in \mathbb N$ arbitrary,
\begin{equation*} 
\VE := \left\{ \vn \in H^1(\E) \, \mid \, \Delta \vn \in \mathbb P_{\p-2}(\E),\, (\nabla \vn \cdot \n_{\E}) _{|_\e} \in \mathbb P_{\p-1}(\e) \ \forall \e \in \mathcal E_n \right\}.
\end{equation*}
It is proved in \cite[Lemma 3.1]{nonconformingVEMbasic} that the following is a set of degrees of freedom for the space $\VE$. Given $\vn \in \VE$, we associate the edge moments defined in \eqref{local dofs}
\begin{equation} \label{edge moments}
\frac{1}{h_e} \int_\e \vn m_{\alpha}^\e \, \ds, \quad \forall \alpha=0,\dots, \p-1, \, \forall \e \in \mathcal E^\E ,
\end{equation}
plus the bulk moments of the form 
\begin{equation} \label{bulk moments}
\frac{1}{\vert \E \vert} \int_\E \vn m_{\boldalpha} \, \dx, \quad \forall \vert \boldalpha \vert=0,\dots, \p-2 ,
\end{equation}
where $\{m_{\boldalpha}\}_{\vert \boldalpha \vert=0}^{\p-2}$ is {\it{any}} basis of $\mathbb P_{\p-2}(\E)$.

For all $\g \in H^{1/2}(\partial \Omega)$, 
the global non-conforming spaces in~\eqref{non conforming space non
  homogeneous}
are defined as in the harmonic case: 
\begin{equation} \label{global Manzini}
\Vnkgc := \left\{ \vn \in H^{1,\nc}_\g(\taun,\k) \, \mid \, {\vn}_{|_{\E}} \in \VE \ \forall \E \in \taun \right\}.
\end{equation}
The set of global degrees of freedom is obtained by a standard non-conforming coupling of the local counterparts.
The precise treatment of Dirichlet boundary conditions should be dealt with as in Remark \ref{remark how to deal with Dirichlet boundary conditions}.

We want to show that, in the $H^1$ seminorm, the error between a regular target function and its interpolant in the space $\Vnpgc$ defined in \eqref{global Manzini} can be bounded by the best approximation error in the space of piecewise discontinuous polynomials of degree at most $\p$.
Notice that neither the convexity of $\Omega$ nor the convexity of the elements are needed here.

\begin{prop} \label{theorem best error Manzini space}
Given $\g \in H^{1/2}(\partial \Omega)$, let $\psi \in \Vg$, where $\Vg$ is defined in \eqref{basic notation weak formulation}.
For every polygonal partition $\taun$ of $\Omega$, there exists $\psiI \in \Vnpgc$, with $\Vnpgc$ given in \eqref{global Manzini}, such that
\begin{equation*} 
\vert \psi - \psiI \vert_{1,\taun} \le 2 \vert \psi -\qp \vert_{1,\taun} \quad \forall q_p \in \mathcal{S}^{p,-1}(\taun),
\end{equation*}
where $\mathcal{S}^{p,-1}(\taun)$ is the space of piecewise discontinuous polynomials, that is,
\begin{align} \label{space Sp-1}
\mathcal{S}^{p,-1}(\taun):= \{ q \in L^2(\Omega): \, q_{|_\E} \in \mathbb P_\p(\E) \, \forall \E \in \taun \}.
\end{align}
\end{prop}
\begin{proof}
The proof follows the lines of that of
  Proposition~\ref{theorem best approximation VE functions}.
Given $\psi \in \Vg$, we define $\psiI\in \Vnpgc$ by imposing its degrees of freedom as follows:
\begin{equation} \label{definition Manzini interpolant}
\begin{split}
\frac{1}{h_e} \int_\e (\psiI - \psi) \q_{\p-1}^\e \, \ds&= 0\quad \forall \q_{\p-1}^\e \in \mathbb P_{\p-1}(e),\, \forall \e \in \mathcal E^\E ,\, \forall \E \in \taun.\\
\frac{1}{\vert \E \vert} \int_\E (\psiI - \psi) \q_{\p-2} \, \dx&= 0\quad \forall \q_{\p-2} \in \mathbb P_{\p-2}(\E),\, \forall \E \in \taun,\\
\end{split}
\end{equation}
It is important to note that, since the degrees of freedom \eqref{edge moments} and \eqref{bulk moments} are unisolvent for the space $\Vnpgc$, the interpolant $\psiI$ is defined in a unique way.
Having this, we observe that, for all $\E \in \taun$,
\begin{equation*} 
\vert \psi - \psiI \vert_{1,\E} \le \vert \psi - \qp \vert _{1,\E} + \vert \psiI - \qp \vert_{1,\E}\quad \forall \qp \in \mathbb P_\p(\E).
\end{equation*}
We focus on the second term. By integration by parts we get
\[
\begin{split}
\vert \psiI - \qp \vert^2_{1,\E}	
& = \int_\E - \Delta (\psiI - \qp) (\psiI - \qp) \, \dx +
\int_{\partial \E} 
\nabla(\psiI-\qp)\cdot\n_K
\, (\psiI-\qp) \, \ds \\
& = \int_\E - \Delta (\psiI - \qp) (\psi - \qp) \, \dx+ \int_{\partial
  \E} 
\nabla(\psiI-\qp)\cdot\n_K
\, (\psi-\qp) \, \ds,\\
\end{split}
\]
where, in the last identity, the definition of the non-conforming space $\Vnpgc$, given in \eqref{global Manzini} and the definition \eqref{definition Manzini interpolant} of $\psiI$ via the degrees of freedom were used.
Integrating by parts back, we obtain
\[
\vert \psiI - \qp \vert^2_{1,\E} = \int_\E \nabla(\psiI-\qp) \cdot \nabla (\psi - \qp) \, \dx \le \vert \psiI - \qp \vert_{1,\E} \vert \psi - \qp\vert_{1,\E}.
\]
This concludes the proof.	
\end{proof}

We are now ready to prove a bound of the $L^2$ error of the method. We will assume henceforth that $\Omega$ is a convex domain split into a collection of convex polygons.
An analogous analysis could be performed in the non-convex case, and slightly worse error estimates could be proven, see Remark~\ref{rem:nonconvex}.
Nonetheless, here we stick to the convex setting, since we deem it is clearer.

\begin{thm} \label{theorem L2 estimates}
Let $\Omega$ be a polygonal convex domain and let $\{\taun\}_{n\in
  \mathbb N}$ be a sequence of decompositions into convex polygons satisfying (\textbf{D1})-(\textbf{D4}).
Let $\u$ and $\un$ be the solutions to \eqref{Laplace problem weak formulation} and \eqref{VEM}, respectively; we assume that  $\u$ is
the restriction to $\Omega$ of a $H^{s+1}$, $s\ge1$, function
(which we still denote, with a slight abuse of notation, $\u$) over $\Omegaext$, where $\Omegaext$ is defined in \eqref{inflated domain}.
Then, 
\begin{equation*} 
\begin{split}
&\Vert \u - \un \Vert_{0,\Omega} \le c\, 
 \left\{\frac{\h^{\min(s,\p)+1}}{\p^{s+1}} \Vert \u \Vert_{s+1,\Omegaext} \right.\\
&\qquad\qquad \left.+ \max\left( \frac{\h}{\p}, \h \,\alpha^*(\p) \left(
    \frac{\log(\p)}{\p}  \right)^{\max_{\E \in \taun}\lambda_\E}
\right) \left( \vert \u - \un \vert_{1,\taun} + \inf_{\qDeltap\in \mathcal{S}^{p,\Delta,-1}(\taun)}\vert \u - \qDeltap \vert_{1,\taun} \right)\right\},\\
\end{split}
\end{equation*}
where $c$ is a positive constant depending only on $\rho_1$, $\rho_2$, $\rho_3$, $\rho_4$ and $\Lambda$,
$\alpha^*(\p)$ is the ``upper'' stability constant appearing in \eqref{stability}, 
$\mathcal{S}^{p,\Delta,-1}(\taun)$ is 
defined in~\eqref{space SpDelta-1}, 
and 
$\lambda_\E \pi$ denotes the smallest exterior angle of $\E$ for each $\E \in \taun$.
\end{thm}
\begin{proof}
We consider the following dual problem: Find $\psi \in H^1(\Omega)$  
such that
\begin{align}\label{dual problem}
\left\{
\begin{alignedat}{2}
-\Delta \psi & = \u - \un && \quad \text{in }\Omega \\
\psi & = 0 && \quad \text{on } \partial \Omega.
\end{alignedat}
\right.
\end{align}
Standard stability and a priori regularity theory implies
that $\psi\in H^2(\Omega)$ and
\begin{equation} \label{a priori bound}
\Vert \psi \Vert_{2,\Omega} 
\lesssim  \Vert \u - \un \Vert_{0,\Omega},
\end{equation}
where the hidden constant depends only on the domain
$\Omega$, see e.g. \cite[Theorem 3.2.1.2]{grisvard}.
\medskip
	
Using \eqref{dual problem} and \eqref{non conformity term}, and taking into account that $\u - \un\in H^{1,\nc}_{0} (\taun,\p)$,
we obtain
the following equivalent 
expression for the $L^2$ error:
\begin{align}
\Vert \u - \un \Vert^2_{0,\Omega} &= \sum_{\E \in \taun} \int_\E
                                    (-\Delta \psi) (\u - \un) \, \dx =
                                    \sum_{\E \in \taun} \left\{
                                    \int_\E \nabla \psi \cdot \nabla
                                    (\u - \un) \, \dx - \int_{\partial
                                    \E} 
\nabla \psi\cdot\n_K
\, (\u - \un) \, \ds  \right\}\notag\\
& = \sum_{\E \in \taun} \aE(\psi - \psiI, \u - \un) + \sum_{\E \in \taun} \aE(\psiI, \u - \un) - \mathcal N_n(\psi, \u - \un) \notag\\
&=: T_1+T_2+T_3,  \label{initial equivalence L2 error}
\end{align}
where 
$\psiI$ is the (unique) function in $\Vnpzc$, the enlarged space of functions with zero Dirichlet traces introduced in \eqref{global Manzini}, defined from $\psi$ via \eqref{definition Manzini interpolant}; in particular, $\psiI$ is not piecewise harmonic, in general. 
	
We begin by bounding term $T_1$. Owing to the Cauchy-Schwarz inequality and Proposition~\ref{theorem best error Manzini space}, we have
\[
|T_1| \le \vert \psi - \psiI \vert_{1,\taun} \vert \u - \un \vert_{1,\taun} \le 2 \vert \psi - \qp \vert_{1,\taun} \vert \u - \un \vert_{1,\taun} \quad \forall \qp \in \mathcal{S}^{p,-1}(\taun),
\]
where $\mathcal{S}^{p,-1}(\taun)$ is the space of piecewise
discontinuous polynomials introduced in \eqref{space Sp-1}.
By taking $\qp$ equal to the best approximation of $\psi$ in $\mathcal{S}^{p,-1}(\taun)$ 
and using \cite[Lemma 4.2]{hpVEMbasic}, together with 
\eqref{a priori bound}, we have
\[
|T_1| \lesssim \frac{\h}{\p} \Vert \psi \Vert_{2,\Omega} \vert \u - \un \vert_{1,\taun} \lesssim \frac{\h}{\p} \Vert \u - \un \Vert_{0,\Omega} \vert \u-\un \vert_{1,\taun}.
\]

Next, we focus on term $T_3$ on the right-hand side of \eqref{initial equivalence L2 error}. Following the same steps as in the proof of Lemma \ref{lemma approximation non-conforming term}, we obtain
\begin{equation*} 
\begin{split}
|T_3|  = |\mathcal N_n(\psi, \u - \un)| 	 
\le \sum_{\e \in \mathcal E_n} \left\Vert \nabla \psi - \Pie (\nabla \psi) \right \Vert_{0,\e} \left\Vert \llbracket \u-\un \rrbracket - \Pie \llbracket \u-\un \rrbracket \right\Vert_{0,\e},\\
\end{split}
\end{equation*}
where $\Pie$ denotes here again, with an abuse of notation, the $L^2$
projector onto vectorial polynomial spaces.
Applying \cite[Theorem
3.1]{ChernovoptimalconvergencetracepolynomialsL2projectiononasimplex}
and  \cite[Lemma 4.1]{hpVEMbasic} similarly as in the proof of
Lemma~\ref{lemma approximation non-conforming term}, together with 
\eqref{a priori bound} ($\vert \nabla \psi \vert_{1,\E} \le \Vert \psi
\Vert_{2,\E}$),
 we get
\begin{equation*} 
\begin{split}
|T_3| 
\lesssim \frac{\h}{\p} \Vert \u - \un \Vert_{0,\Omega} \vert \u - \un \vert_ {1,\taun}.\\
\end{split}
\end{equation*}
Finally, we study term $T_2$ on the right-hand side of \eqref{initial equivalence L2 error}, which can be split as
\begin{equation} \label{initial equivalence on B}
T_2 = \sum_{\E \in \taun} \aE(\psiI, \u - \un) = \sum_{\E \in \taun} \aE(\u, \psiI) - \sum_{\E \in \taun} \aE(\un, \psiI) =: T_4+T_5.
\end{equation}
The first term $T_5$ is related to 
the non-conformity of the discretization spaces, 
whereas the second term $E$ reflects the fact that method \eqref{VEM} does not employ the original bilinear form.

We start to bound term $T_4$. Using computations analogous to those in the proof of Lemma \ref{lemma approximation non-conforming term}, it is possible to deduce
\[
\begin{split}
|T_4|	& = \bigg| \sum_{\E \in \taun} \aE(\u,\psiI) \bigg| = \bigg|
\sum_{\E \in \taun} \int_{\partial \E} 
\nabla\u\cdot\n_K 
\, \psiI \, \ds \bigg| = |\mathcal N_n(\u,\psiI) | =  |\mathcal N_n(\u,\psiI - \psi)| \\
&\le \sum_{\e \in \mathcal E_n} \left\Vert \nabla u - \Pie (\nabla u) \right \Vert_{0,\e} \left\Vert \llbracket \psiI-\psi \rrbracket - \Pie \llbracket \psiI-\psi \rrbracket \right\Vert_{0,\e},
\end{split}
\]
where in the fourth identity we
used the fact that $\mathcal N_n(\u,\psi)=0$, which holds since $\u$ and $\psi$ are sufficiently regular, and in the last step we used~\eqref{initial bound non conformity}. 
Again, $\Pie$ has to be understood as the $L^2$ projection onto the vectorial polynomial spaces of degree at most $\p-1$ on $\e$.
Applying \cite[Theorem~3.1]{ChernovoptimalconvergencetracepolynomialsL2projectiononasimplex},  Proposition~\ref{theorem best error Manzini space}, \cite[Lemma 4.2]{hpVEMbasic}, and finally \eqref{a priori bound}, leads to
\[
\begin{split}
|T_4| \lesssim \p^{-1} \vert \nabla (\u - \Pinabla \u) \vert_{1,\taun} \vert \psi-\psiI \vert_{1,\taun}
\lesssim \frac{\h^{\min(s,\p)}}{\p^s} \Vert \u \Vert_{s+1,\Omegaext} \frac{\h}{\p} \Vert \u - \un \Vert_{0,\Omega},
\end{split}
\]
where we recall that $\Omegaext$ is defined in \eqref{inflated domain} and where $\Pinabla$ is any piecewise energy projector from $H^1(\E)$ into $\mathbb P_\p(\E)$, for all $\E\in \taun$.
	
Finally, it remains to treat term $T_5$ on the right-hand side of
\eqref{initial equivalence on B}. To this purpose, we consider the
following splittings of $\psi$ and $\psiI$.
Firstly, we split $\psi$ into $\psi=\psi^1 + \psi^2$, where $\psi^1$ and $\psi^2$ are, element by element, solutions to the local problems
\begin{align} \label{decomposition psi}
\left\{
\begin{alignedat}{2}
-\Delta \psi^1 & = -\Delta \psi && \quad \text{in }\E \\
\psi^1 & = 0 && \quad \text{on } \partial \E,
\end{alignedat}
\right.  \quad \quad
\left\{
\begin{alignedat}{2}
-\Delta \psi^2 & = 0 && \quad \text{in }\E \\
\psi^2 & = \psi && \quad \text{on } \partial \E
\end{alignedat}
\right.
\end{align}
for all $\E \in \taun$. Using \eqref{dual problem}, we can also observe that $\psi^2-\psi$ solve the local problems
\begin{align*} 
\left\{
\begin{alignedat}{2}
-\Delta (\psi-\psi^2) & = \u-\un && \quad \text{in }\E \\
\psi-\psi^2 & = 0 && \quad \text{on } \partial \E,
\end{alignedat}
\right.
\end{align*}
Then, (local) standard a priori regularity theory
and, afterwards, summation over all elements $\E \in \taun$ imply the global bound
\begin{equation} \label{a priori estimate psi2-psi}
\Vert \psi^2-\psi \Vert_{2,\taun} \lesssim \Vert \u - \un \Vert_{0,\Omega},
\end{equation}
where the broken norm $\Vert \cdot \Vert_{2,\taun}$ is defined in \eqref{broken s Sobolev}.
With the triangle inequality, \eqref{a priori bound}, and \eqref{a priori estimate psi2-psi}, we get 
\begin{equation} \label{important a priori bound}
\Vert \psi^2 \Vert _{2, \taun} \le \Vert \psi - \psi^2 \Vert_{2,\taun} + \Vert \psi \Vert_{2,\Omega} \lesssim \Vert \u - \un \Vert_{0,\Omega}.
\end{equation}
Secondly, we split $\psiI \in \Vnpzc$ into $\psiI = \psiI^1+\psiI^2$. We define $\psiI^2$ as the unique element in $\Vnpz$ introduced in \eqref{global non conforming harmonic VES}, which satisfies
\begin{equation} \label{dofs psiI2}
\frac{1}{h_e} \int_\e \psiI^2\qpmu^\e \, \ds = \frac{1}{h_e} \int_\e \psiI \qpmu^\e \, \ds \quad \forall \qpmu^\e \in \mathbb P_{p-1}(e), \, \forall e \in \mathcal E_n.
\end{equation}
Existence and uniqueness of $\psiI^2$ follow from the fact that we are defining $\psiI^2$ via unisolvent degrees of freedom for the space $\Vnpz$.
Owing to \eqref{dofs psiI2}, the definition of $\psiI$ in \eqref{definition Manzini interpolant}, and \eqref{decomposition psi}, we deduce
\[
\frac{1}{h_e} \int_\e \psiI^2\qpmu^\e \, \ds= \frac{1}{h_e} \int_\e \psiI\qpmu^\e \, \ds = \frac{1}{h_e} \int_\e \psi\qpmu^\e \, \ds= \frac{1}{h_e} \int_\e \psi^2\qpmu^\e \, \ds \quad \forall \qpmu^\e \in \mathbb P_{\p-1}(\e),\, \forall \e \in \mathcal E_n.
\]
This entails that $\psiI^2$ approximates $\psi^2$ in the sense of Proposition~\ref{theorem best approximation VE functions}.
Having this, 
the function $\psiI^1=\psiI- \psiI^2\in \Vnpzc$ satisfies
\begin{align*} 
\left\{
\begin{alignedat}{2}
\frac{1}{\vert \e \vert} \int_\e \psiI^1\qpmu^\e \, \ds & = 0 && \quad \forall \qpmu^\e \in \mathbb P_{\p-1}(\e),\, \forall \e \in \mathcal E^\E, \, \forall \E \in \taun, \\
\frac{1}{\vert \E \vert} \int_\E \psiI^1\qpmd \, \dx & = \frac{1}{\vert \E \vert} \int_\E (\psiI- \psiI^2)\qpmd \, \dx && \quad \forall \qpmd \in \mathbb P_{\p-2}(\E) , \, \forall \E \in \taun.
\end{alignedat}
\right.
\end{align*}
Moreover, since $\un \in \Vnkg$, $\psiI^1$ has the essential feature that it satisfies
\begin{equation} \label{essential feature psiI1}
\aE(\un,\psiI^1) = \int_\E (\underbrace{-\Delta \un}_{=0}) \psiI^1 \, \dx + \underbrace{\int_{\partial \E} (\nabla \un \cdot \n_{\E}) \psiI^1 \, \ds}_{=0} = 0.
\end{equation}
We have now all the tools for bounding term $T_5$. Using \eqref{essential feature psiI1}, \eqref{VEM}, and \eqref{consistency}, we get
\begin{equation*} 
\begin{split}
|T_5|	& = \bigg| \sum_{\E \in \taun} \aE(\un,\psiI^2) \bigg| = \bigg| \sum_{\E \in \taun} \left\{\anE(\un,\psiI^2) - \aE(\un, \psiI^2) \right\} \bigg|\\
& = \bigg| \sum_{\E\in \taun} \left\{ \anE(\un-\qDeltap, \psiI^2-\qtildeDeltap) - \aE(\un-\qDeltap, \psiI^2-\qtildeDeltap) \right\} \bigg|
\quad \forall \qDeltap,\, \qtildeDeltap \in \mathcal{S}^{p,\Delta,-1}(\taun),
\end{split}
\end{equation*}
where we recall that $\mathcal{S}^{p,\Delta,-1}(\taun)$ is defined in \eqref{space SpDelta-1}.
It is important to highlight that it is in fact a key point of the error analysis to have piecewise harmonic functions in both entries of the discrete bilinear form.
By applying the continuity property \eqref{continuity_an} and the
Cauchy-Schwarz inequality, then the triangle inequality and Proposition~\ref{theorem best approximation VE functions}, we deduce
\begin{align*}
|T_5| &\lesssim \alpha^*(\p) \vert \un - \qDeltap \vert_{1,\taun} \vert \psiI^2 - \qtildeDeltap \vert_{1,\taun} \\
& \le \alpha^*(\p) (\vert \u - \un\vert_{1,\taun} + \vert \u - \qDeltap \vert_{1,\taun}) (\vert \psi^2 - \psiI^2 \vert_{1,\taun} + \vert \psi^2 - \qtildeDeltap \vert_{1,\taun}) \\
&\lesssim \alpha^*(\p) (\vert \u - \un\vert_{1,\taun} + \vert \u - \qDeltap \vert_{1,\taun}) \vert \psi^2 - \qtildeDeltap \vert_{1,\taun}.
\end{align*}
Thanks to Lemma \ref{lemma approximation harmonic polynomials}
(here, $s=1$) and the bound~\eqref{important a priori bound}, we have
\[
\begin{split}
|T_5| &\lesssim \alpha^*(\p) (\vert \u - \un \vert_{1,\taun} + \vert \u
- \qDeltap\vert_{1,\taun})\,\h \left(
  \frac{\log(\p)}{\p}\right)^{\min_{\E \in \taun} \lambda_\E} \left(
  \sum_{\E \in \taun} \Vert \psi^2 \Vert^2_{2,\E}
\right)^{\frac{1}{2}}\\
&\lesssim \alpha^*(\p) (\vert \u - \un \vert_{1,\taun} + \vert \u - \qDeltap \vert_{1,\taun}) \h \left( \frac{\log(\p)}{\p} \right)^{\min_{\E \in \taun} \lambda_\E} \Vert \u - \un \Vert_{0,\Omega},
\end{split}
\]
where we recall that,  for any $\E \in \taun$, $\lambda_\E \, \pi$ denotes the smallest exterior angle of $\E$.

By combining the estimates on all the terms $T_1$ to $T_5$, we get the assertion.
\end{proof}

\begin{remark}\label{rem:nonconvex}
As already highlighted, the case of non-convex $\Omega$ can be treated analogously.
More precisely, given $\omega$ the largest reentrant angle 
of $\Omega$, the
solution of \eqref{Laplace problem weak formulation} belongs to 
$H^{1+t}(\Omega)$, with $t=\frac{\pi}{\omega}-\varepsilon$
for all $\varepsilon >0$ arbitrarily small.
Standard stability and a priori regularity theory, see \cite[Theorem 2.1]{babuvska1988regularity}, gives
\[
\Vert \psi \Vert_{1+t,\Omega} 
\le c \Vert \u - \un \Vert_{0,\Omega}
\]
for some positive constant $c$ depending only on the domain $\Omega$.
An analogous bound is valid for the 
counterpart of \eqref{a priori estimate psi2-psi} in the non-convex case.
Having this, a straightforward modification of the proof of Theorem \ref{theorem L2 estimates} leads to the $\h$- and $\p$-error bounds
\[
\begin{split}
\Vert \u - \un \Vert_{0,\Omega} \le & \left\{c\, 
\frac{\h^{\min(s,\p)+t}}{\p^{s+t}}
\Vert \u \Vert_{s+1,\Omegaext} \right. \\
&+ \max\left( \left(\frac{\h}{\p}\right)^{t}, \,\h^t \,\alpha^*(\p) \left( \frac{\log(\p)}{\p}  \right)^{\max_{\E \in \taun}(\lambda_\E) \,t} \right)\\
& \left.\quad \quad \quad \cdot\left( \vert \u - \un \vert_{1,\taun} + \inf_{\qDeltap\in \mathcal{S}^{p,\Delta,-1}(\taun)}\vert \u - \qDeltap \vert_{1,\taun} \right) \right\},\\
\end{split}
\]
where $c$ is a positive constant depending only on the constants $\rho_1$, $\rho_2$, $\rho_3$, $\rho_4$, and $\Lambda$ appearing in (\textbf{D1})-(\textbf{D4})
and in the proof of Lemma \ref{lemma approximation non-conforming term}.

The presence of non-convex polygons in the decomposition~$\taun$ leads to a possible additional loss in the convergence rate in $\p$ of the $L^2$ error, which will depend on the largest interior and exterior angles of the polygons.
\end{remark}

\subsection{Hints for the extension to the 3D case} \label{subsection extension 3D}
The aim of this section is to give a hint concerning the extension of what we have presented and discussed so far to the three dimensional case.

Concerning the definition of local harmonic VE spaces, one mimics the
strategy suggested in \cite{nonconformingVEMbasic} and defines, for
every polyhedron $\E$ in $\mathbb R^3$ and any fixed $p \in \mathbb N$,
\begin{equation*} 
\VDeltaE := \left\{ \vn \in H^1(\E) \mid \Delta \vn = 0 \text{ in}\, \E, \, (\nabla \vn \cdot \n_K)_{|_\F} \in \mathbb P_{\p-1}(\F)  \ \forall \F \text{ faces of  } \E  \right\}.
\end{equation*}

We observe that the definition of the local 3D space is a
straightforward extension of its 2D counterpart.
We underline that this is not the case when using {\it conforming} 
VEM.  In that case, typically, one also requires to have a modified
version of the local VE spaces on each face, see
\cite{equivalentprojectorsforVEM}. On the one hand, this
allows 
the construction of continuous functions over 
the boundary of a polyhedron, as well as the construction of
projectors onto proper polynomial spaces;
on the other, it 
complicates the $\p$-analysis of the method.
In the non-conforming framework, however, one does not need to fix any sort of continuity across the interface between faces of a polyhedron and thus
it suffices to impose that normal derivatives are polynomials.

The global 3D non-conforming space is built as in the 2D case. Also,
the degrees of freedom are given by scaled face moments with respect
to 
polynomials up to order $\p-1$.
\medskip

The abstract definition of the 2D local discrete bilinear form in
\eqref{local discrete bilinear form} can also be employed in the 3D case. The (properly scaled) 3D counterpart of the 2D explicit stabilization defined in
\eqref{explicit stabilization} would be
\[
\SE (\un,\vn) = \sum_{\F \text{ faces of $\E$}} \frac{\p}{\h_\F} (\PiF \un, \PiF \vn)_{0,\F},
\]
where, for any face $\F$,
$\PiF$ denotes the $L^2$ projector onto $\mathbb P_{\p-1}(\F)$ of the traces on $\F$ of functions in the 3D VE space. 
Nonetheless, it is not clear whether explicit bounds in terms of $\p$ of the stability constants appearing in \eqref{stability bounds} can be proved for this form.
In fact, in the 2D case, $\h\p$-polynomial inverse estimates in 1D were the key tool for proving Theorem \ref{theorem explicit bounds on explicit stability}.
In the 3D framework, one needs to employ $\h\p$-polynomial inverse estimates on general polygons based on weighted norms. 
We highlight that the approach of \cite[Chapter 3]{hpDEFEM_polygon}, see also~\cite{cangiani2016_ESAIM}, could be followed in order to prove such~$\h\p$-weighted inverse inequalities.
However, as this extension is quite technical, we do not investigate it here.

Independently of the specific choice of the stabilization, provided that it is symmetric and satisfies~\eqref{stability}, the abstract error analysis is dealt with similarly to the 2D case, see Theorem \ref{theorem abstract error analysis}.
The only modification is in the definition of the non-conformity term, which in 3D is defined as
\[
\mathcal N_n (\u,\vv) = \sum_{\F \in \mathcal E_n^3} \int_\F \nabla \u \cdot \llbracket \vv \rrbracket_\F \, \ds
\]
for all conforming functions $\u$ and all non-conforming functions $\vv$, where $\mathcal E_n^3$
denotes the set of faces in the polyhedral decomposition, and $\llbracket \cdot \rrbracket_\F$ is defined as in \eqref{jump operator} in terms of normal derivatives over faces.

The proof of $\h$- and $\p$-error bounds for this non-conforming term follows the same lines as in the 2D case, since \cite[Theorem 3.1]{ChernovoptimalconvergencetracepolynomialsL2projectiononasimplex} holds true on simplices in arbitrary space dimension.
For the best approximation error, one should use
the 3D version of Lemma \ref{lemma approximation harmonic
  polynomials}, which can be found e.g. in \cite[Theorem
3.12]{MoiolaPhDthesis}. 


\section{Numerical results} \label{section numerical results}
We present in Section \ref{section h and p NR} some numerical tests for the $\h$-version and the $\p$-version of the method, validating the theoretical results obtained in Section \ref{section VEM};
we conclude with a discussion and some tests on the $\h\p$-version in Section \ref{subsection NR hp}.
As already mentioned, we refer to Appendix \ref{section appendix implementation details} for details on the implementation of the method.

\subsection{Numerical results: $\h$- and $\p$-version} \label{section h and p NR}
In this section, we present numerical experiments validating the
theoretical error estimates in the $H^1 (\taun)$ ($H^1$, for short) and $L^2$ norms
discussed in Theorems \ref{theorem h and p VEM}, \ref{theorem exponential convergence analytic solution}, and \ref{theorem L2 estimates}.

For the following numerical experiments, we consider boundary value
problems of the form \eqref{Laplace problem strong formulation}, on $\Omega:=(0,1)^2$, 
with known exact solutions given by 
\begin{itemize}
\item $u_1(x,y)=e^x \sin(y)$,
\item $u_2(x,y)=u_2(r,\theta)=r^2 \left(\log(r) \sin(2\theta)+\theta \cos(2\theta) \right)$.
\end{itemize}
We underline that $u_1$ is an analytic function in $\Omega$, whereas $u_2 \in H^{3-\epsilon}(\Omega)$ for every $\epsilon>0$ arbitrarily small;
moreover, $u_2$ represents the natural singular solution at $\mathbf 0
= (0,0)$ of the Poisson problem on a square domain, see
e.g. \cite{babuvska1988regularity}.

We discretize these problems on
sequences of quasi-uniform 
Cartesian meshes and
Voronoi-Lloyd meshes of the type shown in Figure \ref{fig:meshes}, left and
center, respectively. We also test on a problem with exact solution $u_1$
on the domain $\Omega$ given by the union of four Escher horses 
as in Figure \ref{fig:meshes}, right.

\begin{figure}[htbp]
\begin{minipage}{0.335\textwidth} 
\includegraphics[width=\textwidth]{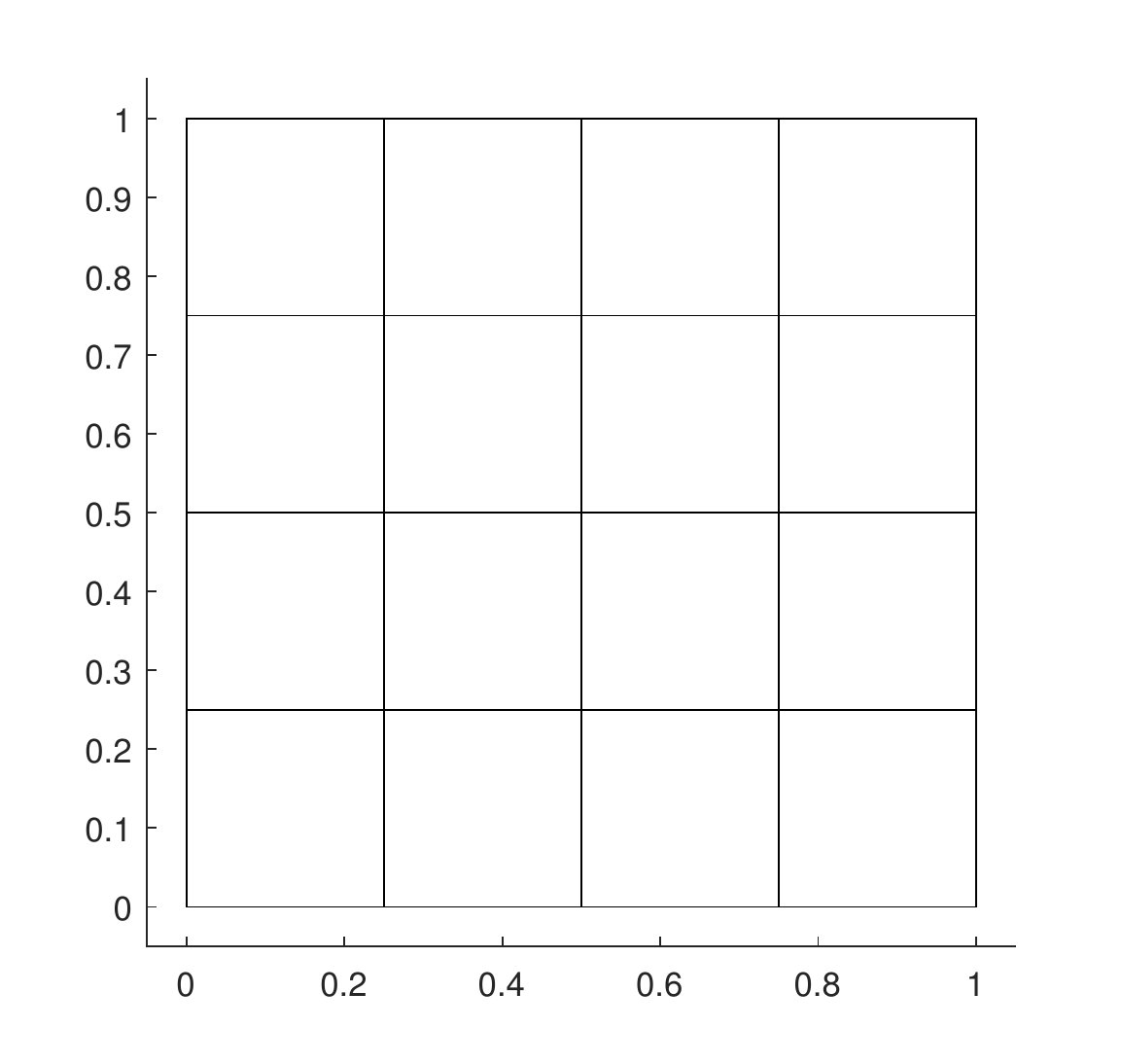}
\end{minipage}
\begin{minipage}{0.335\textwidth}
\includegraphics[width=\textwidth]{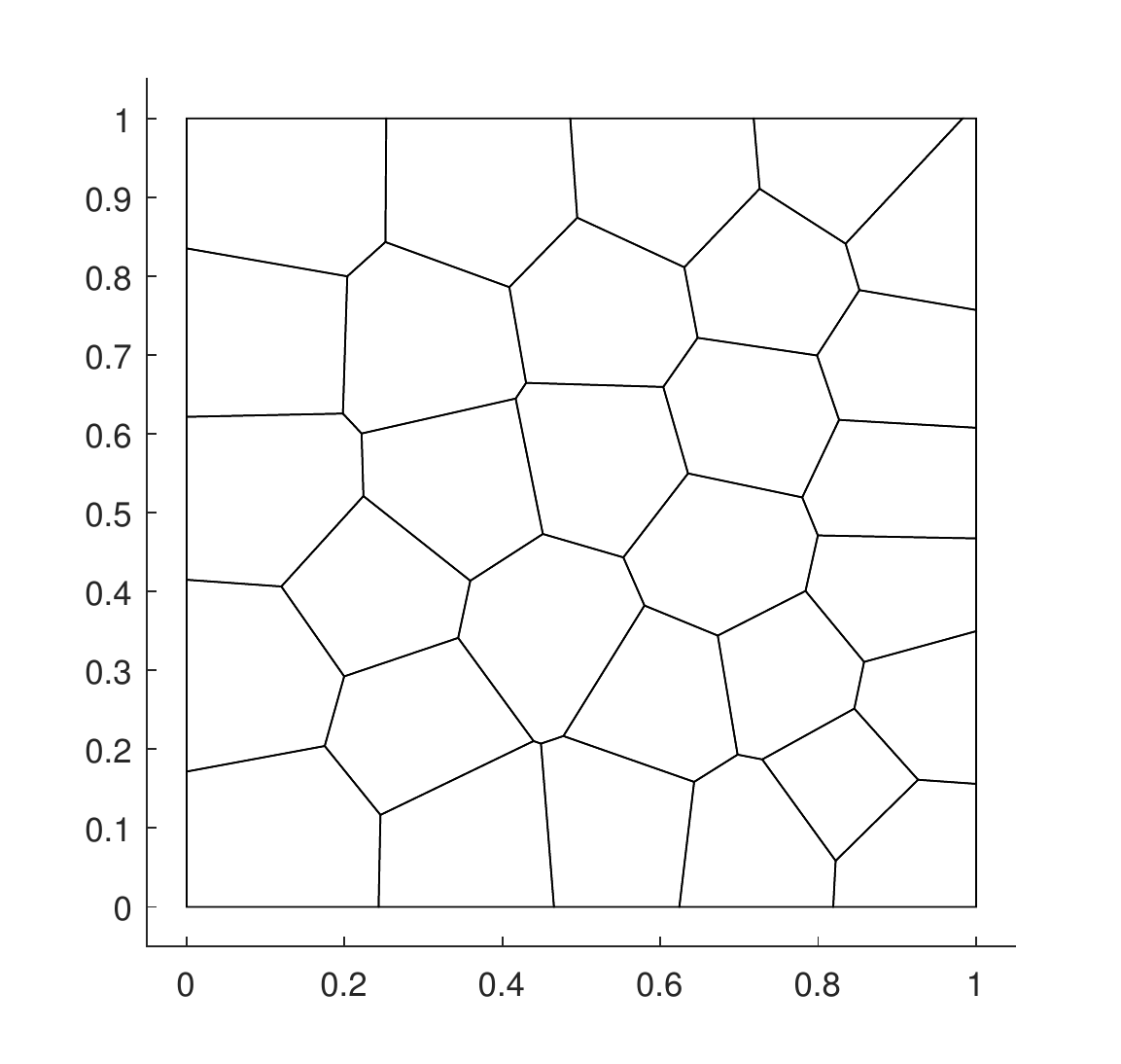}
\end{minipage}
\begin{minipage}{0.31\textwidth}
\includegraphics[width=\textwidth]{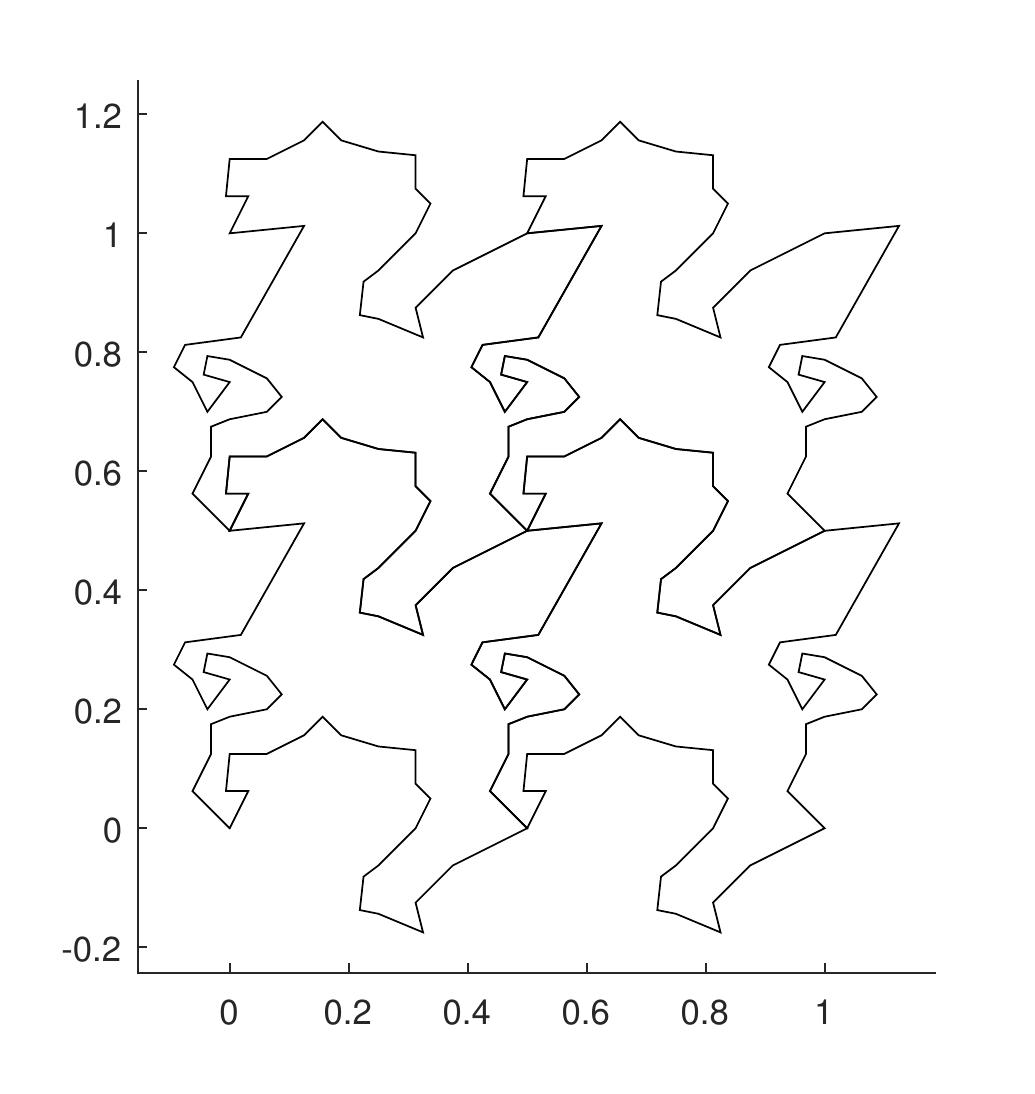}
\end{minipage}
\caption{Different types of meshes: mesh made of squares (left),
  Voronoi-Llyod mesh (center),
and mesh made of Escher horses (right).}
\label{fig:meshes} 
\end{figure}

It is important to note that, since an explicit representation of the numerical approximation $\un$ inside each element is not available, due to the ``virtuality'' of the basis functions,
we cannot compute the $L^2$ and $H^1$ errors of the method directly. Instead, we will compute the following relative errors between $\u$ and $\Pinabla \un$, $\Pinabla$ being defined in \eqref{H1 bulk projector}:
\begin{align} \label{computable errors}
\frac{\lVert \u-\Pinabla \un \rVert_{0,\Omega}}{\Vert \u \Vert_{0,\Omega}}, \quad \quad \quad \quad \quad \frac{\lVert \u-\Pinabla \un \rVert_{1,\taun}}{\Vert \u \Vert_{1,\Omega}}.
\end{align}
We observe that the ``computable'' $H^1$ error in \eqref{computable errors} is related to the exact $H^1$ error. In fact, thanks to Theorem~\ref{theorem abstract error analysis}, we have
\[
\begin{split}
\vert u - \un \vert_{1,\taun} 	& \lesssim \inf_{\qDeltap \in \mathcal{S}^{p,\Delta,-1}(\taun)} \vert \u - \qDeltap \vert_{1,\taun} + \sup_{\vn \in \Vnpz} \frac{\mathcal N_n (\u,\vn)}{\vert \vn \vert_{1,\taun}}  \\
						& \le \vert \u - \Pinabla \un \vert_{1,\taun} + \sup_{\vn \in \Vnpz} \frac{\mathcal N_n (\u,\vn)}{\vert \vn \vert_{1,\taun}};
\end{split}
\]
the convergence of the second term on the right-hand side is provided in Lemma~\ref{lemma approximation non-conforming term}.
Moreover, by the triangle inequality and the stability of the
$H^1$-projection, one also has
\[
\begin{split}
\vert \u - \Pinabla \un \vert_{1,\taun} &\le \vert \u - \Pinabla \u
\vert_{1,\taun}+\vert \Pinabla(\u - \un) \vert_{1,\taun}\\
&\le \vert \u - \Pinabla \u \vert_{1,\taun}+
\vert \u - \un \vert_{1,\taun};
\end{split}
\]
the convergence of the second term on the right-hand side is provided
in Lemma~\ref{lemma approximation harmonic polynomials}.

\subsubsection{Numerical results: $\h$-version} \label{subsection NR h}
In this section, we verify the algebraic rate of convergence of the $\h$-version of the method, validating thus Theorems \ref{theorem h and p VEM} and \ref{theorem L2 estimates}
for different degrees of accuracy $\p=1,2,3,4,5$ of the method.


The numerical results for the problems in $\Omega = (0,1)^2$ with exact solutions $u_1$ and $u_2$, obtained on sequences of Cartesian and Voronoi-Lloyd meshes, are depicted in Figure \ref{fig:hversion_u1} and Figure \ref{fig:hversion_u2}.


\begin{figure}[h]
\begin{center}
\includegraphics[width=.45\textwidth]{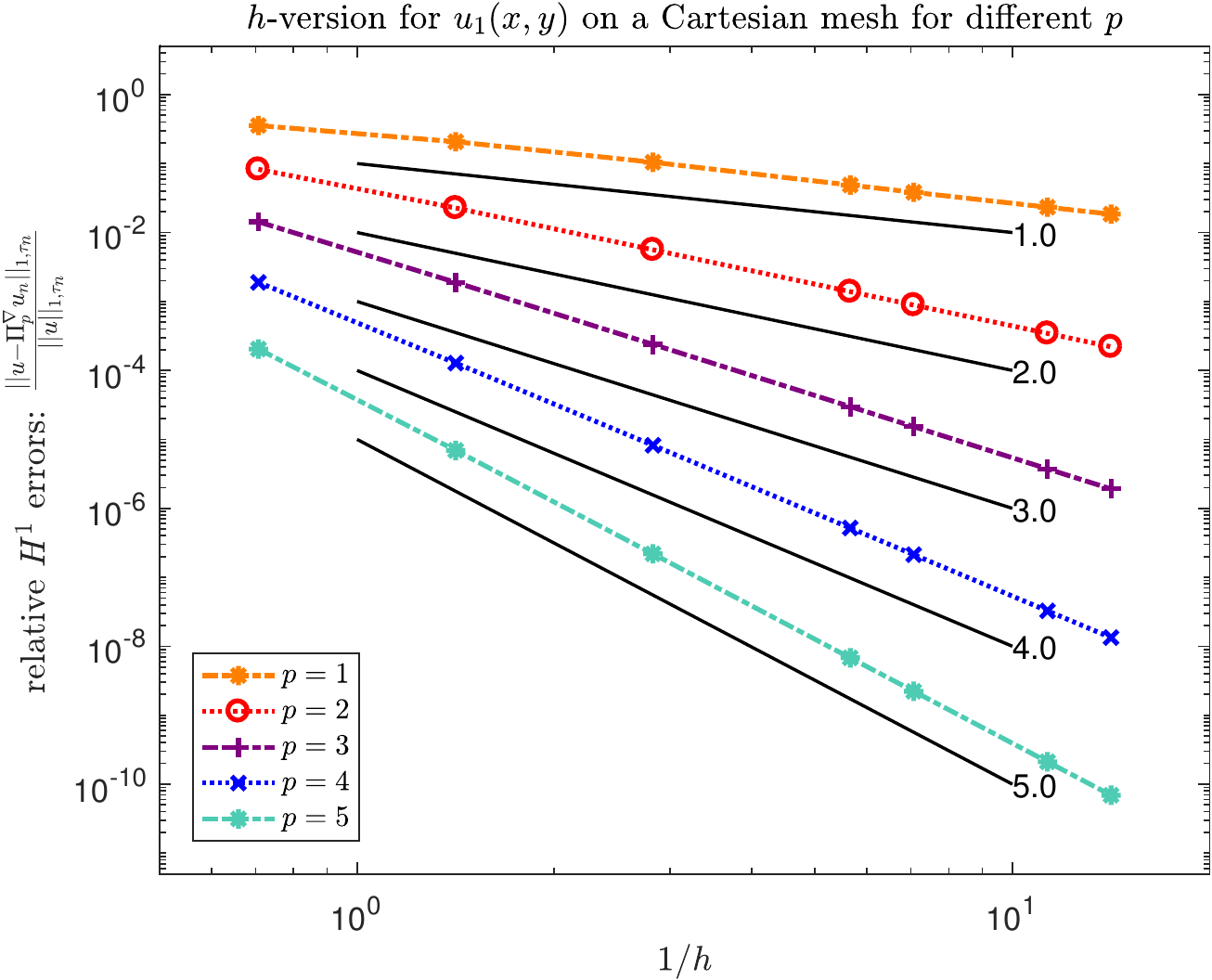}
\hspace{0.05\textwidth}
\includegraphics[width=.45\textwidth]{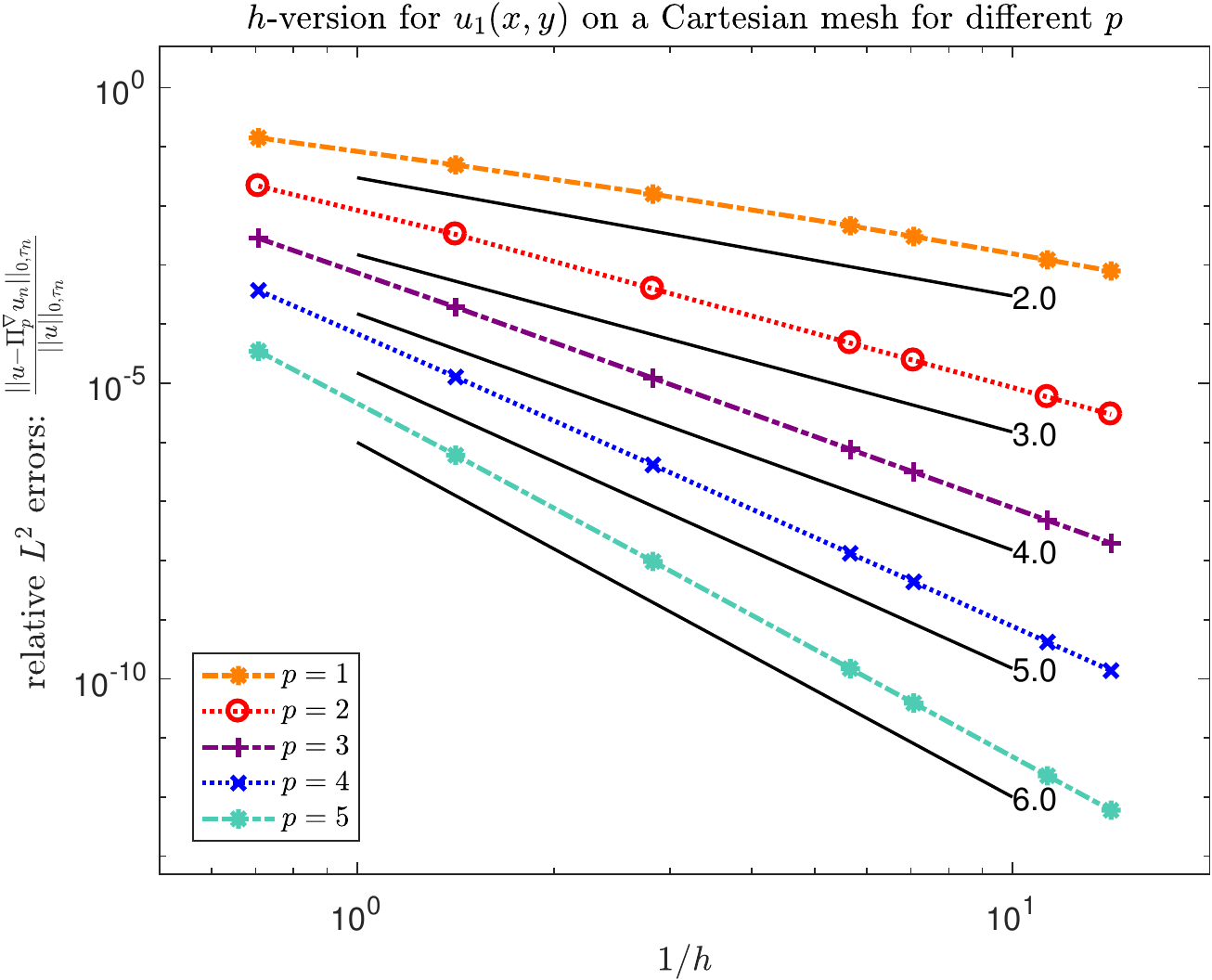} 

\vspace{0.4cm}

\includegraphics[width=.45\textwidth]{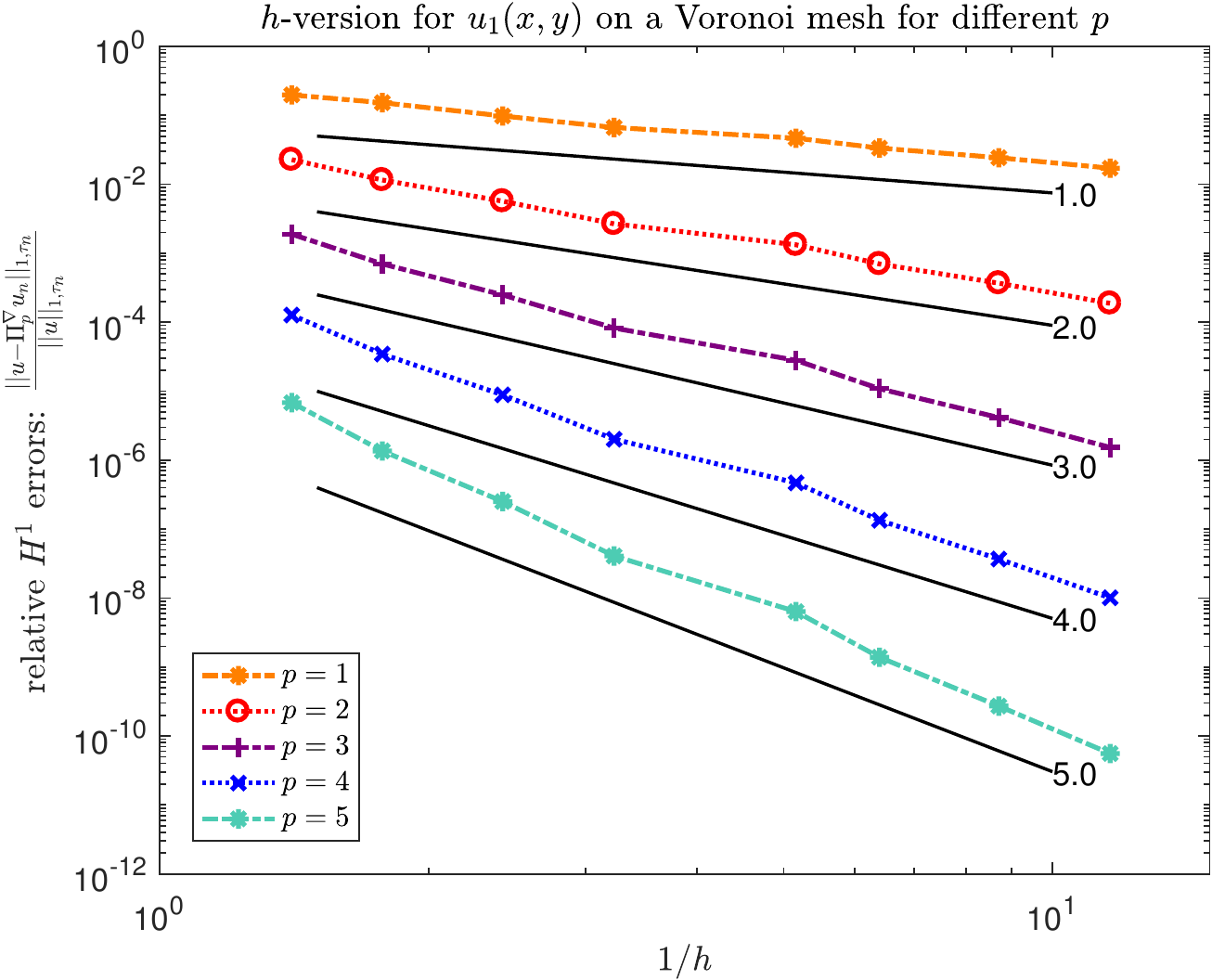}
\hspace{0.05\textwidth}
\includegraphics[width=.45\textwidth]{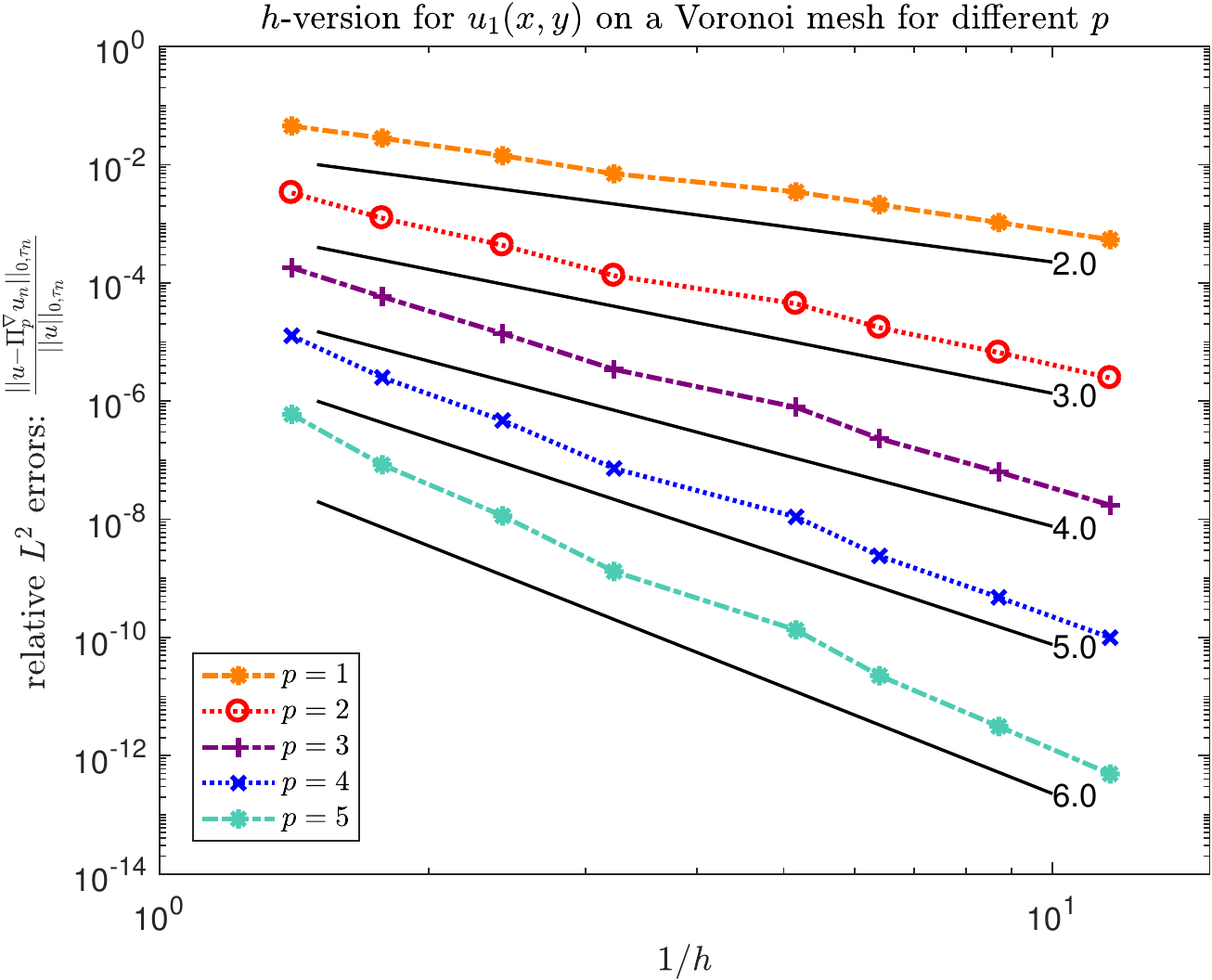}
\end{center}
\caption{Convergence of the $h$-version of the method for the analytic solution $u_1$ on quasi-uniform Cartesian (first row) and Voronoi-Lloyd (second row) meshes; relative $H^1$ errors (left) and relative $L^2$ errors (right) defined in \eqref{computable errors}.}
\label{fig:hversion_u1} 
\end{figure}

\begin{figure}[h]
\begin{center}
\includegraphics[width=.45\textwidth]{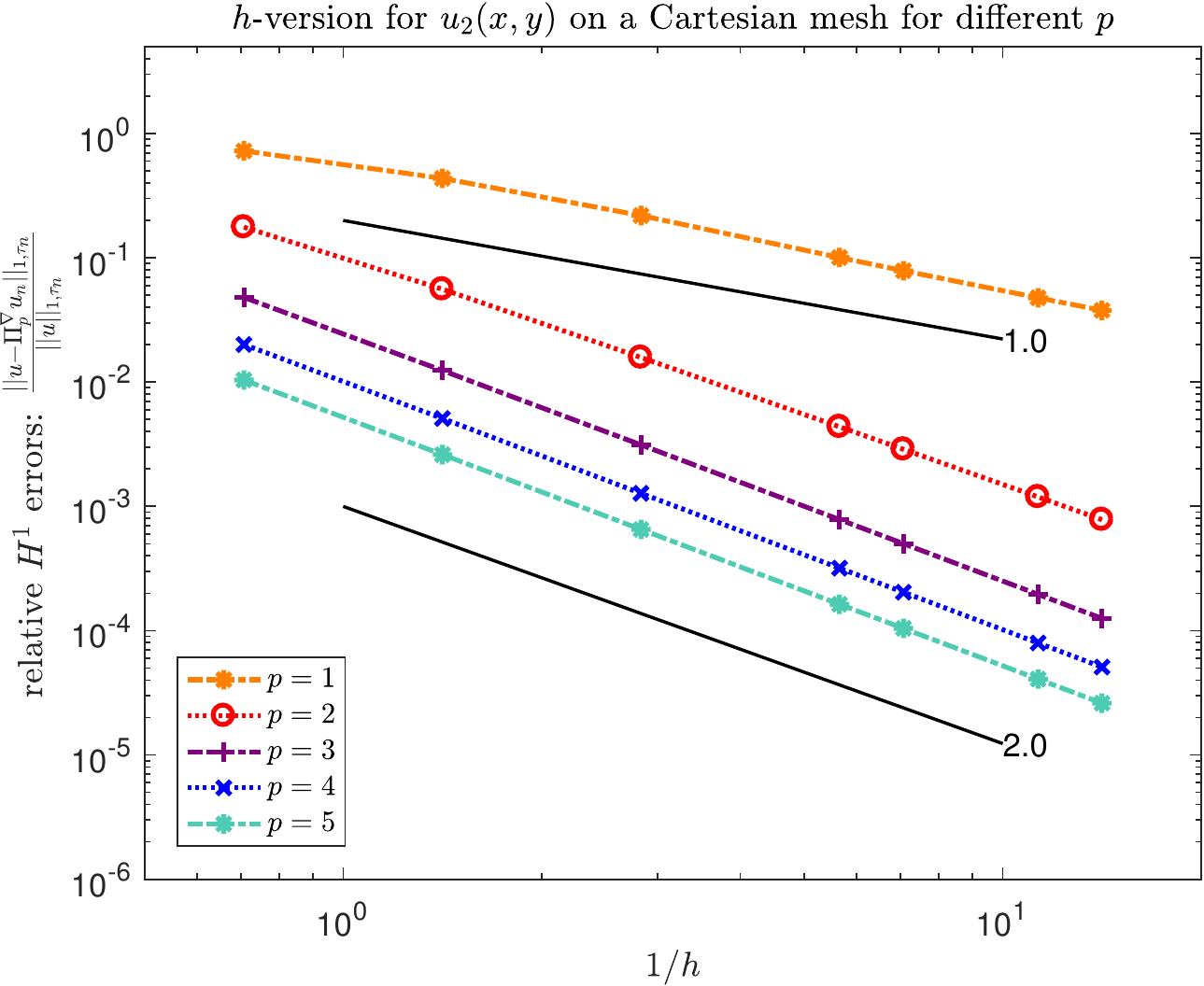}
\hspace{0.05\textwidth}
\includegraphics[width=.45\textwidth]{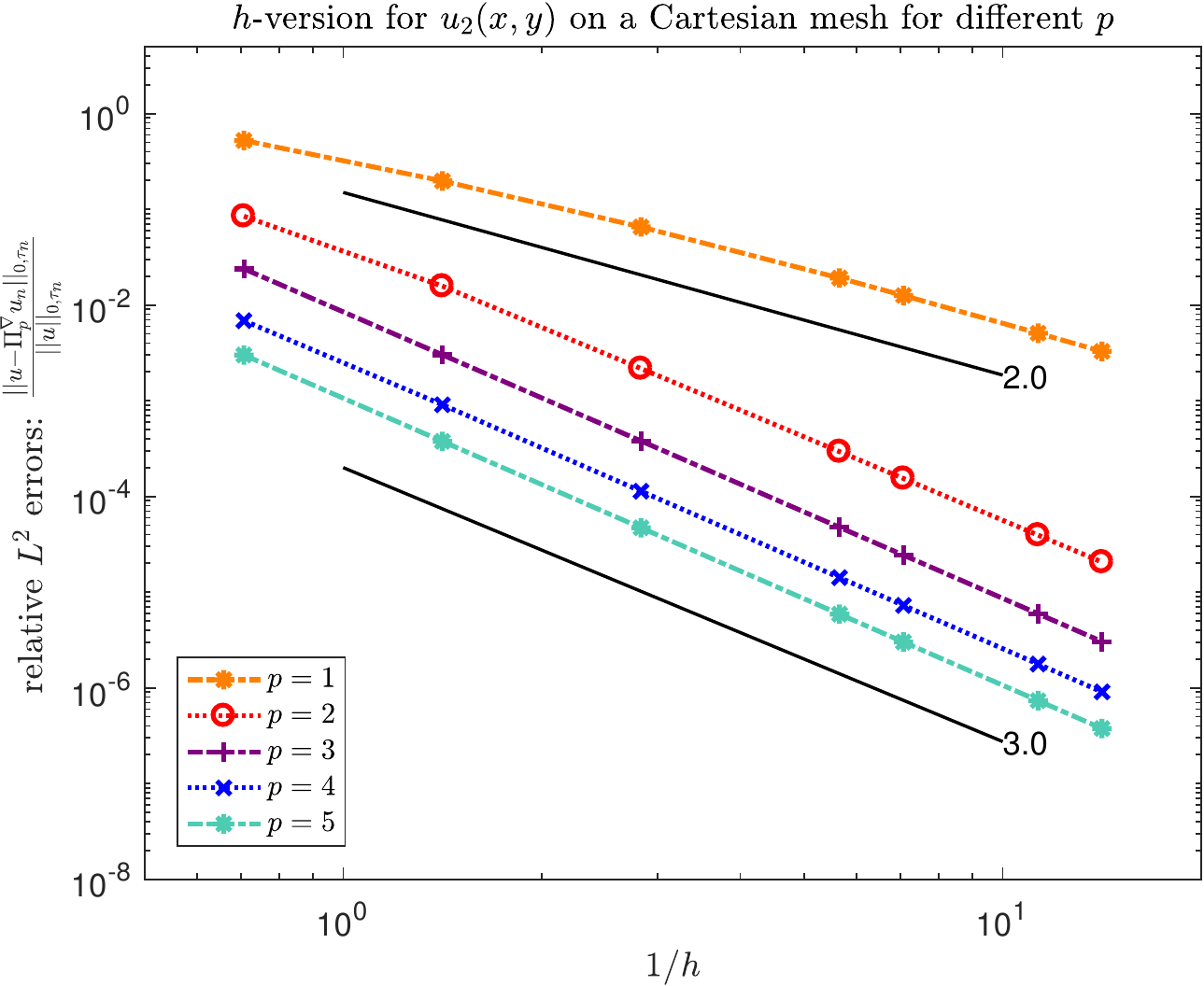} 
	
\vspace{0.4cm}
	
\includegraphics[width=.45\textwidth]{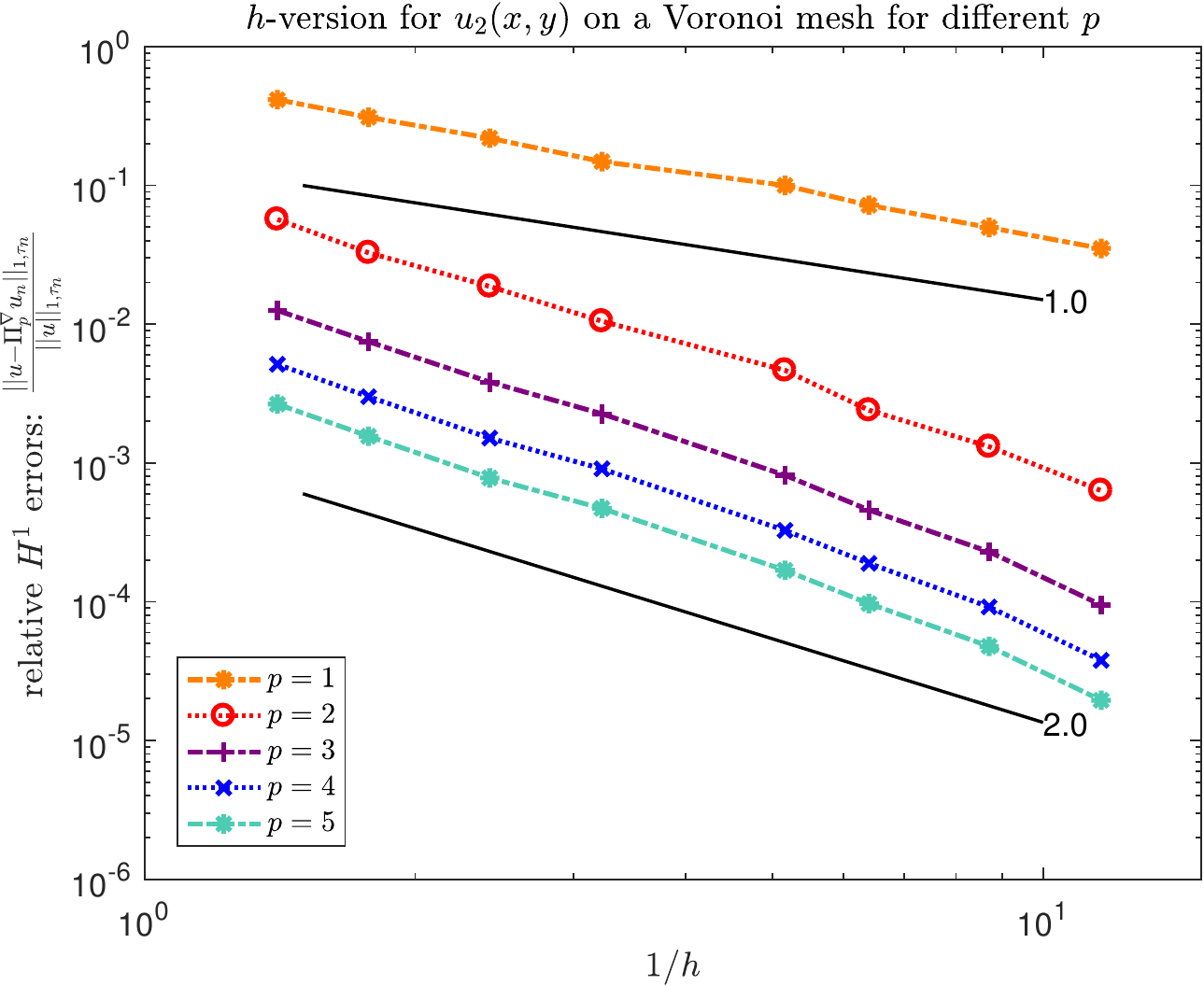}
\hspace{0.05\textwidth}
\includegraphics[width=.45\textwidth]{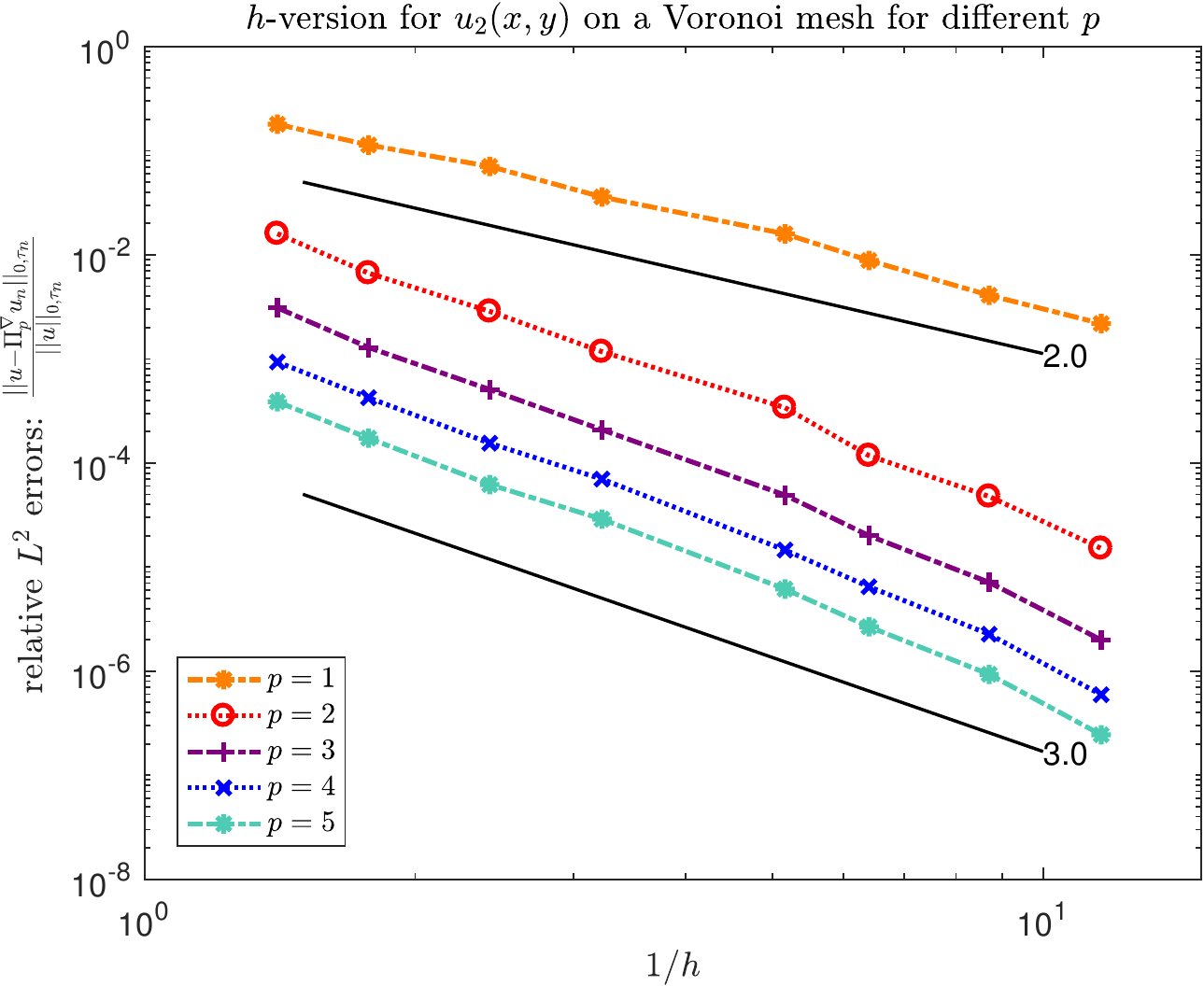}
\end{center}
\caption{Convergence of the $h$-version of the method for the solution $u_2$ with finite Sobolev regularity on quasi-uniform Cartesian (first row) and Voronoi-Lloyd (second row) meshes; relative $H^1$ errors (left) and relative $L^2$ errors (right) defined in \eqref{computable errors}.}
\label{fig:hversion_u2} 
\end{figure}

From Theorems \ref{theorem h and p VEM} and \ref{theorem L2 estimates}, we expect the $H^1$ and $L^2$ errors to behave like 
$\mathcal{O}(\h^{\min(t,\p)})$ and $\mathcal{O}(\h^{\min(t,\p)+1})$, respectively,
where $t+1$ is the regularity of the exact solution $u$, and $p$ is
the degree of accuracy. 
The numerical results in Figures \ref{fig:hversion_u1} and
\ref{fig:hversion_u2} are in agreement with these theoretical
estimates. In fact, for $u_1$, which belongs to $H^s(\Omega)$
for all $s \geqslant 0$, we see that the $H^1$ error
actually converges with order $\mathcal{O}(h^{p})$, and the $L^2$
error with order $\mathcal{O}(h^{p+1})$ for all degrees of accuracy.
On the other hand, 
we observe convergence rates $1$ and $2$,
respectively, for $p=1$, and convergence rates $2$ and $3$,
respectively, for $p=2,3,4,5$. This is due to the fact that the expected convergence is of order 
$\mathcal{O}(h^{\min\{2-\epsilon,p\}})$ in the $H^1$ norm and
$\mathcal{O}(h^{\min\{2-\epsilon,p\}+1})$ in the $L^2$ norm.

\subsubsection{Numerical results: $\p$-version} \label{subsection NR p}

In this section, we validate the exponential convergence of the $\p$-version of the method for the model problem \eqref{Laplace problem strong formulation}
with exact solution $u_1$ on $\Omega=(0,1)^2$ on a Cartesian mesh and a Voronoi mesh made of four elements, respectively, as well as on the domain
$\Omega$ given by the union of four Escher horses (see Figure
\ref{fig:meshes}, right).
The obtained results are depicted in Figure
\ref{fig:pversion_u1}, where
the logarithm of the relative errors defined in \eqref{computable errors} is plotted against the polynomial degree $p$.

\begin{figure}[h]
\begin{center}
\includegraphics[width=.45\textwidth]{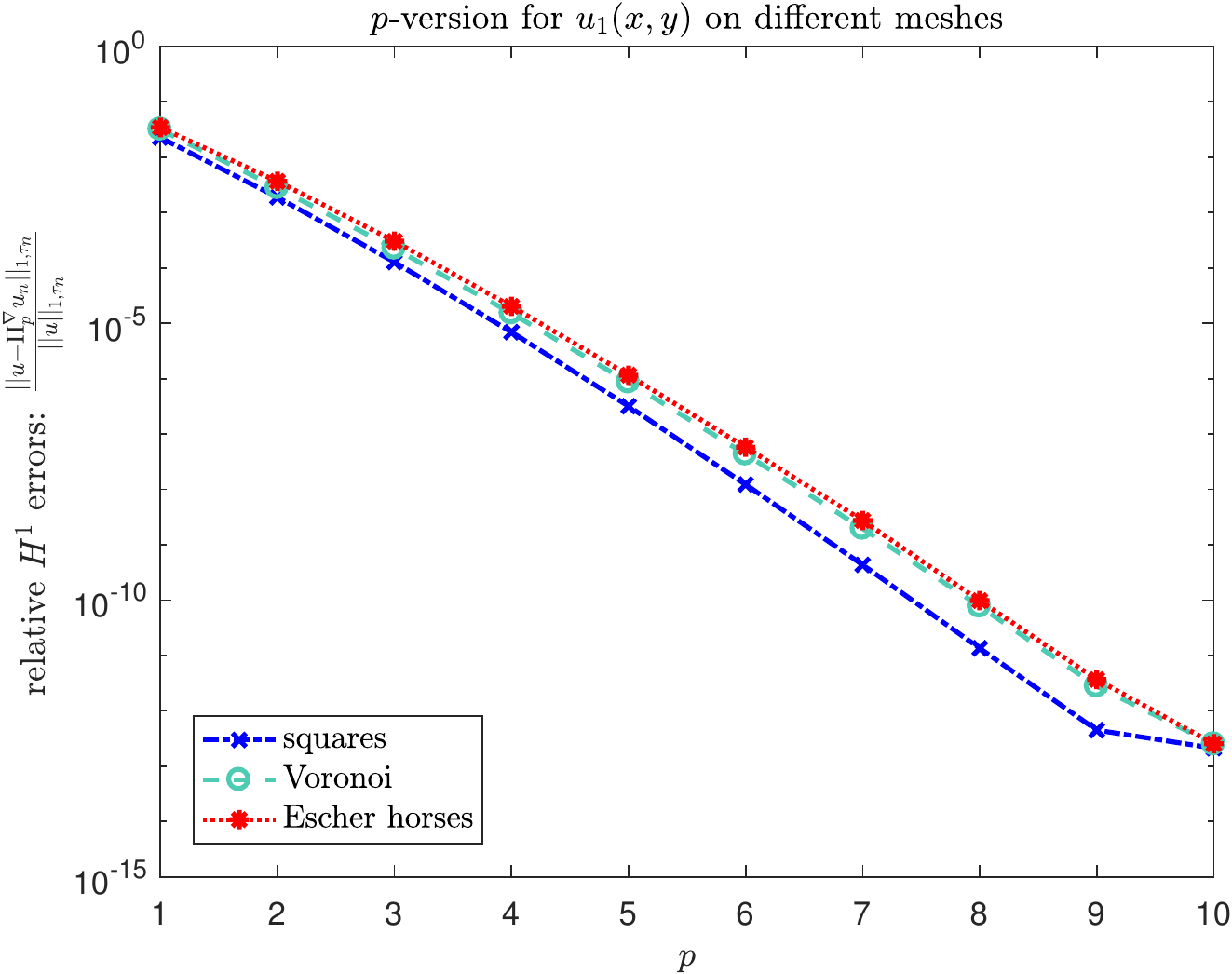}
\hspace{0.05\textwidth}
\includegraphics[width=.45\textwidth]{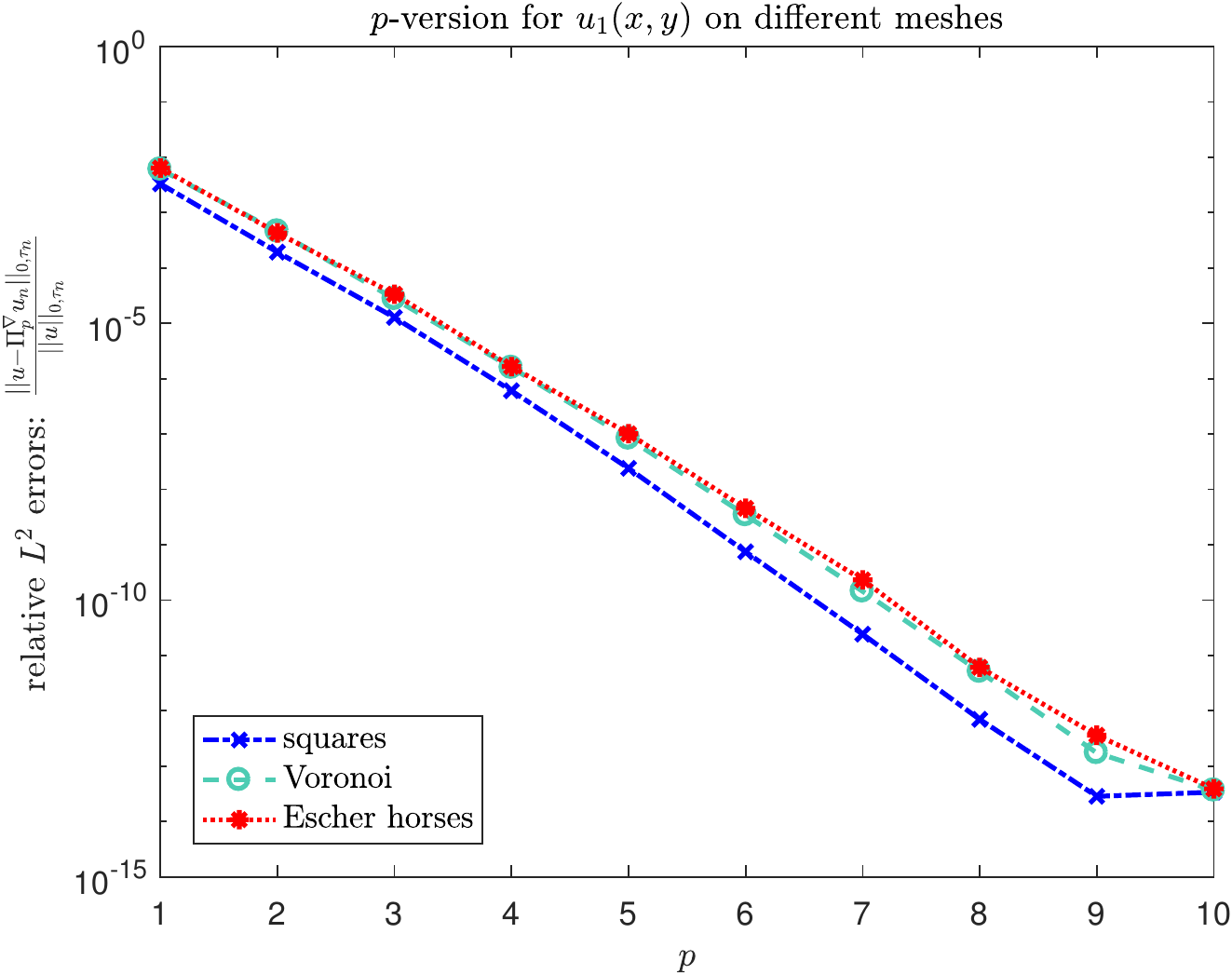} 
\end{center}
\caption{Convergence of the $\p$-version of the method for the analytic solution $u_1$ on a quasi-uniform Cartesian mesh, a Voronoi-Lloyd mesh, and a Escher horses mesh; relative $H^1$ errors (left) and relative $L^2$ errors (right) defined in \eqref{computable errors}.}
\label{fig:pversion_u1} 
\end{figure}
%
One can clearly observe that the exponential convergence 
predicted in Theorem \ref{theorem exponential convergence analytic solution}
is attained, even when employing a very coarse mesh with (non-convex)
non-star-shaped elements, as the one in Figure \ref{fig:meshes}, right.

\subsection{The $\h\p$-version and approximation of corner singularities} \label{subsection NR hp}
So far, both the theoretical analysis and the numerical tests were performed
considering approximation spaces with uniform degree of accuracy $\p$ and with quasi-uniform meshes.

In general, however, the solutions to elliptic problems over polygonal
domains have natural singularities arising in neighbourhoods of the
corners of the domain. 
In particular, for problem \eqref{Laplace problem weak formulation} in a domain $\Omega$ with reentrant corners,
the solution might have a regularity lower than $H^2$, even if the Dirichlet boundary datum $\g$ is smooth;
for a precise functional setting regarding regularity of solutions to elliptic PDEs,
we refer to \cite{grisvard, babuvska1988regularity, SchwabpandhpFEM} and the references therein.
%
%
%
This implies that both the $\h$- and the $\p$-versions of standard Galerkin methods, in general, have limited
approximation properties. In 
particular, employing quasi-uniform meshes and uniform degree of
accuracy, does not entail any sort of exponential convergence.

A possible way to recover exponential convergence, even in presence of corner singularities, is to use the so-called $\h\p$-strategy
firstly designed by Babu\v ska and Guo \cite{BabuGuo_hpFEM, babuskaguo_curvilinearhpFEM, babuvska1988regularity} in the FEM framework, and then generalized to the VEM in \cite{hpVEMcorner}.
This strategy consists in combining mesh refinement towards the corners of the domain and increasing the number of degrees of freedom over the
polygonal decomposition in a non-uniform way.
In this section, we discuss and numerically test an $hp$-version of the presented non-conforming harmonic VEM.

To this purpose, we recall the concept of sequences of geometrically graded polygonal meshes $\{\taun\}_{n\in \mathbb N}$.
For a given $n \in \mathbb N$, $\taun$ is a polygonal mesh consisting of $n+1$ {layers}, where we define a \textit{layer} as follows.
The so-called $0$-th {layer} is the set of all polygons in $\taun$ abutting the vertices of $\Omega$. The other layers are defined inductively by requiring that the $\ell$-th {layer} consists of those polygons, which abut the polygons in
the ($\ell-1$)-th layer. More precisely, for all $\ell =1,\dots,n$, we set
\begin{equation*} 
\L_{n,\ell} := \L_\ell := \left\{ \E \in \taun \mid \overline \E \cap \overline{\E_{\ell-1}} \ne \emptyset \text{ for some } \E_{\ell-1} \in \L_{\ell-1},\, \E \not \subseteq \cup_{j=0}^{\ell-1} L_j \right\}.
\end{equation*}
The $\h\p$-gospel states that, in order to achieve exponential convergence of the error, one has to employ geometrically graded sequences of meshes.
For this reason, we consider sequences $\{\taun\}_{n\in \mathbb N}$ satisfying (\textbf{D1})-(\textbf{D3}), but not (\textbf{D4}); we require instead
\begin{itemize}
\item[(\textbf{D5})]  for all $n\in \mathbb N$, there exists $\sigma \in (0,1)$, called \textit{grading parameter}, such that
\begin{equation} \label{grading assumption}
\hE \approx
\begin{cases}
\sigma^{n} & \text{if } \E \in \L_0\\
\frac{1-\sigma}{\sigma} \dist(\E, \mathcal V^\Omega) & \text{if } \E\in \L_\ell,\, \ell=1,\dots,n,
\end{cases}
\end{equation}
where $\mathcal V^\Omega$ denotes the set of vertices of the polygonal domain $\Omega$.
\end{itemize}
Sequences $\{\taun\}_{n\in \mathbb N}$ satisfying (\textbf{D5}) have the property that the layers ``near'' the corners of the domain consist of elements with measure converging to zero, whereas the other layers consist of polygons with fixed size.
In Figure \ref{figure example grading meshes}, we depict three meshes
that represent the third elements $\mathcal T_3$ in certain sequences
of meshes 
of the L-shaped domain
\begin{equation} \label{L shaped domain}
\Omega := (-1,1)^2\setminus (-1,0)^2,
\end{equation}
which are graded, for simplicity, only towards the vertex $\mathbf 0$.

\begin{figure}[htbp]
\begin{minipage}{0.33\textwidth} 
\includegraphics[width=\textwidth]{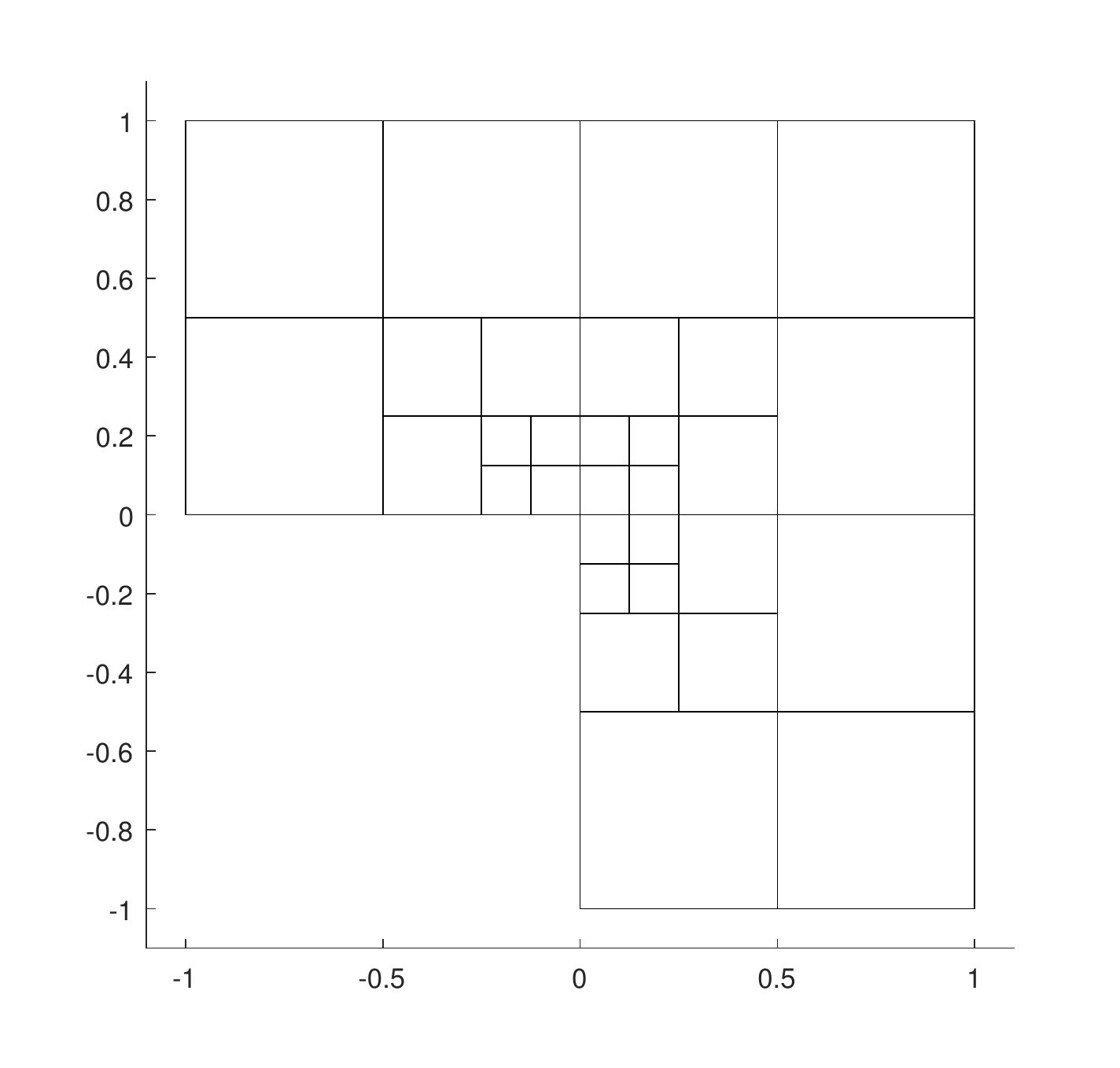}
\end{minipage}
\begin{minipage}{0.33\textwidth}
\includegraphics[width=\textwidth]{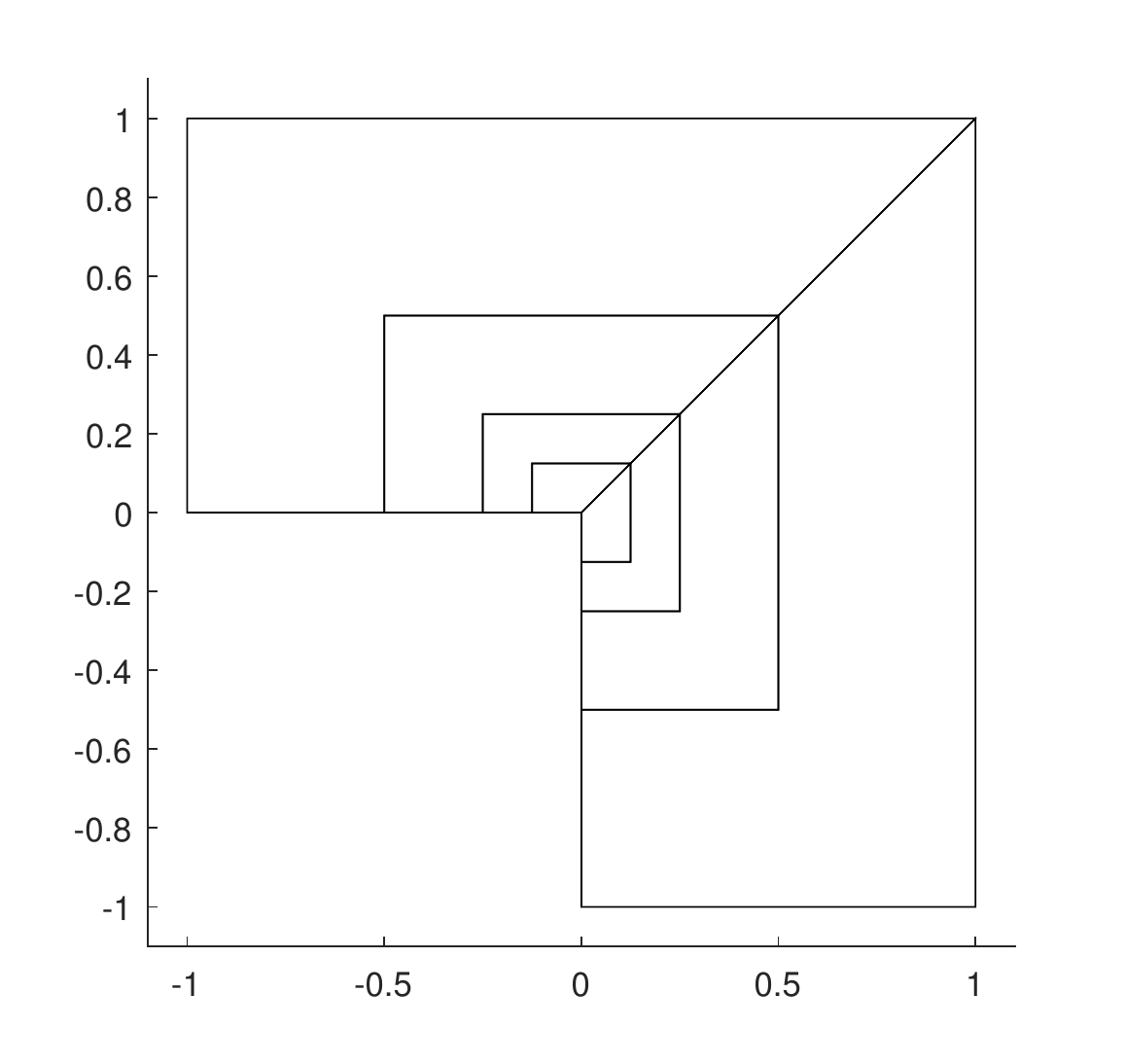}
\end{minipage}
\begin{minipage}{0.33\textwidth}
\includegraphics[width=\textwidth]{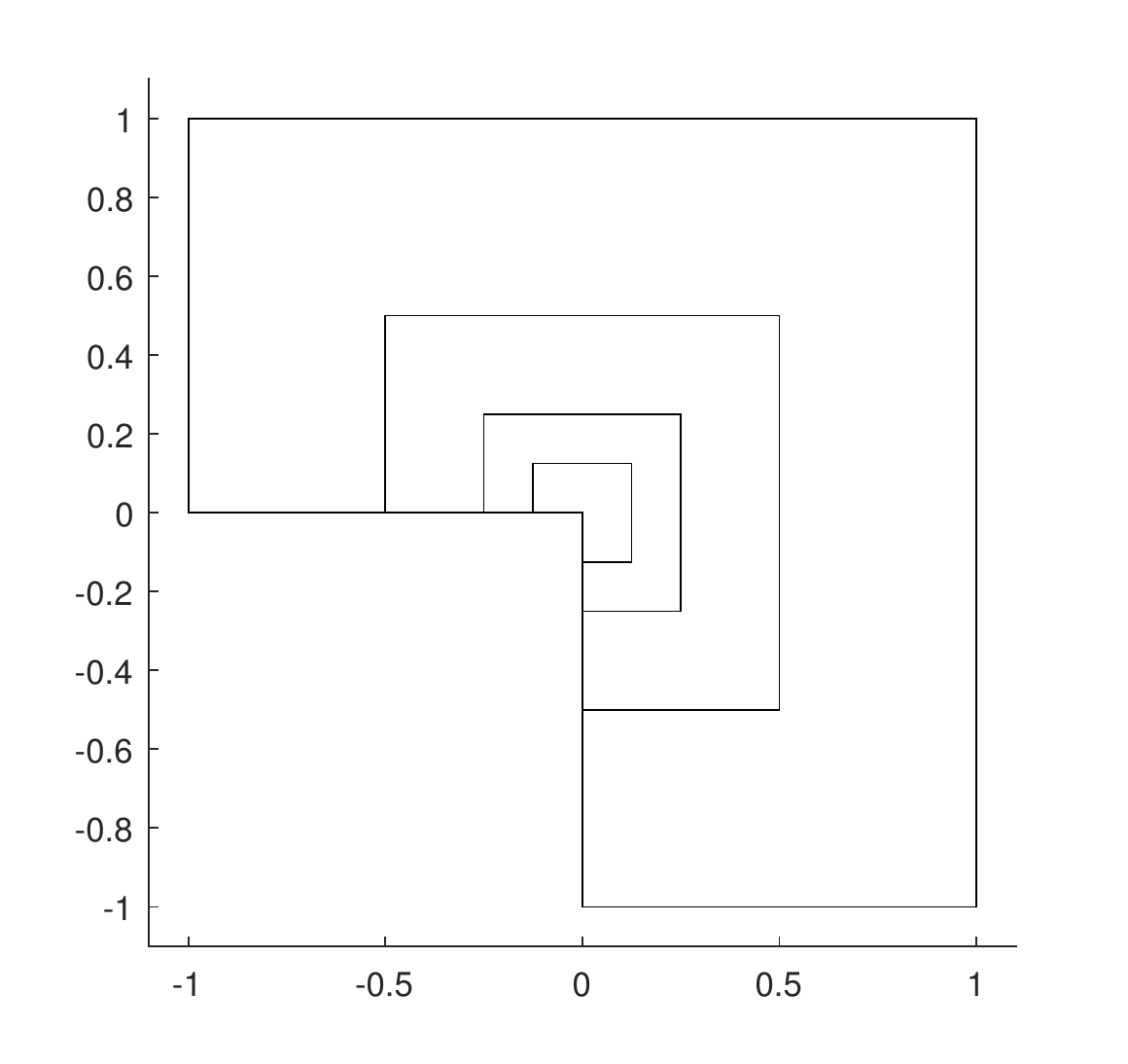}
\end{minipage}
\caption{Third element $\mathcal T_3$ in three different sequences of geometrically graded meshes (type (a)-(c) from left to right) with $\sigma=0.5$.}
\label{figure example grading meshes} 
\end{figure}

We still miss a crucial ingredient for a complete description of the $\h\p$-strategy, namely harmonic VE spaces with non-uniform degrees of accuracy.
For all $n\in \mathbb N$, we can order the elements in $\taun$ as $\E_1$, $\E_2$, \dots, $\E_{\card(\taun)}$; then we consider a vector $\pbold_n \in \mathbb N^{\card(\taun)}$ whose entries are defined as follows:
\begin{equation} \label{graded degree of accuracy}
(\pbold_n)_j :=
\begin{cases}
1 & \text{if } \E_j \in L_0\\
\max(1, \lceil \mu(\ell+1) \rceil) & \text{if } \E_j \in \L_\ell,\, \ell =  1,\dots,n,\\
\end{cases}
\end{equation}
where $\mu$ is a positive parameter to be assigned, and where $\lceil \cdot \rceil$ is the ceiling function.

Having $\pbold_n$ for all $n\in \mathbb N$, we consider the elements $\e_1$, $\e_2$, \dots, $\e_{\card(\mathcal E_n)}$ in $\mathcal E_n$;
we consequently define a vector $\pboldE_n \in \mathbb N^{\card(\mathcal E_n)}$, whose entries are built using the following rule (\emph{maximum rule}):
\begin{equation*} 
(\pboldE_n)_j :=
\begin{cases}
(\pbold_n)_{i} 					& \text{if } \e_j \in \mathcal E_n^B \text{ and } \e_j \subset \partial\E_i\\
\max((\pbold_n)_{i_1}, (\pbold_n)_{i_2}) 	& \text{if } \e_j \in \mathcal E_n^I \text{ and } \e_j \subset \partial \E_{i_1} \cap \partial \E_{i_2}.\\
\end{cases}
\end{equation*}
At this point, we define the local harmonic VE spaces with non-uniform degrees of accuracy as follows. For all $\E \in \taun$, we set
\[
\VDeltaE := \left\{ \vn \in H^1{(\E)} \, \mid \, \Delta \vn = 0 \text{ in } \E,\, (\nabla\vn\cdot\n_\E) _{|_{\e_j}} \in \mathbb P_{(\pboldE_n)_j}(\e_j) \, \forall \e_j \text{ edge of } \E  \right\}.
\]
The global non-conforming space and the set of global degrees of freedom are defined similarly to those for the case of uniform degree, see Section \ref{section VEM}.
The difference is that now the degrees of freedom and the corresponding ``level of non-conformity'' of the method vary from edge to edge.
This approach is similar to that discussed in~\cite{hpVEMcorner} for the $\h\p$-version of the conforming standard VEM.

Under this construction, one should be able to prove the following
convergence result 
in terms of the number of degrees of freedom.
There exists $\mu >0$ such that the choice \eqref{graded degree of accuracy} guarantees
\begin{equation} \label{exponential convergence}
\vert \u - \un \vert_{1,\taun} \le c \exp{\left(-b \sqrt[2]{\# \text{dofs} }\right)},
\end{equation}
for some positive constants $b$ and $c$, depending on $u$, $\rho_1$, $\rho_2$, $\Lambda$, and $\sigma$, where $\# \text{dofs}$ denotes the number of degrees of freedom of the discretization space. This exponential convergence in terms of the dimension of the
approximation space was proven for conforming harmonic VEM in
\cite{HarmonicVEM} and for Trefftz DG-FEM in \cite{hmps_harmonicpolynomialsapproximationandTrefftzhpdgFEM}.
In the present non-conforming harmonic VEM, the setting of the
proof of such exponential convergence would follow 
the same lines as that of the two methods mentioned above. We omit
a detailed analysis and
present here some numerical results.

We underline that the exponential convergence in \eqref{exponential convergence} is faster (in terms of the dimension of the space) than that of
standard $\h\p$-FEM \cite{SchwabpandhpFEM} and $\h\p$-VEM \cite{hpVEMcorner}, whose decay rate is $\mathcal{O}(\exp{ \left(-b\sqrt[3]{\# \text{dofs}} \right)})$,
due to the use of harmonic subspaces instead of complete FE or
  VE spaces.

For our numerical tests, we consider the boundary value problems
\eqref{Laplace problem weak formulation} on the L-shaped domain~$\Omega$ defined in \eqref{L shaped domain}, with exact solution
\begin{align*}
u_3(x,y)=u_3(r,\theta)=r^{\frac{2}{3}} \sin \left(\frac{2}{3} \theta+\frac{\pi}{3} \right).
\end{align*}
We note that $u_3 \in H^{\frac{5}{3}-\epsilon}(\Omega)$ for every $\epsilon>0$ arbitrarily small, and also $u_3 \in H^{\frac{5}{3}-\epsilon}(\Omegaext)$,
where~$\Omegaext$ is defined in \eqref{inflated domain}; we stress that $u_3$ is the natural solution, singular at $\mathbf 0=(0,0)$, which arises when solving a Poisson
problem in the L-shaped domain~$\Omega$.

In Figure \ref{fig:hpversion_u2}, we show the convergence of the $\h\p$-version of the method for different values of the grading parameter $\sigma$ used in \eqref{grading assumption} and with degrees of accuracy graded according to \eqref{graded degree of accuracy}, having set $\mu =1$.
We plot the logarithm of the relative $H^1$ error \eqref{computable errors} against 
the square root of the number of degrees of freedom. 

Note that, due to the different number of degrees of freedom for each type of mesh, the range of the coordinates varies from plot to plot.
The straight lines for $\sigma=0.5$ and $\sigma=\sqrt{2}-1$ indicate agreement with \eqref{exponential convergence} for meshes of type (a) and (b).
However, when employing the mesh of type~(c) with all grading parameters,
and when employing grading parameter $\sigma=(\sqrt{2}-1)^2$ for meshes of all types,
we do not observe exponential convergence \eqref{exponential convergence}. In the former case, we deem that this is due to the shape of the elements, whereas, in the latter,
this could be due to the fact that the size of the elements in the outer layers is too large if picking the parameter $\mu$ in \eqref{graded degree of accuracy} equal to~$1$.
\begin{figure}[htbp]
\begin{minipage}{0.325\textwidth}
\includegraphics[width=\textwidth]{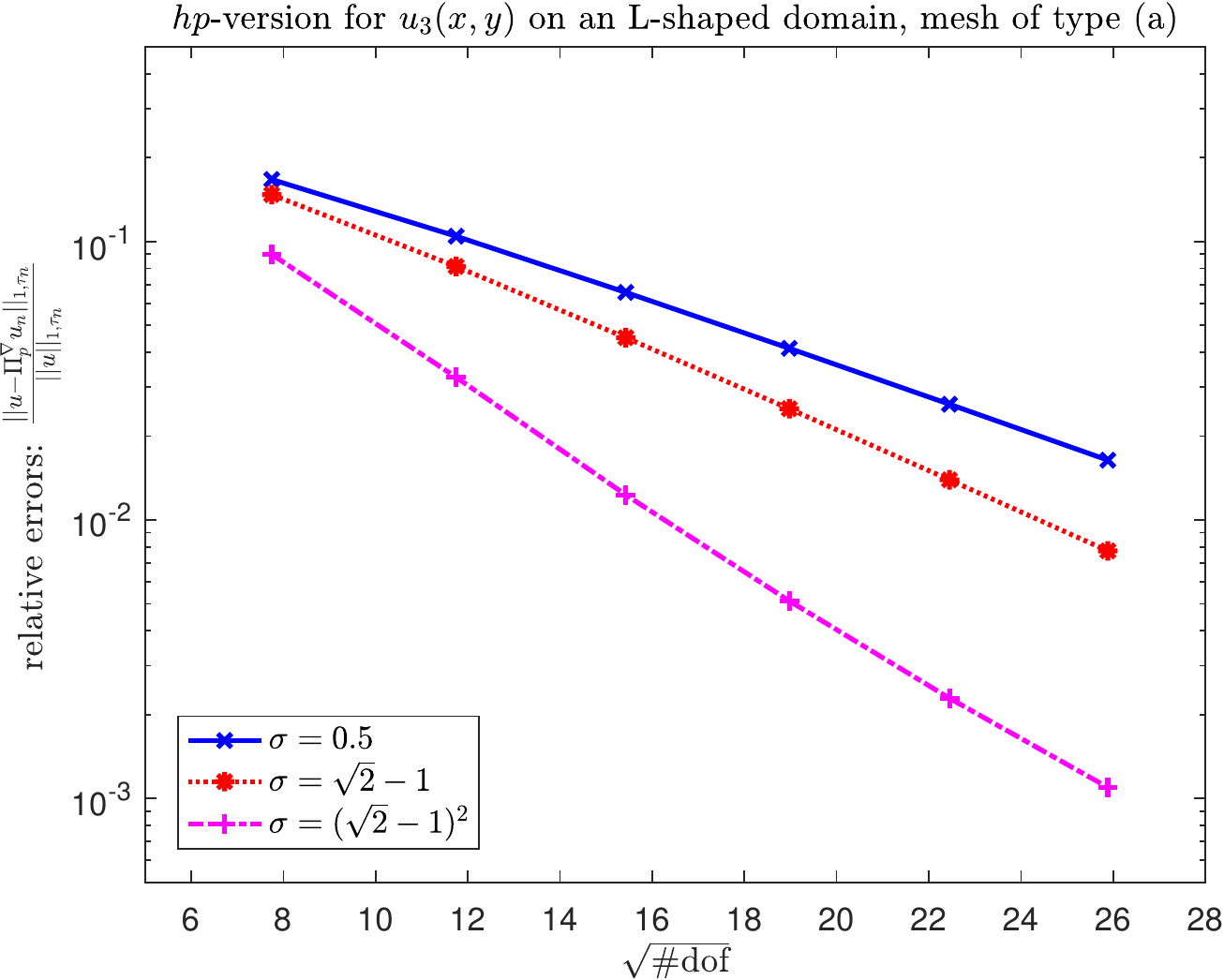}
\end{minipage}
\hfill
\begin{minipage}{0.325\textwidth} 
\includegraphics[width=\textwidth]{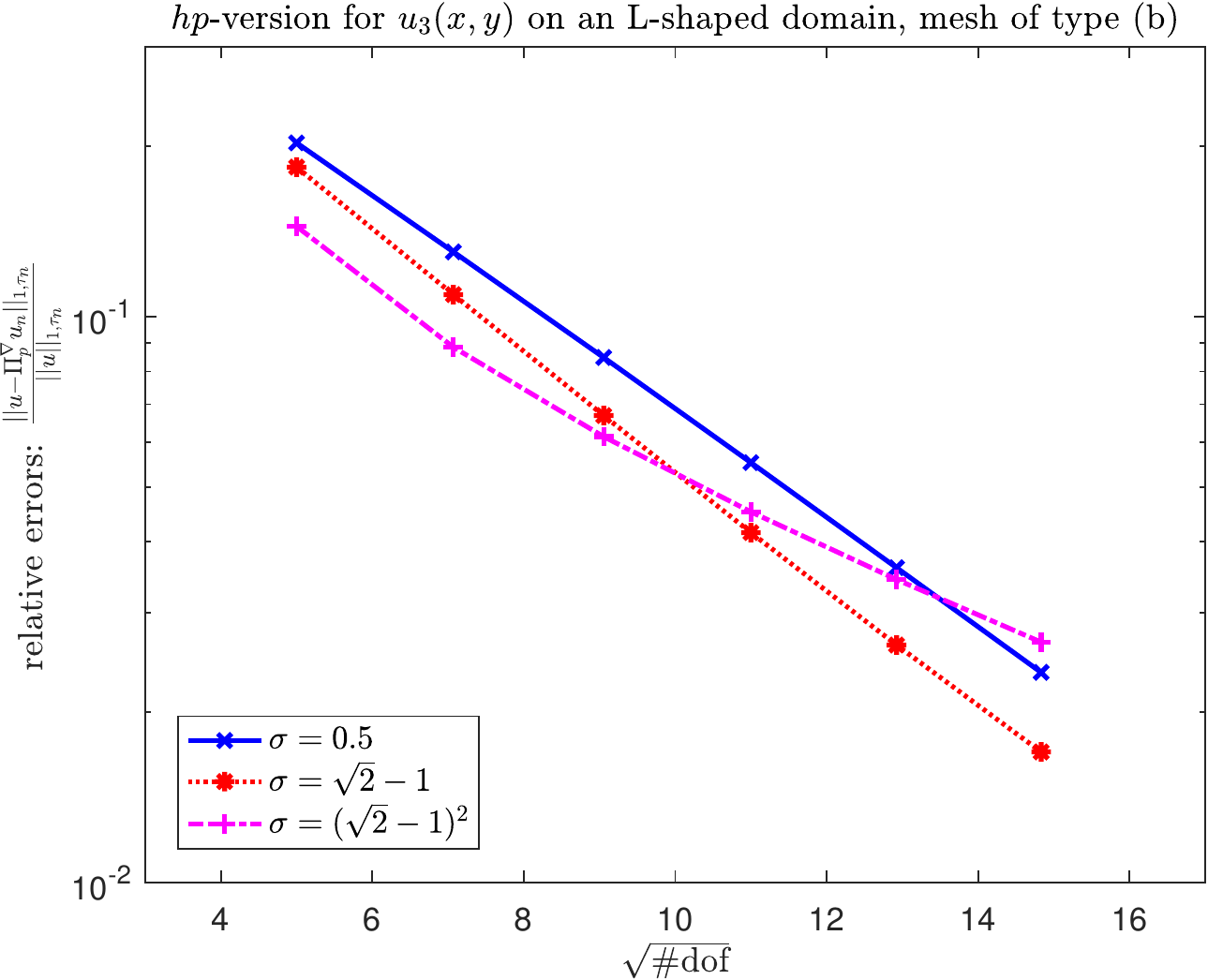}
\end{minipage}
\hfill
\begin{minipage}{0.325\textwidth} 
\includegraphics[width=\textwidth]{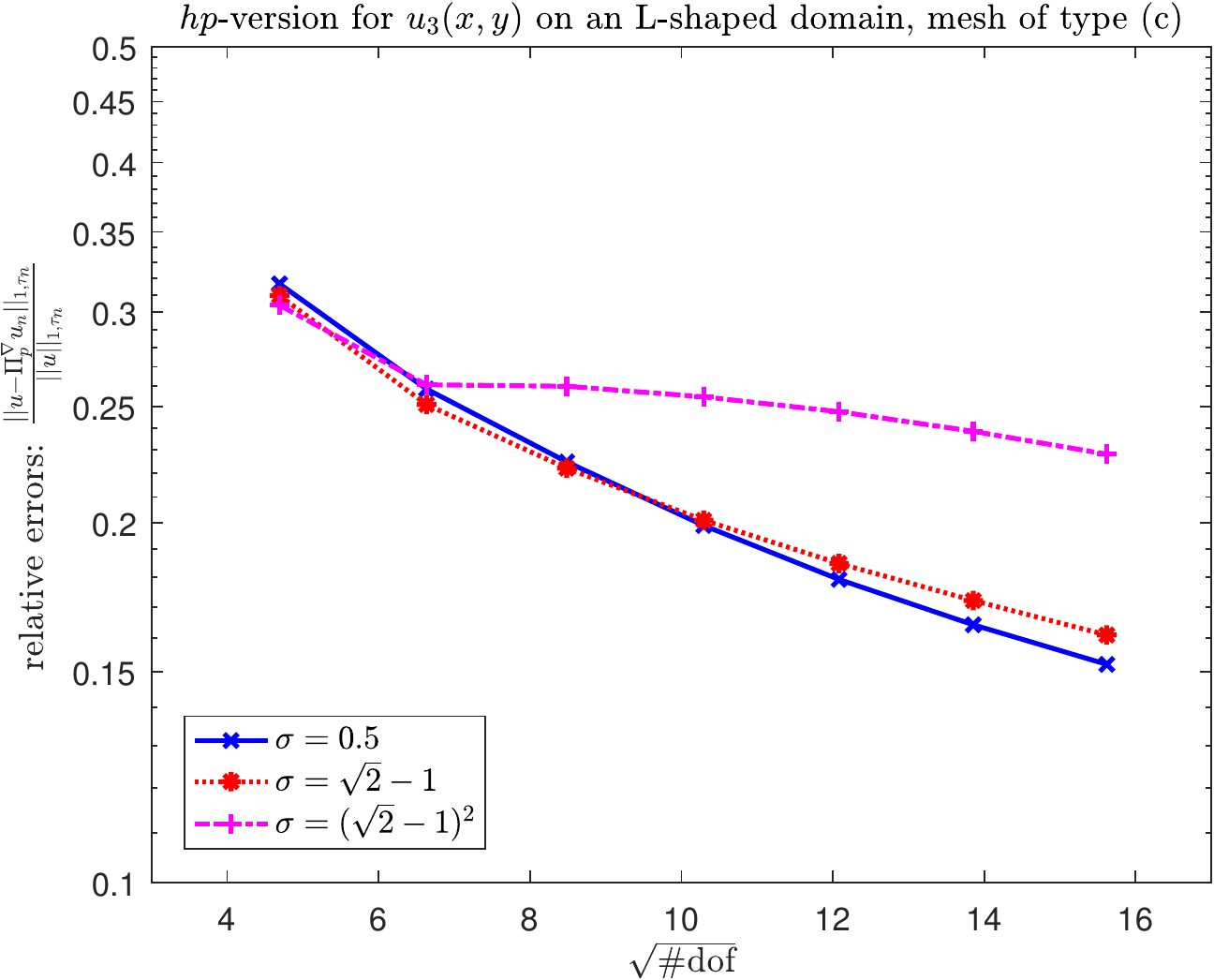}
\end{minipage}
\caption{Convergence of the $hp$-version of the method for the
solution $u_3$ on an L-shaped domain~$\Omega$, for the three sequences of graded meshes represented in Figure~\ref{figure example grading meshes}; relative $H^1$ errors defined in \eqref{computable errors}.
The grading parameter $\sigma$ is set to $1/2$, $\sqrt{2}-1$ and $(\sqrt{2}-1)^2$.}
\label{fig:hpversion_u2} 
\end{figure}

We point out that, in the framework of the conforming harmonic VEM~\cite{HarmonicVEM}, a similar behaviour for the mesh of type (c) was observed.
Instead, when employing the $\h\p$-version of the standard
(non-harmonic) VEM~\cite{hpVEMcorner}, the performance is more robust and the decay of the error is always straight exponential.
This suboptimal behaviour might be intrinsic in the use of harmonic
polynomials, or might be due to the choice of the harmonic polynomial basis employed 
in the construction of the method, see Appendix~\ref{section appendix implementation details}.

\section{Conclusions}
%
%
In this paper, we investigated non-conforming harmonic VEM for the approximation of solutions to 2D Dirichlet-Laplace problems,
providing error bounds in terms both of $\h$, the mesh size, and of $\p$, the degree of accuracy of the method.
We gave some hints concerning the extension of the method to the 3D case, where the design of a suitable stabilization is the only missing item.
Numerical tests validating the theoretical convergence results, as well as testing the $\h\p$-version of the method in presence of corner singularities, were presented.

The technology herein introduced can also be seen as an intermediate step towards the construction of non-conforming Trefftz-VE spaces for the approximation of solutions to Helmholtz problems,
which has been recently investigated in \cite{ncTVEM_Helmholtz}.

\section*{Acknowledgements}
The authors have been funded by the Austrian Science Fund (FWF) through the projects P 29197-N32 and F 65.
They are very grateful to the anonymous referees for their valuable and constructive comments, which have contributed to the improvement of the paper.

\begin{appendices}
\section{Details on the implementation} \label{section appendix implementation details}
\label{subsection implementation details}
In this section, we discuss some practical aspects concerning the implementation of the non-conforming harmonic VEM in 2D.
We employ henceforth the notation of \cite{hitchhikersguideVEM}.
It is worth to underline that we present herein only the case with uniform degree of accuracy; the implementation of the $\h\p$ version is dealt with similarly.
As a first step, we begin by fixing the notation for the various bases instrumental for the
construction of the method.\vspace{-0.2truecm}

\paragraph*{Basis of $\mathbb P_{\p-1}(\e)$ for a given $\e \in \mathcal{E}^\E$.} Using the same notation as in \eqref{local dofs},
we denote the basis of $\mathbb P_{\p-1}(\e)$, $\e\in \mathcal E^\E$, by 
$\{ m_{r}^\e \}_{r=0,\ldots,p-1}$.
The choice we make is 
\begin{align} \label{def_malpha_legendre}
m_{r}^\e(\x):=\mathbb L_{r}\left( \phi_\e^{-1}(\x) \right) \quad \forall r=0,\dots,\p-1,
\end{align}
where $\phi_e: [-1,1] \to e$ is the linear transformation mapping the interval $[-1,1]$ to the edge $\e$, and $\mathbb L_{r}$ 
is the Legendre polynomial of degree $r$ 
over $[-1,1]$.
We recall, see e.g. \cite{SchwabpandhpFEM}, for future use the orthogonality property 
\begin{align} \label{orthogonality legendre}
(m_r^e,m_s^e)_{0,e} = 
\frac{h_e}{2} \int_{-1}^{1} \mathbb L_r(t) \mathbb L_s(t) \, \text{d}t = \frac{h_e}{2r+1} \delta_{rs} \quad \forall r,s=0,\dots,p-1,
\end{align}
where $\delta_{rs}$ is the Kronecker delta ($1$ if $r=s$, $0$ otherwise).

\paragraph*{Basis of $\mathbb H_p(\E)$ for a given $\E \in \taun$.}
We denote the basis of the space of harmonic polynomials $\mathbb
H_p(K)$ by 
$\{ q_\alpha^\Delta \}_{\alpha=1,\ldots,\npDelta}$, where $\npDelta:=\dim \mathbb H_p(K) = 2\p+1$.
The choice we make 
for this basis is 
\begin{equation*} 
\begin{split}
&q_1^\Delta(\x) = 1;\\
&q_{2l}^\Delta(\x) = \sum_{k=1, \, k \text{ odd }}^{l} (-1)^{\frac{k-1}{2}} \binom{l}{k} \left( \frac{x-x_\E}{h_K} \right)^{l-k} \left( \frac{y-y_\E}{h_K} \right)^{k} \quad \forall l=1,\dots,\p; \\
&q_{2l+1}^\Delta(\x) = \sum_{k=0, \, k \text{ even }}^{l} (-1)^{\frac{k}{2}} \binom{l}{k} \left( \frac{x-x_\E}{h_K} \right)^{l-k} \left( \frac{y-y_\E}{h_K} \right)^{k} \quad \forall l=1,\dots,\p.\\
\end{split}
\end{equation*}
The fact that this is actually a basis for $\mathbb H_\p(\E)$ is
proven, e.g., in \cite[Theorem 5.24]{axler2013harmonic}.

\paragraph*{Basis for $\VDeltaE$ for a given $\E \in \taun$.} For this
local VE space introduced in \eqref{local VE space}, we employ the
canonical basis 
$\left\{ \varphi_{j,r} \right\}_{j=1,\ldots, \NE \atop r=0,\ldots,\p-1}$
defined though \eqref{definition canonical basis}, where we also
recall that $N_K$ denotes the number of edges of $K$.
\medskip

In the following, we derive the matrix representation of the local discrete bilinear form introduced in \eqref{local discrete bilinear form}.
We begin with the computation of the matrix representation of the
projector $\PiE$ acting from $\VE$ to $\mathbb H_p(\E)$ and defined in \eqref{H1 bulk projector}.
To this purpose, given any basis function 
$\varphi_{j,r} \in \VDeltaE$, $j=1,\dots,N_K$, $r=0,\dots,p-1$, 
we expand $\PiE \varphi_{j,r}$ in terms of basis $\{
q_\alpha^\Delta \}_{\alpha=1,\ldots,\npDelta}$
of $\mathbb H_p(\E)$, i.e.,
\begin{equation} \label{Pinabla expansion}
\PiE \varphi_{j,r} = \sum_{\alpha=1}^{\npDelta} s_\alpha^{(j,r)} q_\alpha^\Delta.
\end{equation}

Using \eqref{H1 bulk projector} and testing \eqref{Pinabla expansion}
with functions $q_\beta^\Delta$, $\beta=1,\dots,\npDelta$, we get 
that the coefficients $s_{\alpha}^{(j,r)}$ can be computed by
solving for $\boldsymbol{s}^{(j,r)}:=[s_1^{(j,r)},\ldots,
s_{\npDelta}^{(j,r)}]^T$ the $\npDelta \times \npDelta$ algebraic linear system
\[
\boldsymbol{G}\boldsymbol{s}^{(j,r)}=\boldsymbol{b}^{(j,r)},
\]
where
\[
\boldsymbol{G}=
\begin{bmatrix}
(q_1^\Delta,1)_{0,\partial \E} & (q_2^\Delta,1)_{0,\partial \E} & \cdots & (q_{n_p^\Delta}^\Delta,1)_{0,\partial \E} \\
0 & (\nabla q_2^\Delta,\nabla q_2^\Delta)_{0,\E} & \cdots & (\nabla q_{\npDelta}^\Delta,\nabla q_2^\Delta)_{0,\E} \\
\vdots & \vdots & \ddots & \vdots \\
0 & (\nabla q_{\npDelta}^\Delta,\nabla q_2^\Delta)_{0,\E} & \cdots & (\nabla q_{\npDelta}^\Delta,\nabla q_{\npDelta}^\Delta)_{0,\E}
\end{bmatrix},
\quad
\boldsymbol{b}^{(j,r)}=
\begin{bmatrix}
(\varphi_{j,r},1)_{0,\partial \E} \\
(\nabla \varphi_{j,r},\nabla q_2^\Delta)_{0,\E} \\
\vdots \\
(\nabla \varphi_{j,r},\nabla q_{\npDelta}^\Delta)_{0,\E} 
\end{bmatrix}.
\]
%
Collecting all the $N_K p$ (column) vectors
$\boldsymbol{b}^{(j,r)}$ in a matrix $\boldsymbol{B}\in \mathbb
R^{\npDelta \times N_K p}$, namely, setting
$\boldsymbol{B}:=[\boldsymbol{b}^{(1,1)},\dots,\boldsymbol{b}^{(N_K,p)}]$,
the matrix representation $\boldPiStar$ of the projector
$\PiE$ acting from $\VDeltaE$ to $\mathbb H_p(\E)$ is given by
\begin{align*}
\boldPiStar 
= \boldsymbol{G}^{-1} \boldsymbol{B} \in \mathbb R^{\npDelta \times N_K p}.
\end{align*}
Subsequently, we define
\begin{align*}
\boldsymbol{D}:=\begin{bmatrix}
\dof_{1,1}(q_1^\Delta) & \cdots & \dof_{1,1}(q_{\npDelta}^\Delta) \\
\vdots & \ddots & \vdots \\
\dof_{N_\E, p}(q_1^\Delta) & \cdots & \dof_{N_\E, p}(q_{\npDelta}^\Delta)
\end{bmatrix} \in \mathbb R^{N_\E p \times \npDelta}.
\end{align*}
Let $\boldPi$ 
be the matrix representation of the operator $\PiE$ seen now as a map from $\VDeltaE$ into $\VDeltaE \supseteq \mathbb H_\p(\E)$.
Then, following \cite{hitchhikersguideVEM}, it is possible to show that
\begin{align*}
\boldPi 
= \boldsymbol{D} \boldsymbol{G}^{-1} \boldsymbol{B} \in \mathbb R^{N_K p \times N_K p}.
\end{align*}
Next, denoting by $\widetilde{\boldsymbol{G}} \in \mathbb R^{\npDelta
  \times \npDelta}$ the matrix coinciding with $\boldsymbol{G}$ apart
from the first row which is set to zero,
the matrix representation of the bilinear form in~\eqref{local discrete bilinear form} is
\begin{equation*} 
(
\boldPiStar 
)^T \, \widetilde{\boldsymbol{G}} \,(
\boldPiStar 
) + (\boldsymbol{I} - 
\boldPi 
)^T \, \boldsymbol{S} \, (\boldsymbol{I} - 
\boldPi 
).
\end{equation*}
Here, $\boldsymbol{S}$ denotes the matrix representation of an
explicit stabilization $S^K(\cdot,\cdot)$. For the stabilization
defined in \eqref{explicit stabilization}, we have 
\begin{align*} 
\boldsymbol{S}((k-1)N_\E + r,(l-1)N_\E+s) &= \sum_{i=1}^{N_\E} \frac{\p}{\h_{e_i}} (\Piei \varphi_{l,s}, \Piei \varphi_{k,r})_{0,\e_i} \\ &\quad \forall k,l=1,\dots,N_\E, 
\forall r,s=0,\dots,p-1.
\end{align*}
By expanding $\Piei \varphi_{l,s}$ and $\Piei \varphi_{k,r}$ in the
basis $\left\{ m_\gamma^{e_i} \right\}_{\gamma=0,\ldots,p-1}$ of $\mathbb P_{\p-1}(\e_i)$, i.e.,
\begin{align} \label{Piei}
\Piei \varphi_{l,s} = \sum_{\gamma=0}^{p-1} t_\gamma^{(l,s),e_i} m_\gamma^{e_i}, \quad \Piei \varphi_{k,r} = \sum_{\zeta=0}^{p-1} t_\zeta^{(k,r),{e_i}} m_\zeta^{e_i}, \quad \forall k,l=1,\dots,N_\E,\, \forall r,s=0,\dots,\p-1,
\end{align}
we can write
\begin{equation*} 
\begin{split}
\boldsymbol{S}((k-1)N_\E+ r,(l-1)N_\E+s) 	
&= \sum_{i=1}^{N_\E} \sum_{\gamma=0}^{p-1} \sum_{\zeta=0}^{p-1} t_\gamma^{(l,s),{e_i}} t_\zeta^{(k,r),{e_i}} \frac{\p}{\h_{e_i}} (m_\gamma^{e_i}, m_\zeta^{e_i})_{0,{e_i}}\\
&\forall k,l=1,\dots,N_\E,\, \forall r,s=0,\dots,\p-1.
\end{split}
\end{equation*}
For the basis defined in \eqref{def_malpha_legendre}, using the
orthogonality of the Legendre polynomials \eqref{orthogonality legendre}, this expression can be simplified leading to a diagonal stability matrix $\textbf{S}$:
\begin{equation*} 
\begin{split}
\boldsymbol{S}((k-1)N_\E+ r,(k-1)N_\E+r)&=
\sum_{i=1}^{N_\E} \sum_{\zeta=0}^{p-1}
\frac{\p}{2r+1} (t_\zeta^{(k,r),{e_i}})^2 \\
&\forall k=1,\dots,N_\E,\, \forall r=0,\dots,\p-1.
\end{split}
\end{equation*}
For fixed $i,k \in \{1,\dots,N_\E \}$ and $r \in \{0,\dots,\p-1\}$, the coefficients $t_\zeta^{(k,r),e_i}$ are obtained by testing $\Piei \varphi_{k,r}$, defined in \eqref{Piei},
with $m_{\zeta}^{e_i}$, $\zeta=0,\dots,\p-1$, and by taking into account the definition of $\Piei$ in \eqref{L2 edge projector}, the orthogonality relation \eqref{orthogonality legendre}
and the definition of $\varphi_{k,r}$ in \eqref{definition canonical basis}. 
This gives
\[
t_\zeta^{(k,r),e_i} = \frac{2 \zeta+1}{h_{e_i}} (\varphi_{k,r},m_\zeta^{e_i})_{0,e_i}
= (2 \zeta+1) \delta_{i k} \delta_{r \zeta} \quad \forall \zeta=0,\dots,p-1.
\]

The global system of linear equations corresponding to method
\eqref{VEM} is assembled as in the standard non-conforming FEM.
Finally, one imposes in a non-conforming fashion the Dirichlet boundary datum $\g$ by
\begin{align*}
\int_\e \un \qpmue \, \ds = \int_\e \g \qpmue \, \ds \quad \forall \qpmue \in \mathbb P_{\p-1}(\e),
\end{align*}
where, in practice, $g$ is replaced by $g_p$, see Remark \ref{remark how to deal with Dirichlet boundary conditions}.
\end{appendices}

{\footnotesize
\bibliography{bibliogr}
}
\bibliographystyle{plain}

\end{document}